\documentclass[UTF8,reqno,12pt]{amsart}
\usepackage{amsmath, amsthm, amssymb, amsfonts} % 数学核心宏包合并
\usepackage{newtxtext, newtxmath} % 替代txfonts 的现代字体方案newtxtext, newtxmath
\usepackage{eucal, mathrsfs}      % 数学符号宏包
\usepackage{bm}                   % 数学粗体
\usepackage{mathtools}   % 提供 \DeclarePairedDelimiter 等高级数学命令
\usepackage{mleftright}  % 优化自动括号尺寸
\usepackage{relsize}        %调整prod的大小

\usepackage[scale=0.85,
            left=2.4cm, right=2.4cm,          %改成0.4是为了编译，正常可3.4
            top=3cm, bottom=2.5cm,
            headheight=1.2cm, %showframe,     %showframe 显示边框检验论文有无出界
            footskip=1cm]{geometry}
\usepackage{enumerate}            % 增强列表环境
\usepackage{marginnote}           % 边注功能
\usepackage{ragged2e}             % 段落对齐

\usepackage{graphicx, pgfplots}    % 图形支持
\usepackage{array, booktabs}       % 表格增强
\usepackage{threeparttable}        % 表格脚注
\usepackage{diagbox}               % 斜线表头
\usepackage{caption}               % 标题格式
\usepackage{colortbl}             % 表格颜色
\usepackage{tikz}                  % 矢量绘图
\usepackage{color}
\usepackage[colorlinks,linkcolor=red,pagebackref]{hyperref}
\hypersetup{
    linkcolor=blue,          % color of internal links
    citecolor=red,        % color of links to bibliography
    filecolor=blue,      % color of file links
    urlcolor=cyan
}

\definecolor{MyDarkBlue}{cmyk}{0.8,0.3,0.8,0.4}
\definecolor{yellow}{rgb}{0.99,0.99,0.70}
\definecolor{white}{rgb}{1.0,1.0,1.0}
\definecolor{black}{rgb}{0.00,0.00,0.00}

\usepackage{xcolor}  % 支持更丰富的颜色选项

\newcommand{\red}[1]{\textcolor{\red}{#1}}

\numberwithin{equation}{section}
\def\theequation{\arabic{section}.\arabic{equation}}
\newtheorem{theorem}{Theorem}[section]
\newtheorem{lemma}[theorem]{Lemma}
\newtheorem{remark}[theorem]{Remark}
\newtheorem{definition}[theorem]{Definition}
\newtheorem{proposition}[theorem]{Proposition}

\newtheorem{corollary}[theorem]{Corollary}
\newtheorem{assumption}[theorem]{Assumption}

\newcommand{\be}{
\begin{equation}}
\newcommand{\ee}{\end{equation}
}
\newcommand{\ce}{
\begin{equation*}}
\newcommand{\de}{\end{equation*}
}

\def\e{{\mathrm{e}}}
\def\eps{\varepsilon}

\def\red{\textcolor{red}}

\def\[{{\Big[}}
\def\]{{\Big]}}
\def\<{{\langle}}
\def\>{{\rangle}}
\def\({{\Big(}}
\def\){{\Big)}}

\def\bx{{\mathbf{x}}}

\def\dif{{\mathord{{\rm d}}}}

\def\min{{\mathord{{\rm min}}}}

\def\no{\nonumber}
\def\={&\!\!=\!\!&}
\def\bt{\begin{theorem}}
\def\et{\end{theorem}}
\def\bl{\begin{lemma}}
\def\el{\end{lemma}}
\def\br{\begin{remark}}
\def\er{\end{remark}}

\def\bd{\begin{definition}}
\def\ed{\end{definition}}
\def\bp{\begin{proposition}}
\def\ep{\end{proposition}}
\def\bc{\begin{corollary}}
\def\ec{\end{corollary}}
\def\bx{\begin{Examples}}
\def\ex{\end{Examples}}
\def\cA{{\mathcal A}}

\def\cC{{\mathcal C}}
\def\cD{{\mathcal D}}

\def\cF{{\mathcal F}}

\def\cK{{\mathcal K}}

\def\cQ{{\mathcal Q}}

\def\cW{{\mathcal W}}

\def\mE{{\mathbb E}}

\def\mI{{\mathbb I}}
\def\mJ{{\mathbb J}}

\def\mN{{\mathbb N}}

\def\mR{{\mathbb R}}

\def\mW{{\mathbb W}}

\def\sB{{\mathscr B}}

\def\sF{{\mathscr F}}

\def\sL{{\mathscr L}}

\def\geq{\geqslant}
\def\leq{\leqslant}

\def\R{{\mathbb R}}

\begin{document}

\title[Optimal Rates for Ergodic SDEs Driven by Multiplicative $\alpha$-Stable Processes]{Optimal Rates for Ergodic SDEs Driven by Multiplicative $\alpha$-Stable Processes in Wasserstein-1 distance}

\author[X.~Jin]{Xinghu Jin}
\author[X.~Zhang]{Xiaolong Zhang$^*$}

\address[X.~Zhang]{(Corresponding Author) School of Mathematics and Statistics \\ 
                   Anhui Normal University \\ 
                   Wuhu, Anhui 241002 \\ 
                   P.R. China}
\email{zhangxl.math@ahnu.edu.cn}
\address[X.~Jin]{School of Mathematics \\ 
                  Hefei University of Technology \\ 
                  Hefei, Anhui 230009 \\ 
                  P.R. China}
\email{xinghujin@hfut.edu.cn}

\maketitle

{\small \noindent {\bf Abstract:}
This paper establishes the quantitative stability of invariant measures $\mu_{\alpha}$ for $\mathbb{R}^d$-valued ergodic stochastic differential equations driven by rotationally invariant multiplicative $\alpha$-stable processes with $\alpha\in(1,2]$. Under structural assumptions on the coefficients with a fixed parameter vector $\bm{\theta}$, we derive optimal convergence rates in the Wasserstein-$1$ ($\cW_{1}$) distance between the invariant measures introduced above, namely,

\begin{enumerate}
    \item[(i)] For any interval $[\alpha_0, \vartheta_0] \subset (1,2)$, there exists $C_1 = C(\alpha_0, \vartheta_0,\bm{\theta},d) > 0$ such that
    \begin{align*}
    \cW_{1}(\mu_\alpha, \mu_\vartheta) \leq C_1 |\alpha - \vartheta|, \quad \forall \alpha, \vartheta \in [\alpha_0, \vartheta_0].
    \end{align*}

    \item[(ii)] For any $\alpha_0\in (1,2)$, there exists $C_2 = C(\alpha_0, \bm{\theta}) > 0$ such that
\begin{align*}
    \cW_{1}(\mu_\alpha, \mu_2) \leq C_2\, d(2 - \alpha), \quad \forall \alpha \in [\alpha_0, 2).
    \end{align*}
\end{enumerate}
The optimality of these rates is rigorously verified by explicit calculations for the Ornstein-Uhlenbeck systems in \cite{Deng2023Optimal}. It is worth emphasizing that \cite{Deng2023Optimal} addressed only case (ii) under additive noise, whereas our analysis establishes results for both cases (i) and (ii) under multiplicative $\alpha$-stable noise, employing fundamentally different analytical methods. \\


$\mathbf{2020\ MSC}$: 60B10, 60G51, 60H07

% 60B10 Convergence of probability measures
% 60G51 Processes with independent increments; Levy processes
% 60H07 Stochastic calculus of variations and the Malliavin calculus

\bigskip

\section{Introduction}\label{PM}

We investigate the following two prototypical stochastic differential equations (SDEs):\begin{align}
    &\dif X_t = b(X_t)\dif t + \sigma(X_{t-})\dif L^{\scriptscriptstyle\!(\alpha)}_t, \quad X_0=x, \label{sde:L} \\
    &\dif Y_t = b(Y_t)\dif t + \sigma(Y_t)\dif B_t, \quad\quad\quad\!  Y_0=x, \label{sde:B}
\end{align}
where the coefficient pair $(b,\sigma)$ satisfies $b:\mathbb{R}^d\to\mathbb{R}^d$ and $\sigma:\mathbb{R}^d\to\mathbb{R}^d\otimes\mathbb{R}^d$. The driving noise processes consist of $(L^{\scriptscriptstyle\!(\alpha)}_t)_{t\geq0}$, a rotationally invariant $\alpha$-stable L\'evy process with $\alpha\in(1,2)$ characterized by its Fourier transform
{\fontsize{8}{8}\selectfont
$ \mathbb{E}\,e^{i\langle z, L^{\scriptscriptstyle\!(\alpha)}_t\rangle} = e^{-t|z|^\alpha/2}$
} for any $z\in\mR^d$,
and $(B_t)_{t\geq0}$, a standard Brownian motion. The analysis of SDEs driven by L\'evy processes has become increasingly important across various disciplines (see \cite{Arapostathis2019ergodicity, Chen2021solving, Tankov2003financial}). Under certain regularity and dissipativity conditions, it is well-known that SDEs \eqref{sde:L} and  \eqref{sde:B} admit unique (weak or strong) solutions $(X_t)_{t \geq 0}$ and $(Y_t)_{t \geq 0}$, respectively, and possess unique invariant measures denoted by $\mu_{\alpha}$ and $\mu_2$ (see \cite{Wang2020exponential,Zhang2023ergodicity}). On the other hand, the continuity theorem guarantees that, for any fixed $t>0$ the $\alpha$-stable process $L^{\scriptscriptstyle\!(\alpha)}_t$ converges weakly to the Brownian motion $B_t$ as $\alpha \uparrow 2$ (see \cite[Theorem 3.3.17]{Durrett2019probability}). This naturally raises two fundamental questions regarding SDEs \eqref{sde:L} and \eqref{sde:B}:
\begin{itemize}
  \item[(1)] For fixed $t > 0$, does $X_t$ converge weakly to $Y_t$ as $\alpha \uparrow 2$?
        
  \item[(2)] If the answer to (1) is affirmative, does $\mu_{\alpha}$ converge weakly to $\mu_2$ as $\alpha \uparrow 2$?  
\end{itemize}
In this paper, we provide a complete resolution to these questions, which are fundamentally motivated by at least three interconnected research domains.

Fundamental results of such SDEs, including well-posedness, Markov semigroup regularity, and ergodicity, have been extensively studied. We summarize some key contributions below. \textit{(i) For well-posedness of SDEs}: Applebaum \cite{Applebaum2009Levy} established strong existence and uniqueness for SDEs with jump processes or Brownian motion; 
Stroock \cite{Stroock1975diffusion} pioneered the well-posedness of martingale problems for L\'evy-type generators, and Lepeltier and Marchal \cite{Lepeltier1976problem} further examined its relationship to weak solutions; 
Chen and Wang \cite{Chen2016uniqueness} analyzed weak solutions for $\alpha$-stable processes with drifts in Kato classes, and Zhang \cite{Zhang2013stochastic} proved pathwise uniqueness for symmetric $\alpha$-stable SDEs with bounded drifts in fractional Sobolev spaces. \textit{(ii) For regularity of Markov semigroups}: Wang, Xu and Zhang \cite{Wang2015gradient} obtained the Bismut–Elworthy–Li type derivative formulas and gradient estimates for SDEs driven by multiplicative rotationally invariant $\alpha$-stable processes via finite-jump approximations; Chen, Hao and Zhang \cite{Chen2020Holder} proved gradient estimates for cylindrical $\alpha$-stable processes using Littlewood–Paley theory.
Chen, Xu, Zhang and Zhang \cite{Chen2024wd} obtained $\cC^s$-regularity ($s>\alpha$) for the Markov semigroup of SDEs driven by multiplicative rotationally invariant $\alpha$-stable processes. \textit{(iii) For the exponential ergodicity of SDEs}: there are at least three well-developed
methods for studying the exponential ergodicity of a Markov process defined through an SDE, namely, coupling method (see  \cite{Chen1989coupling,Eberle2016reflection,Liang2020gradient,Liang2020A,Luo2019refined}), functional inequality method (see \cite{Rockner2021weak,Wang2000functional,Wang2005functional,Wang2020exponential}) and Meyn-Tweedie's minorization condition (see  \cite{Arapostathis2019ergodicity,Chen2024wd,Gurvich2014diffusion,Xie2020ergodicity,Zhang2023ergodicity}).

Despite theoretical progress, numerical approximation, such as Euler-Maruyama (EM) schemes, of stochastic systems remains challenging. On the other hand, for the one-dimensional case, \cite{Roberts1996Exponential} established that the classical Langevin Monte Carlo diffusion with potential $U(x)=|x|^\gamma$ is exponentially ergodic if and only if $\gamma \geq 1$. Consequently, it implies that continuous Langevin diffusions are unsuitable for simulating heavy-tailed distributions, necessitating $\alpha$-stable driven systems. Recent developments in numerical error quantification employ two principal methodologies:
\textit{(i) For Lindeberg principle techniques}: Pag\'{e}s and Panloup \cite{Pages2023Unadjusted} applied this principle to analyze Wasserstein-$1$ and total variation distances for ergodic measures of diffusions and their adaptive EM schemes. Chen, Shao and Xu \cite{Chen2023A} extended Lindeberg's principle to Markov processes, establishing a general approximation framework. Subsequently, Chen, Deng, Schilling and Xu \cite{Chen2023approximation} employed this approach for additive $\alpha$-stable SDEs. \textit{(ii) For Stein's method}: Gurvich \cite{Gurvich2014diffusion} established Brownian steady-state approximations for continuous-time Markov chains, with applications to queueing systems; Fang, Shao and Xu \cite{Fang2019multivariate} investigated the Wasserstein-$1$ distance between ergodic measures of diffusions and their EM schemes; Chen, Xu, Zhang and Zhang \cite{Chen2024wd} obtained the $\cW_{\mathbf{d}}$-convergence rate of EM schemes for invariant measures of supercritical stable SDEs.

The study of parametric dependence SDEs constitutes a fundamental problem in stochastic analysis and statistical physics \cite{CLW23,P14,RX21}. Understanding how solution processes respond to parameter variations provides critical foundations for both theoretical investigations and practical applications. Here we only introduce some results for multiscale systems and underdamped Langevin dynamics. R\"{o}ckner and Xie \cite{RX21} established a general diffusion approximation for fully coupled multiscale SDEs, and derived four distinct averaged equations corresponding to different orders of small parameters. Consider an underdamped Langevin process $(X_t, V_t)_{t\geq 0}$ with friction coefficient $\gamma > 0$. In the regime $\gamma \to \infty$, the rescaled process $X_{\gamma t}$ converges to the Smoluchowski dynamics (overdamped Langevin SDE) (see \cite{P14}). Cao, Lu, and Wang \cite{CLW23} established explicit quantitative estimates for the exponential decay rate in $L^2$-distance to the stationary density. For certain convex potentials, they demonstrated that the convergence rate exceeds that of the overdamped Langevin dynamics when the friction coefficient $\gamma$ is tuned. Recently, Deng, Schilling and Xu \cite{Deng2023Optimal} and Deng, Li, Schilling and Xu \cite{Deng2024total} obtained the optimal Wasserstein-1 and total variation convergence rates between ergodic measures for additive  $\alpha$-stable SDEs as $\alpha \uparrow 2$.

In this paper, we extend the analysis of \cite{Deng2023Optimal} to multiplicative noise settings and establish optimal Wasserstein-1 convergence rates between ergodic measures for multiplicative $\alpha$-stable SDEs as $\alpha \uparrow \vartheta$ for $\vartheta \in (1,2]$. A key step is establishing exponential ergodicity for SDEs \eqref{sde:L} and \eqref{sde:B} by the coupling method (see \cite{Liang2020gradient,Wang2020exponential}). Under uniform ellipticity and Lipschitz continuity of coefficients, when the drift $b$ satisfies the following dissipative condition:
\begin{align}\label{DC0}
\langle b(x),x\rangle\leq -c_0|x|^2+c_1,\quad \forall x\in\mathbb{R}^d,
\end{align}
it is well-known that both SDEs \eqref{sde:L} and \eqref{sde:B} are exponentially ergodic (see \cite{Xie2020ergodicity,Zhang2023ergodicity}). We denote by $\mu_\alpha$ and $\mu_2$ the invariant measures of SDEs \eqref{sde:L} and \eqref{sde:B}, respectively. However, in this paper, stronger dissipative assumptions than \eqref{DC0} are required to achieve optimal Wasserstein-1 convergence rates (see Theorem \ref{th:main}). We will elaborate on $\vartheta=2$ and $\vartheta\in(1,2)$ two cases in detail. For $\vartheta=2$, we impose the following dissipativity condition:
\begin{itemize}
\item[$(DC)$] \phantomsection\label{DC} 
Let $\tilde{\sigma}:=(\sigma\sigma^{*}-\sigma_0^2 \mathbb{I})^{1/2}$ with sufficiently small $\sigma_0>0$. There exist $c_0,c_1>0$ such that for all $x,y\in\mathbb{R}^d$,
\begin{eqnarray*} 
2\langle x-y,b(x)-b(y)\rangle + \|\tilde{\sigma}(x)-\tilde{\sigma}(y)\|_{\rm HS}^2 - \frac{|(x-y)^*(\tilde{\sigma}(x)-\tilde{\sigma}(y))|^2}{|x-y|^2} \leq -c_0|x-y|^2+c_1.
\end{eqnarray*}
\end{itemize}
Under \hyperref[DC]{$(DC)$}, the exponential contraction of SDE \eqref{sde:B} in Wasserstein-1 distance depends on $c_0$ and $c_1$, implying the convergence rate dominated by $O(d(2-\alpha))$ (see Theorem \ref{th:main} (1)). For $\vartheta\in(1,2)$, we fix $\alpha_0,\vartheta_0\in(1,2)$ such that $\alpha,\vartheta\in[\alpha_0,\vartheta_0]$ and assume the following weaker dissipativity condition:
\begin{itemize}
\item[$(DC')$]\phantomsection\label{DC'} 
There exist $c_0,c_1>0$ such that for all $x,y\in\mathbb{R}^d$,
\begin{eqnarray*} 
\langle x-y,b(x)-b(y)\rangle \leq -c_0|x-y|^2+c_1.
\end{eqnarray*}
\end{itemize}
Under \hyperref[DC']{$(DC')$}, the exponential contraction of SDE \eqref{sde:L} in Wasserstein-1 distance depends on $c_0, c_1, d, \alpha_0, \vartheta_0$, yielding a convergence rate dominated by $O(|\vartheta-\alpha|)$ (see Theorem \ref{th:main} (2)). 
Although exponential contraction for SDE \eqref{sde:L} can also be established under weaker assumptions by Meyn-Tweedie's method from \cite{Zhang2023ergodicity}, such approach fails to provide explicit parameter dependence. Thus, we use the dissipative condition \hyperref[DC']{$(DC')$} for this case.
Additionally, since we use Malliavin calculus to analyze the short-time regularity of the semigroups associated with SDEs \eqref{sde:L} and \eqref{sde:B} (see Lemma \ref{lemma:g}), we require the following conditions on $b$ and $\sigma$:
\begin{assumption}\label{hyp:A}
The coefficients $b$ and $\sigma$ satisfy the following properties: 
\begin{itemize}

\item[$(A1)$]\phantomsection\label{A1} 
(Uniform Ellipticity) For any $|y|=1$, there exists $c_2\geq 1$ such that
\begin{align*}
c_2^{-1}\leq |\sigma(x)y|^2 \leq c_2, \quad \forall x\in\R^d.
\end{align*}

\item[$(A2)$]\phantomsection\label{A2} 
(Regularity) For any $|y_i|=1$ with $i=1,2,3$, there exists $c_3>0$ such that
\begin{align*}
\big|\mathsmaller{\prod_{i=1}^n}\nabla_{y_i}b(x)\big| + \big\|\mathsmaller{\prod_{i=1}^n}\nabla_{y_i}\sigma(x)\big\| \leq c_3, \quad \forall x\in\R^d,\  n=1,2,3.
\end{align*}

\end{itemize}
\end{assumption}

Let $\mathcal{P}_1(\mathbb{R}^d)$ denote the space of probability measures with finite first moments. The Wasserstein-1 distance between $\nu_1, \nu_2 \in \mathcal{P}_1(\mathbb{R}^d)$ is defined via Kantorovich-Rubinstein duality (see \cite[p.60]{V09}):
\begin{align}\label{de:W1}
\mathcal{W}_{1}(\nu_1,\nu_2) := \inf_{(X,Y)\in\Pi(\nu_1,\nu_2)} \mathbb{E}|X-Y|
= \sup_{g \in \mathrm{Lip}(1)} \left| \int g  d\nu_1 - \int g  d\nu_2 \right|,
\end{align}
where $\Pi(\nu_1,\nu_2)$ is the set of probability distributions on $\R^d\times \R^d$ with marginal measures $\nu_1$ and $\nu_2$, and $\mathrm{Lip}(1)$ denotes the class of Lipschitz functions with unit modulus:
$$ \mathrm{Lip}(1) := \left\{ f \in C(\mathbb{R}^d, \mathbb{R}) : |f(x)-f(y)| \leq |x-y| \text{ for all } x,y \in \mathbb{R}^d \right\}. $$
Under Assumption \ref{hyp:A}, SDEs \eqref{sde:L} and \eqref{sde:B} admit unique strong solutions $(X_t^x)_{t \geq 0}$ and $(Y_t^x)_{t \geq 0}$, respectively. Let $\mathscr{L}_t^{\alpha,x}$ and $\mathscr{L}_t^{x}$ denote the distributions at time $t$ of the solutions to SDEs \eqref{sde:L} and  \eqref{sde:B}. Define the parameter tuple $\bm{\theta} := (c_0, c_1, c_2, c_3)$. We now state our main result, and its proof is deferred to Section \ref{ERMO}.

\begin{theorem}\label{th:main}
Under Assumption \ref{hyp:A}, the following results hold:
\begin{enumerate}
\item If \hyperref[DC]{$(DC)$} holds, then for any $\alpha_0 \in (1,2)$, there exists $C_1 = C(\bm{\theta},\alpha_0) > 0$ such that for all $\alpha \in [\alpha_0,2)$ and $t>0$,
\begin{align*}
\cW_{1}\left(\sL_{t}^{\alpha,x}, \sL_{t}^{x}\right) \leq C_1d(2-\alpha), \quad
\cW_{1}(\mu_\alpha,\mu_2) \leq C_1d(2-\alpha).
\end{align*}

\item If \hyperref[DC']{$(DC')$} holds, then for any $[\alpha_0, \vartheta_0] \subset (1,2)$, there exists $C_2 = C(\bm{\theta},\alpha_0,\vartheta_0,d) > 0$ such that for all $\alpha,\vartheta \in [\alpha_0, \vartheta_0]$ and $t>0$,
\begin{align*}
\cW_{1}\big(\sL_{t}^{\alpha,x}, \sL_{t}^{\vartheta,x}\big) \leq C_2|\vartheta-\alpha|, \quad
\cW_{1}(\mu_\alpha,\mu_{\vartheta}) \leq C_2|\vartheta-\alpha|.
\end{align*}
\end{enumerate}
\end{theorem}

\begin{remark}\label{rem:optimal}
There are two highlights on Theorem \ref{th:main}: (i) The convergence rates are optimal, as demonstrated by the Ornstein-Uhlenbeck case in \cite{Deng2023Optimal}. (ii) A key contribution of this paper is the application of the finite-jump approximation method from \cite{Wang2015gradient} and certain commutator estimate\footnote{The commutator of operators $A$ and $B$ is defined as $[A, B] := AB - BA$.} method to establish the short-time regularity estimate for the semigroup of the SDE \eqref{sde:L} (see Subsection \ref{s:B1}). 
Using these methods, we can also extend the results of \cite{Deng2024total} to obtain total variation convergence rates for ergodic measures of multiplicative $\alpha$-stable and $\vartheta$-stable SDEs when $\alpha \uparrow \vartheta$ with $\vartheta \in (1,2]$. However, we do not pursue this extension in this paper, since it involves repetitive calculations and does not have significant new mathematical challenges.
 \end{remark}

This paper is organized as follows: Section \ref{ERMO} presents the proof of our main result, Theorem \ref{th:main}. Specifically, Subsection \ref{sec:prelim} introduces the key notations used in Section \ref{ERMO} and establishes two fundamental results: 
(i) higher-order gradient estimates for Markov semigroups (Lemma \ref{lemma:g}), 
and (ii) exponential ergodicity for SDEs \eqref{sde:L} and \eqref{sde:B} (Lemma \ref{ergodicx}). 
Building on Lemma \ref{lemma:g}, we obtain estimates for the differences of infinitesimal generators
for SDEs \eqref{sde:L} and \eqref{sde:B} (Proposition \ref{genediff}). 
Using these preparatory results from Subsection \ref{sec:prelim}, 
we prove Theorem \ref{th:main} in Subsection \ref{sub:PMT}. Section \ref{sec:ge} is devoted to proving Lemma \ref{lemma:g} via Malliavin calculus. 
Subsection \ref{s:B1} establishes part (i) of Lemma \ref{lemma:g}, 
which provides gradient estimates for SDE \eqref{sde:L}. 
Subsection \ref{s:B2} addresses part (ii) of Lemma \ref{lemma:g}, 
focusing on gradient estimates for SDE \eqref{sde:B}. Appendixes \ref{supp:sec} and  \ref{sec:secJf} provide  supplementary material containing necessary results and proof details that are not the key part of this paper. Under dissipativity conditions in \hyperref[DC]{$(DC)$} and \hyperref[DC']{$(DC')$}, we use the coupling method to discuss the exponential ergodicities for SDEs  \eqref{sde:B} and \eqref{sde:L} respectively in Appendix \ref{supp:sec}. We use the method of finite-jump approximation developed in \cite{Wang2015gradient} to investigate certain regularities for the semigroup of SDE \eqref{sde:L}, while some approximation lemmas and moment estimates are proved in Appendix  \ref{sec:secJf}. 
Throughout this paper, we use the following conventions:

\textit{Function spaces}: For $m,n\in\mathbb{N}$, let $\cC(\mathbb{R}^d,\mathbb{R}^m)$ be the space of continuous functions $f: \mathbb{R}^d \to \mathbb{R}^m$. We further denote by
\begin{align*}
&\quad\ \cC^n(\mathbb{R}^d,\mathbb{R}^m)  := \{f \in \cC(\mathbb{R}^d,\mathbb{R}^m) : f\text{ is }n\text{-times continuously differentiable}\}, \\
&\cC^\infty_0(\mathbb{R}^d,\mathbb{R}^m)  :=\{f \in \cC(\mathbb{R}^d,\mathbb{R}^m): f \text{ is a smooth function that vanishes at infinity}\}.
\end{align*}
Specifically, when $m=1$, let $\cC^n(\mathbb{R}^d):=\cC^n(\mathbb{R}^d,\mathbb{R})$ and $\cC^\infty_0(\mathbb{R}^d):=\cC^\infty_0(\mathbb{R}^d,\mathbb{R})$.

\textit{Matrix operators}: For matrices $A,B\in\mathbb{R}^d\otimes\mathbb{R}^d$, we denote by
$$ \langle A,B\rangle_{\rm HS} := \sum_{i,j=1}^d A_{ij}B_{ij}, \quad \|A\|_{\rm HS} := \sqrt{\langle A,A\rangle_{\rm HS}} $$
with operator norm
$$ \|A\|_{\rm op} := \sup_{x\in\partial\mathrm{B}}|Ax| = \sup_{x,y\in\partial\mathrm{B}}|\langle A, xy^*\rangle_{\rm HS}|,$$
where $\partial\mathrm{B}:=\{z\in\mR^d\mid |z|=1\}$ and $*$ denotes transpose. It is standard that $\|A\|_{\rm op} \leq \|A\|_{\rm HS} \leq \sqrt{d}\|A\|_{\rm op}$.

\textit{Directional derivatives}: Let $n\in\mN$. For $f\in\cC^n(\mathbb{R}^d,\mathbb{R}^m)$ and $v_i\in\mathbb{R}^d$ with $i=1,\cdots,n$, denote by
\begin{align*}
&\delta_{v_1}f(x)  := f(x+v_1)-f(x),\qquad
\nabla_{v_1}f(x) := \lim_{\epsilon\to 0}\epsilon^{-1}\delta_{\epsilon v_1}f(x)=\langle v_1,\nabla f(x)\rangle,
\end{align*}
where $\langle\cdot,\cdot\rangle$ is the inner product in $\mathbb{R}^d$.
The $n$-th gradient operator is defined as
$$ \nabla^n f(x) \in \underbrace{(\mathbb{R}^m\otimes\mathbb{R}^d)\otimes\cdots\otimes(\mathbb{R}^m\otimes\mathbb{R}^d)}_{n \text{ copies}} $$
has components $D^\gamma f(x) := \partial_{x_1}^{\gamma_1}\cdots\partial_{x_d}^{\gamma_d}f(x)$ for multi-indices $\gamma:=(\gamma_1,\cdots,\gamma_d)\in\mathbb{N}_0^d$ with $\sum_{i=1}^d\gamma_i=n$, where $\mathbb{N}_0 := \mathbb{N}\cup\{0\}$. Then the operator norm of the $n$-th gradient operator is defined as
$$ \|\nabla^n f(x)\|_{\rm op} := \sup_{x_i\in \partial\mathrm{B}}\big|\mathsmaller{\prod_{i=1}^n}\nabla_{x_i}f(x)\big|, \quad \|\nabla^n f\|_\infty := \sup_{x\in\mathbb{R}^d}\|\nabla^n f(x)\|_{\rm op} $$
with $\nabla^0 f := f$ and $\nabla^1 f := \nabla f$. The subscript ``op" may be omitted when unambiguous.

\textit{Other notations}: We denote by $s\wedge t := \min(s,t)$ and $a^{-}:=\max(0,-a)$. Let $C = C(\theta)>1$ denote positive constant depending on $\theta$, possibly varying between instances. We denote by
$$A \lesssim_\theta B\,(\text{or}\ B\gtrsim_{\theta}A)  \iff A \leq C(\theta)B,\qquad A \asymp_\theta B \iff C(\theta)^{-1}B \leq A \leq C(\theta)B,$$
while we do not pay attention to the dependence of the constant $C$ in notations $\lesssim$, $\gtrsim$ and $\asymp$.

\section{Proof of Main Theorem \ref{th:main}} \label{ERMO}

\subsection{Preliminaries}\label{sec:prelim}

Under Assumption \ref{hyp:A}, SDEs \eqref{sde:L} and \eqref{sde:B} with arbitrary initial value $x\in\mR^d$ each admit unique strong solutions $(X^x_t)_{t \ge 0}$ and $(Y^x_t)_{t \ge 0}$, respectively. The associated Markov semigroups are defined as
\begin{align}\label{A4}
P^{\alpha}_t f(x) := \mE f(X^x_t),\qquad
Q_t f(x) := \mE f(Y^x_t),\quad  \forall f\in\cC^\infty_0(\mR^d).
\end{align}
A straightforward application of It\^o's formula reveals that their infinitesimal generators take the form:
\begin{equation*}
\begin{aligned}
&\qquad\quad\mathcal{A}^{\!\alpha} f(x) := \mathcal{L}^\alpha_\sigma f(x) + \langle b(x), \nabla f(x)\rangle, \\
&\mathcal{A}^{\scriptscriptstyle Q} f(x) := \frac{1}{2}\langle\nabla^2 f(x),(\sigma\sigma^*)(x)\rangle_{\mathrm{HS}} + \langle b(x), \nabla f(x)\rangle,
\end{aligned}
\end{equation*}
where the nonlocal operator $\mathcal{L}^\alpha_\sigma$ is defined by: 
\begin{align}\label{op:levy}
\mathcal{L}^\alpha_\sigma f(x) = c_{d,\alpha}\int_{\mathbb{R}^d \setminus \{0\}} \frac{\delta_{\sigma(x)z}f(x) - \nabla_{\sigma(x)z}f(x)}{|z|^{d+\alpha}}\dif z,\ \text{where}\  c_{d,\alpha} := \frac{\alpha 2^{\alpha-2}\Gamma\left(\frac{d+\alpha}{2}\right)}{\pi^{d/2}\Gamma\left(\frac{2-\alpha}{2}\right)}.
\end{align}
Notably, it holds that $\Delta^{\frac\alpha2} f = 2\mathcal{L}^\alpha_{\mathbb{I}} f$, where $\mI$ denotes the identity matrix. The semigroup operators satisfy the Kolmogorov equations in both forward and backward formulations:
\begin{align}\label{A5}
\frac{\dif}{\dif t} P^{\alpha}_t f = P^{\alpha}_t \mathcal{A}^{\!\alpha}f = \mathcal{A}^{\!\alpha} P^{\alpha}_t f,
\quad \frac{\dif}{\dif t} Q_t f &= Q_t \mathcal{A}^{\scriptscriptstyle Q}f = \mathcal{A}^{\scriptscriptstyle Q} Q_t f.
\end{align}

The following gradient estimates for the semigroups associated with SDEs \eqref{sde:L} and \eqref{sde:B} provide essential regularity properties for our main convergence analysis while the proofs are postponed in Section \ref{sec:ge}.
\begin{lemma}[Gradient Estimates]\label{lemma:g}
For any $g\in\mathrm{Lip}(1)$, under Assumption \ref{hyp:A}, the following gradient estimates hold:
\begin{enumerate}
    \item There exists $C_1 = C(c_2,c_3,\alpha_0,\vartheta_0) > 0$ such that for all $t \in (0,1]$ and $\alpha \in [\alpha_0,\vartheta_0]$ with $1<\alpha_0<\vartheta_0<2$,
    \begin{align}
        \|\nabla^n P^\alpha_t g\|_{\infty} \lesssim_{C_1} t^{-\frac{n-1}{\alpha}}, \quad n = 1,2. \label{est:grad123}
    \end{align}

    \item There exists $C_2 = C(c_2,c_3) > 0$ such that for all $t \in (0,1]$,
\begin{align}
        \|\nabla^n Q_t g\|_{\infty} \lesssim_{C_2} t^{-\frac{n-1}{2}}, \quad n = 1,2,3. \label{est:grad12}
    \end{align}
\end{enumerate}
\end{lemma}

The following lemma establishes exponential convergence to equilibrium, with detailed proofs provided in Appendix \ref{supp:sec}.

\begin{lemma}[Exponential Ergodicity]\label{ergodicx}
Let $\alpha\in[\alpha_0,\vartheta_0]$ with $\alpha_0,\vartheta_0 \in (1,2)$. Under Assumption \ref{hyp:A}, the following two exponential ergodicities hold:

(i) When \hyperref[DC']{$(DC')$} holds, SDE \eqref{sde:L} exhibits exponential ergodicity with ergodic measure $\mu_{\alpha}$: There exist constants $C,\lambda>0$ depending on $d, \alpha_0,\vartheta_0$, and $\bm{\theta}$ such that
\begin{align*}
\cW_{1}\left(\sL_{t}^{\alpha,x},\sL_{t}^{\alpha,y}\right) \leq C\e^{-\lambda t}|x-y|,\quad \forall x,y\in\mathbb{R}^d,\ t>0.
\end{align*}

(ii) When \hyperref[DC]{$(DC)$} holds, SDE \eqref{sde:B} satisfies exponential ergodicity with invariant measure $\mu_2$: There exist constants $C_1,\lambda_1>0$ only depending on $\bm{\theta}$ such that
\begin{align*}
\cW_{1}\left(\sL_{t}^{x},\sL_{t}^{y}\right) \leq C_1\e^{-\lambda_1 t}|x-y|,\quad \forall x,y\in\mathbb{R}^d,\ t>0.
\end{align*}
\end{lemma}

By the preceding two lemmas, we derive cartain operator difference bounds in the following proposition. These bounds constitute the foundational components for the proof of Theorem~\ref{th:main}.
\begin{proposition}\label{genediff}
For any $g\in\mathrm{Lip}(1)$, under Assumption \ref{hyp:A}, the following two estimates hold:
\begin{enumerate}
    \item[(i)] For any $\alpha_0 \in (1,2)$, there exists $C_1 = C(c_2,c_3,\alpha_0) > 0$ such that for all $t \in (0,1]$, $x\in\mR^d$ and $\alpha \in [\alpha_0,2)$:
    \begin{align}\label{genediff:1}
    \sup_{g \in \mathrm{Lip}(1)}\left|(\mathcal{A}^{\!\alpha}-\mathcal{A}^{\scriptscriptstyle Q})Q_tg(x)\right|
    \lesssim_{C_1}d\left((3-\alpha)t^{-\frac{1}{2}} - (3-\alpha)^{-1}t^{\frac{1-\alpha}{2}}\right).
    \end{align}

    \item[(ii)] For any $[\alpha_0,\vartheta_0] \subset (1,2)$, there exists $C_2 = C(c_2,c_3,\alpha_0,\vartheta_0) > 0$ such that for all $t \in (0,1]$, $x\in\mR^d$  and $\alpha, \vartheta \in [\alpha_0,\vartheta_0]$:
\begin{align}\label{genediff:2}
    \sup_{g \in \mathrm{Lip}(1)}\left|(\mathcal{A}^{\!\alpha}-\mathcal{A}^\vartheta)P^\alpha_tg(x)\right|
    \lesssim_{C_2}d|\vartheta-\alpha|t^{-\frac{1}{\alpha}}.
    \end{align}
\end{enumerate}
\end{proposition}
\begin{proof}[Proof of Proposition \ref{genediff}]
\textbf{Part (i):} Let $t\in(0,1]$ and $x\in\mR^d$ be fixed. Denoting by $f := Q_tg$, we have
\begin{align}\label{A1}
\|\nabla^n f\|_\infty\overset{\eqref{est:grad12}}\lesssim_{c_2,c_3} t^{-\frac{n-1}{2}},\quad \forall n = 1,2,3.
\end{align}
By \eqref{op:levy}, denoting by $A:=\sigma(x)$, we consider the decomposition:
\begin{equation*}
    (\mathcal{A}^{\!\alpha}-\mathcal{A}^{\scriptscriptstyle Q})f(x) = J_1 + J_2,
\end{equation*}
where
\begin{align*}
   &\qquad\qquad\qquad J_1 := c_{d,\alpha}\int_{|z|>1}\frac{f(x+A z)-f(x)}{|z|^{d+\alpha}}dz, \\
   & J_2 := c_{d,\alpha}\int_{0<|z|\leq1}\frac{\delta_{A z}f(x)-\nabla_{A z}f(x)}{|z|^{d+\alpha}}dz - \frac{1}{2}\langle\nabla^2f(x),AA^{*}\rangle_{\mathrm{HS}}.
\end{align*}
\textbf{Estimation of $J_1$:} By Assumption \ref{hyp:A}, using Taylor's expansion with integral remainder, we have
\begin{align*}
    \ |J_1| &\leq c_{d,\alpha}\int_{|z|>1}\int^1_0|\left\langle \nabla f(x+rA z),A z\right\rangle|\dif r\frac{\dif z}{|z|^{d+\alpha}}\\
&\lesssim_{c_2} c_{d,\alpha}\|\nabla f\|_{\infty}\int_{|z|>1}|z|^{1-d-\alpha}\dif z
\lesssim_{c_2}\|\nabla f\|_{\infty}\frac{c_{d,\alpha}S_d}{\alpha-1},
\end{align*}
where $S_d:=\frac{2\pi^{\frac{d}{2}}}{\Gamma(\frac{d}{2})}$ is the surface area of the unit ball in $\mR^d$.
For any $\alpha_0\in(1,2)$ and $\alpha \in [\alpha_0,2)$, it holds that
\begin{align}
\frac{c_{d,\alpha} S_d}{\alpha-1}
\overset{\eqref{op:levy}}\leq \frac{\alpha2^{\alpha-2}\Gamma
(\frac{d+\alpha}{2})}{\Gamma
(2-\frac\alpha2)
\Gamma(\frac{d}{2})}
\cdot\frac{2-\alpha}{\alpha_0-1}
\lesssim \frac{\Gamma
(\frac{d+2}{2})}{
\Gamma(\frac{d}{2})}
\cdot\frac{2-\alpha}{\alpha_0-1}\lesssim_{\alpha_0}d (2-\alpha).
 \label{CS12}
\end{align}
Then for any $\alpha_0\in(1,2)$ and $\alpha \in [\alpha_0,2)$, we have
\begin{align*}
    |J_1| \lesssim_{c_2,\alpha_0}d (2-\alpha)\|\nabla f\|_{\infty}\overset{\eqref{A1}}\lesssim_{c_2,c_3,\alpha_0}d (2-\alpha).
\end{align*}
\textbf{Estimation of $J_2$:}
Applying Taylor's expansion with integral remainder, we have
\begin{align*}
    J_2 &= c_{d,\alpha}\int_{0<|z|\leq1}\int_0^1\frac{\langle\nabla^2f(x+r\sigma(x) z),(A z)(A z)^*\rangle_{\mathrm{HS}}(1-r)}{|z|^{d+\alpha}}drdz
 - \frac{1}{2}\langle\nabla^2f(x),AA^*\rangle_{\mathrm{HS}} =J_{21} + J_{22},
\end{align*}
where
\begin{align*}
&\quad\ J_{21}
:=\frac{c_{d,\alpha}}{2}\int_{0<|z|\leq 1}
\frac{\langle\nabla^2f(x),
(A z)(A z)^{*}\rangle_{\rm HS}}{|z|^{d+\alpha}}\dif z
-\frac12\,\langle
\nabla^2f(x),AA^{*}
\rangle_{\rm HS},\\
& J_{22}:= c_{d,\alpha}\int_{0<|z|\leq 1}\int_0^1
\frac{\langle\nabla^2f(x+rA z) -\nabla^2f(x),
(A z)(A z)^{*}\rangle_{\rm HS}(1-r)}{|z|^{d+\alpha}}\dif r\dif z.
\end{align*}
For $J_{21}$, by the symmetric property, we note that for any $1\leq i,j\leq d$, 
$$
\int_{0<|z|\leq 1} \frac{z_i z_j}{|z|^{d+\alpha}}\dif z =\frac{\delta_{ij}}{d} \int_{0<|z|\leq 1} \frac{|z|^2}{|z|^{d+\alpha}}\dif z
=\frac{\delta_{ij}S_d}{d(2-\alpha)},
	$$
where $z=(z_1,\cdots ,z_d)$, $\delta_{ij}$ is the Kronecker delta satisfying that the function is $1$ if $i=j$, and $0$ otherwise. Then we have
\begin{align*}
J_{21}
=&\frac{c_{d,\alpha}}{2}\int_{0<|z|\leq1}
\frac{\langle\nabla^2f(x),
A (zz^{*})A \rangle_{\rm HS}}{|z|^{d+\alpha}}\dif z
-\frac12 \langle
\nabla^2f(x),AA^{*}
\rangle_{\rm HS}\\
=&\frac{c_{d,\alpha}S_d}{2d(2-\alpha)}\langle
\nabla^2f(x),AA^{*}
\rangle_{\rm HS}-\frac12\langle
\nabla^2f(x),AA^{*}
\rangle_{\rm HS}\\
=&\left(\frac{c_{d,\alpha}S_d}{2d(2-\alpha)}
-\frac{1}{2}\right)\langle
\nabla^2f(x),AA^{*}
\rangle_{\rm HS}.
\end{align*}
On the other hand, by \cite[Lemma 3.3]{Deng2023Optimal}, we have
\begin{align}\label{A3}
        \left|
        \frac{c_{d,\alpha} S_d}{d(2-\alpha)}
        -1\right|
        \lesssim (2-\alpha)\log(1+d),\quad \forall \alpha\in(0,2).
\end{align}
By Assumption \ref{hyp:A}, it is clear that
\begin{align*}
|J_{21}| \lesssim_{\bm{\theta}}(2-\alpha)\log(1+d)\|\nabla^2 f\|_{\infty}
\overset{\eqref{A1}}\lesssim_{c_2,c_3}(2-\alpha)\log(1+d) t^{-\frac{1}{2}}.
\end{align*}
For $J_{22}$, as the procedure of $J_{21}$, we note that for any $r\in[0,1]$, \begin{align*}
\|\nabla^2f(x+rA z)-\nabla^2f(x)\|_{\rm op}
\lesssim_{c_2} (\|\nabla^3f\|_{\infty} |z|)\wedge \|\nabla^2f\|_{\infty}
\overset{\eqref{A1}}\lesssim_{c_2,c_3} ( t^{-1}|z|)\wedge t^{-\frac{1}{2}}.
\end{align*}
Then we get
\begin{align*}
|J_{22}|&\lesssim_{c_2,c_3}  c_{d,\alpha}\int_{0<|z|\leq 1}
\frac{( t^{-1}|z|)\wedge t^{-\frac{1}{2}}}{|z|^{d+\alpha-2}}\dif z
\lesssim_{c_2,c_3}  c_{d,\alpha}S_d t^{-1}\int_0^1
(r\wedge t^{\frac{1}{2}})r^{1-\alpha}\dif r\\
&\lesssim_{c_2,c_3} \frac{c_{d,\alpha}S_d}{(2-\alpha)}
\left(t^{-\frac{1}{2}}-(3-\alpha)^{-1}
t^{\frac{1-\alpha}{2}}\right)
\overset{\eqref{op:levy}}\lesssim_{c_2,c_3} d\left(t^{-\frac{1}{2}}-(3-\alpha)^{-1}
 t^{\frac{1-\alpha}{2}}\right).
\end{align*}
We get that
\begin{align*}
|J_{2}| \lesssim_{c_2,c_3}d\left((3-\alpha)t^{-\frac{1}{2}}-(3-\alpha)^{-1}
 t^{\frac{1-\alpha}{2}}\right).
\end{align*}
Combining the above estimates, we get \eqref{genediff:1}.

\textbf{Part (ii):}  We fix $t\in(0,1]$, $\alpha,\vartheta\in[\alpha_0,\vartheta_0]$, without loss of generality, assume that $\alpha<\vartheta$ and recall the notations of \textbf{Part (i)}. Denoting by $h:=P^\alpha_tg$, by Lemma \ref{lemma:g}, there exists $C= C(c_2,c_3,\alpha_0,\vartheta_0) > 0$ such that:
\begin{align}\label{A2}
\|\nabla^n h\|_\infty\overset{\eqref{est:grad123}}\lesssim_{C} t^{-\frac{n-1}{\alpha}},\quad \forall n = 1,2.
\end{align}
By \eqref{op:levy}, it follows that
$$
(\cA^\alpha-\cA^\vartheta)h(x)
=\mJ_1+\mJ_2,
$$
where
\begin{align*}
&\mJ_1:=c_{d,\alpha}\int_{\mR^d \setminus \{0\} }
\left(\delta_{A z}h(x)-\nabla_{A z} h(x)\right)\left(\frac{1}{|z|^{d+\alpha}}
-\frac{1}{|z|^{d+\vartheta}}\right)\dif z,\\
&\qquad\  \mJ_2:=(c_{d,\alpha}-c_{d,\vartheta})\int_{\mR^d \setminus \{0\} }
\frac{\delta_{A z}h(x)-\nabla_{A z} h(x)}{|z|^{d+\vartheta}}\dif z.
\end{align*}
\textbf{Estimation of $\mJ_1$:} By Assumption \ref{hyp:A}, we have
\begin{align*}
|\mJ_1| &\lesssim_{c_3} c_{d,\alpha}\left(\|\nabla^2h\|_{\infty}
\int_{0<|z| \leq 1} \frac{|z|^2(1-|z|^{\vartheta-\alpha})}
{|z|^{d+\vartheta}}\dif z
+\|\nabla h\|_{\infty}\int_{|z|>1 } \frac{|z|(|z|^{\vartheta-\alpha}-1)}
{|z|^{d+\vartheta}}\dif z\right)\\
&\lesssim_{c_3} c_{d,\alpha}S_d\left(\|\nabla^2h\|_{\infty}
\int_{0}^1 (1-r^{\vartheta-\alpha})
r^{1-\vartheta}\dif r
+\|\nabla h\|_{\infty}\int_{1}^\infty (r^{\vartheta-\alpha}-1)r^{-\vartheta}
\dif r\right)\\
&\lesssim_{c_3} (\vartheta-\alpha)\left(\|\nabla^2h\|_{\infty}
\frac{c_{d,\alpha}S_d}{(2-\vartheta)(2-\alpha)}
+\|\nabla h\|_{\infty} \frac{c_{d,\alpha}S_d}{(\vartheta-1)
(\alpha-1)}\right).
\end{align*}
Since $\alpha,\vartheta\in[\alpha_0,\vartheta_0]$, by \eqref{CS12}, then we have there exists $C= C(c_2,c_3,\alpha_0,\vartheta_0) > 0$ such that
\begin{align*}
|\mJ_1| &\lesssim_{c_2} (\vartheta-\alpha)\left(\|\nabla^2h\|_{\infty}
\frac{c_{d,\alpha}S_d}{(2-\vartheta)(2-\alpha)}
+\|\nabla h\|_{\infty} \frac{c_{d,\alpha}S_d}{(\vartheta-1)
(\alpha-1)}\right)
\overset{\eqref{A2}}\lesssim_{C} d(\vartheta-\alpha)t^{-\frac{1}{\alpha}}.
\end{align*}
\textbf{Estimation of $\mJ_2$:}  We similarly have
\begin{align*}
|\mJ_2|&\lesssim_{c_2} |c_{d,\alpha}-c_{d,\vartheta}|
\left(\|\nabla^2h\|_{\infty}
\int_{0<|z|\leq 1} \frac{|z|^2}{|z|^{d+\vartheta}}\dif z
+\|\nabla h\|_{\infty}\int_{|z|>1}
\frac{|z|}{|z|^{d+\vartheta}}\dif z\right)\\
&\lesssim_{c_2} |c_{d,\alpha}-c_{d,\vartheta}|S_d
\left(\|\nabla^2h\|_{\infty}\int_{0}^1 r^{1-\vartheta} \dif r
+\|\nabla h\|_{\infty}\int_{1}^\infty r^{-\vartheta}\dif r\right)\\
&\lesssim_{c_2} |c_{d,\alpha}-c_{d,\vartheta}|S_d\left((2-\vartheta)^{-1}
\|\nabla^2h\|_{\infty}+(\vartheta-1)^{-1}\|\nabla h\|_{\infty}\right).
\end{align*}
By Euler's reflection formula of $\Gamma(x)$, dnoting by
$G(x):=x4^{x}\Gamma(d/2+x)\Gamma(x)\sin(\pi x),$
we rewrite
   \begin{align*}
c_{d,\alpha} S_d
=\frac{\alpha2^{\alpha-1}\Gamma
(\frac{d+\alpha}{2})}
{\Gamma(1-\frac\alpha2)
\Gamma(\frac{d}{2})}
=\frac{\alpha2^{\alpha}\Gamma
(\frac{d+\alpha}{2})
\Gamma(\frac\alpha2
)\sin(\frac{\pi\alpha}{2}
)}{2\pi\Gamma
(\frac{d}{2})}=\frac{G(\frac{\alpha}{2})}{\pi\Gamma
(\frac{d}{2})}.
    \end{align*}
Since for any $x\in(\frac{1}{2},1)$, it holds that
$$
|G'(x)| = G(x) \big| x^{-1} + \ln 4 + \psi(d/2+x) + \psi(x) + \pi \cot(\pi x) \big|\lesssim \Gamma(d/2+1),
$$
then we have
$$
|G(x)-G(y)|\lesssim \Gamma(d/2+1) |x-y|,\quad \forall x,y\in (1/2,1),
$$
where $\psi(z):= \frac{\Gamma'(z)}{\Gamma(z)}$ is the digamma function.
Then we have there exists $C= C(c_2,c_3,\alpha_0,\vartheta_0) > 0$ such that
\begin{align*}
|\mJ_2|&\lesssim_{c_2}\Gamma(d/2)^{-1} |G(\alpha/2)-G(\vartheta/2)|\left((2-\vartheta)^{-1}
\|\nabla^2h\|_{\infty}+\|\nabla h\|_{\infty}\right)\\
&\overset{\eqref{A2}}\lesssim_{\!\!\!\!C}d(\vartheta-\alpha)\left((2-\vartheta)^{-1}t^{-\frac{1}{\alpha}}+1\right)\lesssim_{C}d(\vartheta-\alpha)t^{-\frac{1}{\alpha}}.
\end{align*}
Combing the above calculations, we get \eqref{genediff:2} and complete the proof.
\end{proof}

\subsection{Proof of Theorem \ref{th:main}}\label{sub:PMT}
Now, we present the proof of Theorem \ref{th:main}.

\begin{proof}[Proof of Theorem \ref{th:main}]
\textbf{Part (i):} We commence by analyzing the Wasserstein-1 distance between the measures $\sL_{t}^{\alpha,x}$ and $\sL_{t}^{x}$. Through the dual formulation in \eqref{de:W1} and application of Duhamel's principle (see \cite{Pages2023Unadjusted}), we establish
\begin{align}\label{dis11}
\cW_{1}\left(\sL_{t}^{\alpha,x}, \sL_{t}^{x}\right) \overset{\eqref{A4}}= \sup_{g\in\mathrm{Lip}(1)} \left[P^\alpha_tg(x)-Q_tg(x)\right] \overset{\eqref{A5}}= \sup_{g\in\mathrm{Lip}(1)}\int_0^t P^\alpha_{t-s} (\mathcal{A}^{\!\alpha}-\mathcal{A}^{\scriptscriptstyle Q})Q_s g(x)\dif s.
\end{align}
\textbf{Case 1 ($t \leq 1$):} By Proposition \ref{genediff} (i), we yield the existence of a positive constant $C = C(c_2,c_3,\alpha_0)$ satisfying:
\begin{align}
\left|\int_0^{t}P^\alpha_{t-s}(\mathcal{A}^{\!\alpha}-\mathcal{A}^{\scriptscriptstyle Q})Q_sg(x)\dif s\right| &\lesssim_{C}d\int_0^1\left((3-\alpha)s^{-\frac{1}{2}}-(3-\alpha)^{-1}s^{\frac{1-\alpha}{2}} \right)\dif s \nonumber \\
&\lesssim_{C} d((3-\alpha)^3-1)\lesssim_{C} d(2-\alpha). \label{PQ}
\end{align}
\textbf{Case 2 ($t > 1$):} We partition the temporal domain into near-field and far-field components:
\begin{align*}
\int_0^t P^\alpha_{t-s} (\mathcal{A}^{\!\alpha}-\mathcal{A}^{\scriptscriptstyle Q})Q_s g(x)\dif s=\left(\int_0^1+\int_1^t\right) P^\alpha_{t-s} (\mathcal{A}^{\!\alpha}-\mathcal{A}^{\scriptscriptstyle Q})Q_s g(x)\dif s=:I_1+I_2.
\end{align*}
For $I_1$, by \eqref{PQ}, there exists $C= C(c_2,c_3,\alpha_0) > 0$ such that
\begin{align}\label{A6}
|I_1|\lesssim_{C}d\int_0^1\left((3-\alpha)s^{-\frac{1}{2}}-(3-\alpha)^{-1}s^{\frac{1-\alpha}{2}} \right)\dif s\lesssim_{C} d(2-\alpha).
\end{align}
For $I_2$, employing the semigroup property of $(Q_t)_{t\geq 0}$, we restructure:
\begin{align*}
I_2=\int_1^t P^\alpha_{t-s} (\mathcal{A}^{\!\alpha}-\mathcal{A}^{\scriptscriptstyle Q})Q_1Q_{s-1} g(x)\dif s.
\end{align*}
By Lemma \ref{ergodicx} (ii), we have the exponential contraction with $\bm{\theta}$-dependent constants $C_1,\lambda_1>0$, that is,
\begin{align*}
|Q_{s-1} g(x)-Q_{s-1} g(y)|\leq \cW_{1}\left(\sL_{s-1}^{x},\sL_{s-1}^{y}\right) \leq C_1\e^{-\lambda_1 (s-1)}|x-y|, \quad \forall s>1,\ x,y\in\mathbb{R}^{d},
\end{align*}
which implies $F_s:=C^{-1}_1\e^{\lambda_1 (s-1)}Q_{s-1} g\in\mathrm{Lip}(1)$. By Proposition \ref{genediff} (i), there is $C= C(\bm{\theta},\alpha_0) > 0$ such that
\begin{align}\label{A7}
&\quad |I_2|\leq \int_1^t C_1\e^{-\lambda_1 (s-1)} |P^\alpha_{t-s} (\mathcal{A}^{\!\alpha}-\mathcal{A}^{\scriptscriptstyle Q})Q_1F_s(x)|\dif s\no\\
&\lesssim_{C} d\left((3-\alpha)-(3-\alpha)^{-1}\right)\int_1^t C_1\e^{-\lambda_1 (s-1)} \dif s \lesssim_{C}d(2-\alpha).
\end{align}
\textbf{Uniform Estimation:} By \eqref{A6} and \eqref{A7}, we obtain $C= C(\bm{\theta},\alpha_0) > 0$ such that for all $t>0$ and $x\in\mR^d$
\begin{align}\label{A8}
\cW_{1}\left(\sL_{t}^{\alpha,x}, \sL_{t}^{x}\right) \leq C d(2-\alpha).
\end{align}
Applying the triangle inequality for Wasserstein metrics, we have
\begin{align*}
\cW_{1}(\mu_\alpha,\mu_2)&\leq
\cW_{1}(\mu_\alpha,\sL_{t}^{\alpha,x})+ \cW_{1}(\sL_{t}^{\alpha,x},\sL_{t}^{x})+ \cW_{1}(\sL_{t}^{x},\mu_2) \\
& \overset{\eqref{A8}}\leq C d(2-\alpha)+\cW_{1}(\mu_\alpha,\sL_{t}^{\alpha,x})+ \cW_{1}(\sL_{t}^{x},\mu_2).
\end{align*}
Taking $t\to\infty$ via Lemma \ref{ergodicx} (ii), we have
\begin{align*}
\cW_{1}(\mu_\alpha,\mu_2)\lesssim_C d(2-\alpha).
\end{align*}

\textbf{Part (ii):} Let $\alpha,\vartheta\in[\alpha_0,\vartheta_0]\subset (1,2)$. The comparative analysis of $\sL_{t}^{\alpha,x}$ and $\sL_{t}^{\vartheta,x}$ follows analogous methodology. Through parallel developments to \textbf{Part (i)}, we derive
\begin{align*}
\cW_{1}(\sL_{t}^{\alpha,x},\sL_{t}^{\vartheta,x}) &= \sup_{g\in\mathrm{Lip}(1)}\int_0^t P^\vartheta_{t-s} (\mathcal{A}^{\!\alpha}-\mathcal{A}^\vartheta)P^\alpha_s g(x)\dif s.
\end{align*}
By Proposition \ref{genediff} (ii) and ergodic theory via Lemma \ref{ergodicx} (ii), we obtain a constant $C_2= C(\bm{\theta},d,\alpha_0,\vartheta_0) > 0$ such that for all $t>0$ and $x\in\mR^d$,
$$ \cW_{1}(\sL_{t}^{\alpha,x},\sL_{t}^{\vartheta,x}) \lesssim_{C_2}|\vartheta-\alpha|,$$
and further get
$$ \cW_{1}(\mu_\alpha,\mu_\vartheta) \lesssim_{C_2}|\vartheta-\alpha|. $$
Technical details are omitted for concision.
\end{proof}

\section{Gradient Estimates: The Proof of Lemma \ref{lemma:g}} \label{sec:ge}
This section is organized into two subsections: In Subsection \ref{s:B1}, we establish the proof of part (i) in Lemma \ref{lemma:g}. Subsequently, in Subsection \ref{s:B2}, we present the proof of part (ii) in Lemma \ref{lemma:g}. For clarity and completeness of exposition, we begin by recalling essential concepts and properties from Malliavin calculus (see \cite{Nualart2006the, Nualart2018introduction}). Consider the classical Wiener space $(\mathbb{W},\mathbb{H},\mu^{\mathbb{W}})$, where:
\begin{itemize}
    \item $\mathbb{W}$ denotes the space of continuous functions $\omega:[0,\infty)\to\mathbb{R}^d$ with $\omega(0)=0$;
    \item $\mathbb{H}$ represents the Cameron-Martin space consisting of absolutely continuous functions $h\in\mathbb{W}$ with square-integrable generalized derivatives $\dot{h}$, that is,
    $
        \int_0^\infty |\dot{h}(t)|^2 dt < \infty;
    $
    \item $\mu^{\mathbb{W}}$ is the Wiener measure under which the coordinate process $W_t(\omega) := \omega(t)$ is a $d$-dimensional standard Brownian motion.
\end{itemize}
Let $\mathcal{S}$ denote the class of smooth cylindrical random variables of the form
\begin{equation*}
    F = \phi(W(h_1),\ldots,W(h_d)),
\end{equation*}
where $\phi\in\mathcal{C}_0^\infty(\mathbb{R}^d)$ with polynomially bounded partial derivatives, and $h_i\in\mathbb{H}$ with Wiener integrals
\begin{equation*}
    W(h_i) = \int_0^\infty \dot{h}_i(t) dW_t, \quad i=1,\ldots,d.
\end{equation*}
The Malliavin derivative operator $D:\mathcal{S}\to L^2(\mathbb{W};\mathbb{H})$ is defined as the $\mathbb{H}$-valued random variable
\begin{equation*}
    DF = \sum_{i=1}^d \partial_i\phi(W(h_1),\ldots,W(h_d)) \otimes h_i,
\end{equation*}
where $\otimes$ denotes the tensor product in $\mathbb{H}$. Let $\delta$ be the divergence operator adjoint to $D$. By  \cite[pp. 54-55]{Nualart2018introduction}, for any $F, G\in\mathcal{S}$ and $U\in\mathbb{H}$, denoting by
$D_UF:=\langle DF, U \rangle_{\mathbb{H}}$, we have
\begin{align}\label{mal:div}
D_U(FG)=FD_UG+GD_UF,\quad     \mathbb{E}[D_UF] = \mathbb{E}[F\delta(U)].
\end{align}
Here and below, let $\cD(D)$ and $\cD(\delta)$ be the domain of $D$ and $\delta$ respectively.

\subsection{Proof of Lemma \ref{lemma:g} (i)}\label{s:B1}
The $d$-dimensional rotationally invariant $\alpha$-stable process $(L^{\scriptscriptstyle\!(\alpha)}_t)_{t\geq 0}$ admits a canonical representation through Brownian subordination \cite[Theorem 30.1]{Sato1999Levy}. Specifically, there exists independent $d$-dimensional Brownian motion $(W_t)_{t\geq 0}$ and
{\fontsize{9}{8}\selectfont  %\fontsize{8}{9}\selectfont %字号8pt，行距9pt
$\frac{\alpha}{2}$
}-stable subordinator $(S_t)_{t\geq 0}$ satisfying $L^{\scriptscriptstyle\!(\alpha)}_t = W_{S_t}$, where the Laplace transform of $S_t$ is given by
\begin{align}
    \mathbb{E}(\e^{-uS_{t}}) = \e^{-2^{\frac{\alpha-2}{2}}u^{\frac{\alpha}{2}}t}, \quad u > 0. \label{sub:lap}
\end{align}
This subordination structure permits reformulating the original SDE \eqref{sde:L} as
\begin{align}\label{sde:L1}
    \dif X^x_t = b(X^x_t)\dif t + \sigma(X^x_{t-})\dif W_{S_t}, \quad X_0 = x.
\end{align}
Under Assumption \ref{hyp:A}, by fixing any path $\ell_\cdot$ of $S_\cdot$, we analyze the auxiliary SDE
\begin{align}
    \dif X^{\ell,x}_t = b(X^{\ell,x}_t)\dif t + \sigma(X^{\ell,x}_{t-})\dif W_{\ell_t}, \quad X^{\ell,x}_0 = x. \label{sde:fix_xp}
\end{align}
{\it All discussions in this subsection revolve around SDEs \eqref{sde:L1} and  \eqref{sde:fix_xp}}. Additionally, our exposition uses the following two key semigroups: the stable semigroup $P^{\alpha}_t$ and its path-fixed counterpart $P^\ell_t$ are defined for any test function $f \in \cC^\infty_0(\mR^d)$
\begin{align*}
 P^\alpha_t f(x) := \mE \left[f(X^x_t)\right], \quad P^\ell_t f(x) := \mathbb{E}\big[f(X^{\ell,x}_t)\big].
\end{align*}

\subsubsection{The second-order directional derivative formula for $ P^\ell_t $}
For any fixed path $\ell_\cdot$ of $S_\cdot$, we define
\begin{align}\label{de:tl}
\tau_\ell:=\inf\{t\geq0 \mid \ell_t>1\},\quad
\bm{\tau}^{\ell}_t:=t\wedge\tau_\ell.
\end{align}
For any c\`adl\`ag function $\xi_t:[0,\infty)\to\mR^d$, we denote by
$$
\Delta\xi_t:=\xi_t-\xi_{t-}.
$$
Using \cite[p. 310, Theorem 40]{Protter2004stochastic} and Assumption \ref{hyp:A} \hyperref[A2]{$(A2)$}, it is clear that SDE \eqref{sde:fix_xp} forms a $\cC^2$-stochastic flow.
Then for any $v,v_1\in\mR^{d}$ and fixed path $\ell_\cdot$ of $S_\cdot$, we define
\begin{equation}
\begin{split}
    \mathbf{L}^{v}_{\ell,t}
:=&\int^{\bm{\tau}^{\ell}_t}_0
\left\langle\sigma^{-1}
(X^{\ell,x}_{s-})\nabla_vX^{\ell,x}_{s-},\dif W_{\ell_s}\right\rangle
-\int^{\bm{\tau}^{\ell}_t}_0
\left\langle\sigma^{-1}(X^{\ell,x}_{s-}),
\nabla_{\nabla_vX^{\ell,x}_{s-}}
\sigma(X^{\ell,x}_{s-})
\right\rangle_{\small{\mathrm{HS}}}\dif\ell_s\\
&+\sum_{s\in(0,\bm{\tau}^{\ell}_t]}
\left\langle\sigma^{-1}(X^{\ell,x}_{s-})
\nabla_{\nabla_vX^{\ell,x}_{s-}}
\sigma(X^{\ell,x}_{s-})\Delta W_{\ell_{s}},\Delta W_{\ell_{s}}\right\rangle
\end{split}\label{mal:dua:rew}
\end{equation}
and consider the following SDE
\begin{align}
&\mathbf{Z}^{v,v_1}_{\ell,t}=\int^t_0
\nabla_{\mathbf{Z}^{v,v_1}_{\ell,s}}
b(X^{\ell,x}_{s})\dif s
+\int^t_0\nabla_{\mathbf{Z}^{v,v_1}_{\ell,s-}}
\sigma(X^{\ell,x}_{s-})\dif W_{\ell_s}+\mathcal{Z}^{v,v_1}_{\ell,t},
\label{fix:infi:lie}
\end{align}
where
\begin{align}\label{Ztvv1}
\mathcal{Z}^{v,v_1}_{\ell,t}
:=&\sum_{s\in(0,\bm{\tau}
^{\ell}_t]}\Big[\nabla_{v_1} \nabla_{v}X^{\ell,x}_{s-}-\nabla_{ \nabla_{v_1} X^{\ell,x}_{s-}}
\sigma(X^{\ell,x}_{s-})\sigma^{-1}(X^{\ell,x}_{s-})
\left(\nabla_vX^{\ell,x}_{s-}+
 \nabla_{\nabla_vX^{\ell,x}_{s-}}
 \sigma(X^{\ell,x}_{s-})\Delta W_{\ell_{s}}\right)\Big]\Delta\ell_{s}.
\end{align}
It is standard that Assumption \ref{hyp:A}  \hyperref[A2]{$(A2)$} implies SDE \eqref{fix:infi:lie} has a unique strong solution (see  \cite[p. 249, Theorem 6]{Protter2004stochastic}).

The second-order directional derivative formula for $ P^\ell_t $ is established as follows.
\begin{proposition}\label{formermain}
Under Assumption \ref{hyp:A}, it holds that for any $ f \in \cC^\infty_0(\R^d)$ and $x,v,v_1\in\mR^d$,
\begin{align}
\nabla_{v_1}\nabla_{v}P^\ell_tf(x)
 = \ell_{\bm{\tau}^{\ell}_t}^{-1}\mathbb{E}\left[ \nabla f(X^{\ell,x}_t) \nabla_{v_1}X^{\ell,x}_t \mathbf{L}^{v}_{\ell,t} \right] + \ell_{\bm{\tau}^{\ell}_t}^{-1}\mathbb{E}\left[ \nabla f(X^{\ell,x}_t)\mathbf{Z}^{v,v_1}_{\ell,t} \right], \label{df:2th}
\end{align}
where $ \mathbf{L}^{v}_{\ell,t}$ and $ \mathbf{Z}^{v,v_1}_{\ell,t}$ are defined in \eqref{mal:dua:rew} and \eqref{fix:infi:lie} respectively.
\end{proposition}

In order to prove Proposition \ref{formermain}, we will use the method of finite-jump approximation developed in \cite{Wang2015gradient}. More precisely, we construct a sequence of approximation paths $\{\ell^\eps_\cdot, \eps\in[0,1)\}$ of $\ell_{\cdot}$ satisfying that $\ell^\eps_0=0$ and
\begin{align}\label{app:fix:p1}
\ell^\eps_t:=\sum_{s\in(0,t]}\Delta \ell_s1_{ \{\Delta \ell_s\geq \eps \} },\quad t>0.
\end{align}
Let $X^{\ell^{\eps},x}_t$ be the unique strong solution of SDE
\begin{align}
\dif X^{\ell^{\eps},x}_t=b(X^{\ell^{\eps},x}_t)\dif t+\sigma(X^{\ell^{\eps},x}_{t-})\dif W_{\ell^{\eps}_t},\quad X^{\ell^{\eps},x}_0=x.\label{app:fix:sde}
\end{align}
When $\eps=0$, as $\ell^0_t=\ell_t$, we replace $\ell^0_\cdot$ with $\ell_\cdot$ in \eqref{app:fix:sde}.
When $\eps>0$, since $\ell^\eps$ has finite many jumps on any finite interval, we
denote jump points of the process $\{\ell^{\eps}_{t},t\geq0\}$ by $\{t_{k},k\in\mathbb{N}\}$ with $t_{k}<t_{k+1}$. For any $v\in\mathbb{R}^{d},$ define $h$ as $h_{0}=0$ and recursively if $s\in(\ell^{\eps}_{t_{k-1}},\ell^{\eps}_{t_{k}}]$ with $k\in\mathbb{N}$,
\begin{align}
\label{mal:der}
h_s =h_{\ell^{\eps}_{t_{k-1}}}+(\beta_s
-\beta_{\ell^{\eps}_{t_{k-1}}})\sigma^{-1}
(X^{\ell^{\eps},x}_{t_k-})
\big(\nabla_vX^{\ell^{\eps},x}_{t_k-}
+\nabla_{\nabla_vX^{\ell^{\eps},x}_{t_k-}}
\sigma(X^{\ell^{\eps},x}_{t_k-})\Delta W_{\ell^{\eps}_{t_{k}}}\big),
\end{align}
where $\beta:[0,\infty)\to[0,\infty)$ is an absolutely continuous increasing function with $\beta_0=0$ and locally bounded derivative $\dot{\beta}_t$.
Referring to the proof of \cite[Theorem 3.2]{Wang2015gradient}, for any $v\in\mR^d$ and $\eps>0$, setting $h$ defined as \eqref{mal:der}, the Malliavin derivative of $X^{\ell^{\eps},x}_{t}$ with respect to $h$ and the divergence of $h$ have the following representation:
\begin{align}
\label{mal:nor}
&\qquad\qquad\qquad\qquad\qquad\
D_h X^{\ell^{\eps},x}_{t}=\beta_{\ell^{\eps}_t} \nabla_v X^{\ell^{\eps},x}_{t},\\
\delta(h)^{\ell^\eps}_t=&\int^t_0\Big\langle
\sigma^{-1}(X^{\ell^\eps,x}_{s-})
\nabla_vX^{\ell^\eps,x}_{s-},\dif W^\beta_{\ell^\eps_s}\Big\rangle
-\int^t_0\Big\langle\sigma^{-1}
(X^{\ell^\eps,x}_{s-}),\nabla_{\nabla_v X^{\ell^\eps,x}_{s-}}
\sigma(X^{\ell^\eps,x}_{s-})\Big\rangle
_{\small{\mathrm{HS}}}\dif
\beta_{\ell^\eps_s}\nonumber\\
&+\sum_{s\in(0,t]}\Big\langle\sigma^{-1}
(X^{\ell^\eps,x}_{s-})
\nabla_{\nabla_vX^{\ell^\eps,x}_{s-}}
\sigma(X^{\ell^\eps,x}_{s-})\Delta W^\beta_{\ell^\eps_{s}},\Delta W_{\ell^\eps_{s}}\Big\rangle.
\label{mal:dua}
\end{align}

Furthermore, we prepare the following lemma for later use, which will play a crucial role in proving Proposition \ref{formermain} and it's proof is postponed in Appendix \ref{sec:secJf}.

\begin{lemma}\label{fix:infi}
Let $\{\ell^\eps_\cdot, \eps\in[0,1)\}$ be defined in \eqref{app:fix:p1} and $i,j\in\{0, 1\}$. Under Assumption \ref{hyp:A}, the following estimates hold:

\begin{itemize}
\item[(1)] For any $p\geq 2$ and $v,v_1 \in \mathbb{R}^d$, there exists a constant $C=C(p,c_2,c_3)>0$ such that for all $t>0$,
\begin{align}\label{B4}
\mathbb{E}\left[\big|\nabla_{v_1}^i\nabla_{v}^j X^{\ell^{\eps},x}_t\big|^p \right]\mathbf{1}_{i+j>0}
\lesssim_C |v|^p|v_1|^p\exp\Big\{C\big((1+\ell_t)^{2p} + t\big)\Big\}.
\end{align}

\item[(2)] For any $v,v_1\in\mathbb{R}^d$ and $T>0$, it holds that
\begin{align}\label{B5}
\lim_{\eps\downarrow 0}\sup_{t\in[0,T]}\mathbb{E}\left[\big|\nabla_{v_1}^i\nabla_{v}^j (X^{\ell^\eps,x}_t - X^{\ell,x}_t)\big|^2 \right]= 0.
\end{align}

\item[(3)] Let $\beta_t:=t\wedge\ell_{\tau_\ell}$ and $p\geq 2$. For any $v,v_1\in\mathbb{R}^d$ and $T>0$, it holds that
\begin{align}\label{B6}
\lim_{\eps\downarrow0}\sup_{t\in[0,T]}\mathbb{E}\Big[\big|\mathbf{L}^{v}_{\ell,t} - \delta(h)^{\ell^\eps}_t \big|^p + \big|\mathbf{Z}^{v,v_1}_{\ell,t} - [\nabla_{v_1},D_{h}]X^{\ell^\eps,x}_t\big|^2\Big] = 0,
\end{align}
where $[\nabla_{v_1},D_{h}]:=\nabla_{v_1}D_{h}-D_{h}\nabla_{v_1}$, $\delta(h)^{\ell^\eps}_t $ is given as \eqref{mal:dua} for any $\eps>0$, and $ \mathbf{L}^{v}_{\ell,t}$ and $ \mathbf{Z}^{v,v_1}_{\ell,t}$ are defined as \eqref{mal:dua:rew} and \eqref{fix:infi:lie} respectively.
\end{itemize}
\end{lemma}

We now proceed to establish the validity of Proposition \ref{formermain}.

\begin{proof}[Proof of Proposition \ref{formermain}]
Let $v, v_1\in\mathbb{R}^{d}$, $t>0$ be fixed, and consider a fixed path $\ell_\cdot$ of $S_\cdot$.
Define $\beta_t = t \wedge \ell_{\tau_\ell}$ and $h$ is given by \eqref{mal:der} for any $\varepsilon > 0$.
Given the path $\ell_\cdot$ with infinitely many small jumps and $\ell_t > 0$, the limit
\begin{equation*}
\lim_{\varepsilon \downarrow 0} \beta_{\ell^{\varepsilon}_t} = \lim_{\varepsilon \downarrow 0} \ell^{\varepsilon}_t \wedge \ell_{\tau_\ell} = \ell_t \wedge \ell_{\tau_\ell} > 0
\end{equation*}
implies the existence of $\varepsilon_0 \in (0,1)$ satisfying
\begin{align}\label{B3}
\beta_{\ell^{\varepsilon}_t} > 0, \quad \forall \varepsilon \in (0, \varepsilon_0).
\end{align}
The existence of $\cC^2$-stochastic flow $X^{\ell^\varepsilon,x}_t$ for SDE \eqref{app:fix:sde} follows from \cite[p. 310, Theorem 40]{Protter2004stochastic} under Assumption \ref{hyp:A} \hyperref[A2]{$(A2)$}. By \eqref{mal:nor} and \eqref{B3}, we have for any $\varepsilon\in (0, \varepsilon_0)$, the second-order derivative decomposes as
\begin{align*}
\nabla_{v_1}\nabla_{v}\mathbb{E}[f(X^{\ell^{\varepsilon},x}_t)]
&= \mathbb{E}[\nabla f(X^{\ell^{\varepsilon},x}_t) \nabla_{v_1}\nabla_{v}X^{\ell^{\varepsilon},x}_t]
+ \mathbb{E}[\nabla^2 f(X^{\ell^{\varepsilon},x}_t) \nabla_{v_1}X^{\ell^{\varepsilon},x}_t\nabla_{v}X^{\ell^{\varepsilon},x}_t] \\
&= \beta^{-1}_{\ell^{\varepsilon}_t}\mathbb{E}[\nabla f(X^{\ell^{\varepsilon},x}_t) \nabla_{v_1}D_h X^{\ell^{\varepsilon},x}_t]
+ \beta^{-1}_{\ell^\eps_t}\mathbb{E}[\nabla^2 f(X^{\ell^{\eps},x}_t) \nabla_{v_1}X^{\ell^{\eps},x}_tD_h X^{\ell^{\eps},x}_t] \\
&=: I_1 + I_2.
\end{align*}
Applying \eqref{mal:div}, \eqref{mal:nor}, and \eqref{B3} to $I_2$, we have for any $\varepsilon\in (0, \varepsilon_0)$
\begin{align*}
I_2 &= \beta^{-1}_{\ell^{\eps}_t}\mathbb{E}
[D_h(\nabla f(X^{\ell^{\eps},x}_t) \nabla_{v_1}X^{\ell^{\eps},x}_t) ]
-\beta^{-1}_{\ell^{\eps}_t}\mathbb{E}
[\nabla f(X^{\ell^{\eps},x}_t) D_h\nabla_{v_1}X^{\ell^{\eps},x}_t ]\nonumber\\
&\overset{\eqref{mal:dua}}= \beta^{-1}_{\ell^{\eps}_t}\mathbb{E}[\nabla f(X^{\ell^{\eps},x}_t) \nabla_{v_1}X^{\ell^{\eps},x}_t
\delta(h)^{\ell^{\eps}}_t]
-\beta^{-1}_{\ell^{\eps}_t}\mathbb{E}[\nabla f(X^{\ell^{\eps},x}_t) D_h\nabla_{v_1}X^{\ell^{\eps},x}_t ].
\end{align*}
Then we have for any $\varepsilon\in (0, \varepsilon_0)$
\begin{align}
\nabla_{v_1}\nabla_{v}\mathbb{E}
[f(X^{\ell^{\eps},x}_t)]
=\beta^{-1}_{\ell^{\eps}_t}\mathbb{E}\Big[\nabla f(X^{\ell^{\eps},x}_t) \Big(\nabla_{v_1}X^{\ell^{\eps},x}_t\delta(h)
^{\ell^{\eps}}_t+(\nabla_{v_1}D_{h} - D_{h}\nabla_{v_1})X^{\ell^\eps,x}_t\Big) \Big].\label{key:eq}
\end{align}
As $\varepsilon \downarrow 0$, the monotonic convergence $\ell^\varepsilon_t \uparrow \ell_t$ and $\ell^\varepsilon_{t-} \uparrow \ell_{t-}$ hold. By the continuity of $\beta_t$ and Lemma \ref{fix:infi}, we have
\begin{align*}
\lim_{\varepsilon \downarrow 0} \beta^{-1}_{\ell^{\varepsilon}_t} \mathbb{E}[\nabla f(X^{\ell^{\varepsilon},x}_t) \nabla_{v_1}X^{\ell^{\varepsilon},x}_t\delta(h)^{\ell^{\varepsilon}}_t]
&= \beta^{-1}_{\ell_t} \mathbb{E}[\nabla f(X^{\ell,x}_t) \nabla_{v_1}X^{\ell,x}_t\mathbf{L}^{v}_{\ell,t}]
\end{align*}
and
\begin{align*}
\lim_{\varepsilon \downarrow 0} \beta^{-1}_{\ell^{\varepsilon}_t} \mathbb{E}[\nabla f(X^{\ell^{\varepsilon},x}_t) (\nabla_{v_1}D_{h} - D_{h}\nabla_{v_1})X^{\ell^\eps,x}_t]
&= \beta^{-1}_{\ell_t} \mathbb{E}[\nabla f(X^{\ell,x}_t) \mathbf{Z}^{v,v_1}_{\ell,t}].
\end{align*}
Taking $\varepsilon \downarrow 0$ in \eqref{key:eq} while observing the convergence of all terms, we conclude that \eqref{df:2th} holds. Then we complete the proof.
\end{proof}

\subsubsection{Proof of Lemma \ref{lemma:g} (i)}

In this subsection, we present the proof of \eqref{est:grad123} by Proposition \ref{formermain}. We first state the following moment estimates.

\begin{lemma}\label{lastmoment}
Given $v,v_1\in\partial\mathrm{B}$, let Assumption \ref{hyp:A} hold. For any $p\in[2,\infty)$ and $q\in[1,\alpha)$, there exists a constant $C=C(p,q,c_2,c_3)>0$ such that for any $t\in(0,1]$
\begin{align*}
&\Big\{\mathbb{E}\Big[
\ell^{-p}_{\bm{\tau}^{\ell}_t}\mathbb{E}
\big[\big|\mathbf{L}^{v}_{\ell,t}\big|^{p}
\big]\big|_{\ell=S} \Big]\Big\}^{\frac{1}{p}}
+\Big\{\mathbb{E}\Big[
\ell^{-q}_{\bm{\tau}^{\ell}_t}\mathbb{E}
\big[\big|\mathbf{Z}^{v,v_1}_{\ell,t}
\big|^{q}\big]\big|_{\ell=S} \Big]\Big\}^{\frac{1}{q}}
\lesssim_C t^{-\frac{1}{\alpha}}\e^{C(2-\alpha)^{-1}},
\end{align*}
where $\mathbf{L}^{v}_{\ell,t}$ and $\mathbf{Z}^{v,v_1}_{\ell,t}$ are defined in \eqref{mal:dua:rew} and \eqref{fix:infi:lie} respectively.
\end{lemma}
\begin{proof}[Proof of Lemma \ref{lastmoment}]
(i): We prove that for any $p\geq 2$, there exists a constant $C=C(p,c_2,c_3)>0$ such that
\begin{align}
\label{e1:all}
\mathbb{E}\Big[
\ell^{-p}_{\bm{\tau}^{\ell}_t}\mathbb{E}
\big[|\mathbf{L}^{v}_{\ell,t}|^{p}
\big]\big|_{\ell_\cdot=S_\cdot} \Big]
\leq Ct^{-\frac{p}{\alpha}},\quad \forall t\in (0,1].
\end{align}
By \eqref{mal:dua:rew} and   $\mathbf{C}_r$-inequality\footnote{
For any $a,b\in\mathbb{R}$ and $r>0$, $|a+b|^r\leq \max(1,2^{r-1})(|a|^r+|b|^r)$.
}, we have
\begin{align*}
\mathbb{E}\big[|\mathbf{L}^{v}_{\ell,t}|^{p}\big]
\lesssim_p&\mathbb{E}\Big[\Big|\int^{\bm{\tau}^{\ell}_t}_0
\left\langle\sigma^{-1}
(X^{\ell,x}_{s-})\nabla_vX^{\ell,x}_{s-},\dif W_{\ell_s}\right\rangle\Big|^{p}\Big]
+\mathbb{E}\Big[\Big| \int^{\bm{\tau}^{\ell}_t}_0
\left\langle\sigma^{-1}(X^{\ell,x}_{s-}),
\nabla_{\nabla_vX^{\ell,x}_{s-}}
\sigma(X^{\ell,x}_{s-})
\right\rangle_{\small{\mathrm{HS}}}\dif\ell_s\Big|^{p}\Big]
\nonumber\\
&+\mathbb{E}\Big[\Big|\sum_{s\in(0,\bm{\tau}^{\ell}_t]}
\left\langle\sigma^{-1}(X^{\ell,x}_{s-})
\nabla_{\nabla_vX^{\ell,x}_{s-}}
\sigma(X^{\ell,x}_{s-})\Delta W_{\ell_{s}},\Delta W_{\ell_{s}}\right\rangle\Big|^{p}\Big]
=:I_1+I_2+I_3.
\end{align*}
For $I_1$, by Burkholder's inequality (see \cite[Lemma 2.1]{Wang2015gradient}) and  Minkowski's inequality, we have for any $t\in(0,1]$
\begin{align*}
I_1&\lesssim_{p} \mathbb{E}\Big[\Big(\int^{\bm{\tau}^{\ell}_t}_0\big|
\sigma^{-1}(X^{\ell,x}_{s-})
\nabla_vX^{\ell,x}_{s-}\big|^2\dif \ell_s \Big)^{\frac{p}{2}}\Big]\lesssim_{p} \Big(\int^{\bm{\tau}^{\ell}_t}_0 \left(\mathbb{E} \big[\big| \sigma^{-1}(X^{\ell,x}_{s-})
\nabla_vX^{\ell,x}_{s-}\big|^p\big]\right)^{\frac{2}{p}}
\dif \ell_s \Big)^{\frac{p}{2}}.
\end{align*}
By Assumption \ref{hyp:A} and  \eqref{B4}, noting that $\ell_t\mathbf{1}_{[0,\tau_\ell)}(t)<1$, we have for any $t\in(0,1]$
\begin{align*}
I_1&\lesssim_{p,c_2}\Big(\int^{\bm{\tau}^{\ell}_t}_0 \left(
\mathbb{E}\big[\big|\nabla_vX^{\ell,x}_{s-}\big|^p\big]\right)^{\frac{2}{p}}
\dif \ell_s \Big)^{\frac{p}{2}}
\lesssim_{p,c_2,c_3}\ell^{\frac{p}{2}}_{\bm{\tau}^{\ell}_t}.
\end{align*}
For $I_2$ and $I_3$, by  Minkowski's inequality and Assumption \ref{hyp:A}, we also have
\begin{align*}
I_2\lesssim \Big(\int^{\bm{\tau}^{\ell}_t}_0
\left(\mathbb{E}\big[\big|\nabla_vX^{\ell,x}_{s-}
\big|^{p}\big]\right)^{\frac{1}{p}}\dif \ell_s \Big)^{p}
\lesssim_{p,c_2,c_3}\ell^{p}_{\bm{\tau}^{\ell}_t}
\end{align*}
and
\begin{align*}
I_3 &\lesssim \Big(\sum_{s\in[0,\bm{\tau}^{\ell}_t]}
\big(\mathbb{E}\big[|\nabla_vX^{\ell,x}_{s-}|^p |\Delta W_{\ell_{s}}|^{2p}\big]\big)^{\frac{1}{p}}\Big)^{p}
=\Big(\sum_{s\in[0,\bm{\tau}^{\ell}_t]}
\big(\mathbb{E}\big[|\nabla_vX^{\ell,x}_{s-}|^p\big] \big)^{\frac{1}{p}}\Delta \ell_{s}\Big)^{p}\\
&= \Big(\int^{\bm{\tau}^{\ell}_t}_0
\left(\mathbb{E}\big[\big|\nabla_vX^{\ell,x}_{s-}
\big|^{p}\big]\right)^{\frac{1}{p}}\dif \ell_s \Big)^{p}
\lesssim_{p,c_2,c_3}\ell^{p}_{\bm{\tau}^{\ell}_t}.
\end{align*}
Combining the above estimates, denoting by $\tau:=\inf\{t\geq 0 \mid S_t>1\}$, we have
\begin{align}
&\mathbb{E}\big[
\ell^{-p}_{\bm{\tau}^{\ell}_t}\mathbb{E}
\big[|\mathbf{L}^{v}_{\ell,t}|^{p}
\big]\big|_{\ell_\cdot=S_\cdot} \big]
\lesssim_{p,c_2,c_3} 1+\mathbb{E}\big[S_{t\wedge\tau}^{-p/2}\big]
\lesssim_{p,c_2,c_3} 1+\mathbb{E}\big[S_{t}^{-p/2}\big].\label{2th:k3}
\end{align}
By \eqref{sub:lap}, noting that $s^{-r}=\frac{1}{\Gamma(r)}\int^\infty_0u^{r-1}\e^{-us}\dif u$ for any $r>0$, we get
\begin{align}
\mathbb{E}S_t^{-r}&
\overset{\eqref{sub:lap}}\lesssim \frac{1}{\Gamma(r)}\int^\infty_0u^{r-1}
 \e^{-2^{\frac{\alpha-2}{2}}u^{\frac{\alpha}{2}}t} \dif u
\lesssim t^{-\frac{2r}{\alpha}}\frac{\Gamma(2r)}{\Gamma(r)/4^{r}}
\lesssim_r t^{-\frac{2r}{\alpha}}.
\label{S:-r}
\end{align}
Combining \eqref{2th:k3} and \eqref{S:-r}, we prove \eqref{e1:all}.

(ii): We prove that for any $q\in[1,\alpha)$, there exists a constant $C=C(q,c_2,c_3)>0$ such that
\begin{align}
\label{EA}
\mathbb{E}\Big[
\ell^{-q}_{\bm{\tau}^{\ell}_t}\mathbb{E}
\big[\big|\mathbf{Z}^{v,v_1}_{\ell,t}
\big|^{q}\big]\big|_{\ell_\cdot=S_\cdot}\Big]
\lesssim_C t^{\frac{2(1-q)}{\alpha}}\e^{C (2-\alpha)^{-1}},\quad \forall t\in (0,1].
\end{align}
We fix $t\in(0,1]$.
By applying It\^o's formula (see \cite[Theorem 4.4.10]{Applebaum2009Levy}) to \eqref{fix:infi:lie},  we have for any $f\in \cC^2(\mR^d)$,
\begin{align}
f(\mathbf{Z}^{v,v_1}_{\ell,t})
=&f(0)+ \int^t_0\big \langle \nabla f(\mathbf{Z}^{v,v_1}_{\ell,s-}),
\dif \mathbf{Z}^{v,v_1}_{\ell,s}\big\rangle+
\sum_{s\in(0,t]}
\big[\delta_{\Delta\mathbf{Z}^{v,v_1}_{\ell,s}}f
(\mathbf{Z}^{v,v_1}_{\ell,s-})-\big \langle \nabla f(\mathbf{Z}^{v,v_1}_{\ell,s-}), \Delta\mathbf{Z}^{v,v_1}_{\ell,s}\big \rangle \big]\no\\
=&f(0)+ \int^t_0\big\langle \nabla f(\mathbf{Z}^{v,v_1}_{\ell,s-}),
\nabla_{\mathbf{Z}^{v,v_1}_{\ell,s}}
b(X^{\ell,x}_{s})\big\rangle\dif s
+Q^{\ell,x}_t+K^{\ell,x}_t
+\mathcal{E}^{\ell,x}_t,
\label{gZt}
\end{align}
where $\mathcal{E}^{\ell,x}_t$ is a stochastic process satisfying $\mathbb{E}\mathcal{E}^{\ell,x}_t=0$, $Q^{\ell,x}_t$ and $K^{\ell,x}_{t}$ are defined as below by \eqref{Ztvv1}:
\begin{align}
Q^{\ell,x}_t:=&\sum_{s\in(0,\bm{\tau}^{\ell}_t]}
\big[f(\mathbf{Z}^{v,v_1}_{\ell,s})
-f
\big(\mathbf{Z}^{v,v_1}_{\ell,s-}
+\nabla_{\mathbf{Z}^{v,v_1}_{\ell,s-}}
\sigma(X^{\ell,x}_{s-})\Delta W_{\ell_s}\big)-\big \langle \nabla f(\mathbf{Z}^{v,v_1}_{\ell,s-}),
\Delta \mathcal{Z}^{v,v_1}_{\ell,s}\big \rangle \big]\no\\
&+\int^{\bm{\tau}^{\ell}_t}_0
\big\langle \nabla f(\mathbf{Z}^{v,v_1}_{\ell,s-}),\nabla_{v_1} \nabla_{v}X^{\ell,x}_{s-}-\nabla_{ \nabla_{v_1} X^{\ell,x}_{s-}}\sigma(X^{\ell,x}_{s-})
\sigma^{-1}(X^{\ell,x}_{s-})
\nabla_vX^{\ell,x}_{s-}\big\rangle\dif \ell_s,\label{QgZt}\\
K^{\ell,x}_{t}:=&\sum_{s\in(0,t]}
\big[\delta_{\nabla_{\mathbf{Z}^{v,v_1}_{\ell,s-}}
\sigma(X^{\ell,x}_{s-})\Delta W_{\ell_s}}f
(\mathbf{Z}^{v,v_1}_{\ell,s-})-\big \langle \nabla f(\mathbf{Z}^{v,v_1}_{\ell,s-})
,\nabla_{\mathbf{Z}^{v,v_1}_{\ell,s-}}
\sigma(X^{\ell,x}_{s-})\Delta W_{\ell_s}\big \rangle \big]\label{KgZt}.
\end{align}

{\bf Case 1: We consider $t\leq \tau_\ell$.}  By \eqref{gZt} with $f(x)=|x|^2$,  as $|x+a+b|^2-|x+a|^2-2\langle x,b \rangle=\langle b+2a,b \rangle,$ 
we have for any $t\leq \tau_\ell $,
\begin{align*}
|\mathbf{Z}^{v,v_1}_{\ell,t}|^{2}
=&2\int^t_0\big\langle \mathbf{Z}^{v,v_1}_{\ell,s},
\nabla_{\mathbf{Z}^{v,v_1}_{\ell,s}}
b(X^{\ell,x}_{s})\big\rangle\dif s+
\sum_{s\in(0,t]} \big\langle \Delta \mathcal{Z}^{v,v_1}_{\ell,s}+2\nabla_{\mathbf{Z}^{v,v_1}_{\ell,s-}}
\sigma(X^{\ell,x}_{s-})\Delta W_{\ell_s},\Delta \mathcal{Z}^{v,v_1}_{\ell,s}\big\rangle
\\
&+2\int^{t}_0
\big\langle \mathbf{Z}^{v,v_1}_{\ell,s-},\nabla_{v_1} \nabla_{v}X^{\ell,x}_{s-}-\nabla_{ \nabla_{v_1} X^{\ell,x}_{s-}}\sigma(X^{\ell,x}_{s-})
\sigma^{-1}(X^{\ell,x}_{s-})
\nabla_vX^{\ell,x}_{s-}\big\rangle\dif \ell_s\\
&+\sum_{s\in(0,t]}
\big \langle\nabla_{\mathbf{Z}^{v,v_1}_{\ell,s-}}
\sigma(X^{\ell,x}_{s-})\Delta W_{\ell_s} ,\nabla_{\mathbf{Z}^{v,v_1}_{\ell,s-}}
\sigma(X^{\ell,x}_{s-})\Delta W_{\ell_s}\big \rangle+\mathcal{E}^{\ell,x}_t.
\end{align*}
By \eqref{de:tl}, noting that $\ell_{t-}\mathbf{1}_{(0,\tau_{\ell}]}(t)\leq 1$, we have for any $t\leq\tau_\ell$,
\begin{align}\label{in:A1}
\mE\big[|\Delta \mathcal{Z}^{v,v_1}_{\ell,t}|^{2}\big]
\overset{\eqref{Ztvv1}}\lesssim \mE\big[|\nabla_{v_1} \nabla_{v}X^{\ell,x}_{t-}|^{2}+|\nabla_{v}X^{\ell,x}_{t-}|^{4}\Delta \ell_t+|\nabla_{v_1}X^{\ell,x}_{t-}|^{4}\Delta \ell_t\big]
\overset{\eqref{B4}}\lesssim_{\!\!\!\!\!c_2,c_3} (1+\Delta \ell_t)\Delta \ell_t.
\end{align}
By \eqref{in:A1}, we have
for any $t\leq\tau_\ell$,
\begin{align*}
\mE\big[|\mathbf{Z}^{v,v_1}_{\ell,t}|^{2}\big]
\lesssim_{c_2,c_3}&\int^t_0\mE\big[|\mathbf{Z}^{v,v_1}_{\ell,s}|^{2}\big]\dif s+\int^t_0\mE\big[|\mathbf{Z}^{v,v_1}_{\ell,s-}|^{2}\big]\dif \ell_s+\sum_{s\in(0,t]} \mE\big[|\Delta \mathcal{Z}^{v,v_1}_{\ell,s}|^{2}\big]\\
&+\int^t_0\mE\big[|\nabla_{v_1} \nabla_{v}X^{\ell,x}_{s-}|^{2}+|\nabla_{v}X^{\ell,x}_{s-}|^{2}|\nabla_{v_1}X^{\ell,x}_{s-}|^{2}\big]\dif \ell_s\\
\lesssim_{c_2,c_3}&\int^t_0\mE\big[|\mathbf{Z}^{v,v_1}_{\ell,s-}|^{2}\big]\dif (s+\ell_s)
+\ell_t+\ell^2_t.
\end{align*}
Then by Gronwall's inequality, we have for any $t\leq\tau_\ell$, there is $C=C(c_2,c_3)>0$ such that
\begin{align}\label{e:C1}
\mE\big[|\mathbf{Z}^{v,v_1}_{\ell,t}|^{2}\big]\lesssim_{C}(\ell_t+\ell^2_t)\e^{C\ell_t}.
\end{align}

{\bf Case 2: We consider $t\in(0,1]$.}  In order to estimate \eqref{EA}, we first consider
$G_t$ defined as below:
$$
G_t:=\mathbb{E}\big[
\ell^{-q}_{\bm{\tau}^{\ell}_t}\mE[F^{\eps}_{q}(\mathbf{Z}^{v,v_1}_{\ell,t})]|_{\ell_\cdot=S_\cdot}\big],
$$
where $q\in[1,\alpha)$, $\eps\in(0,1]$ and $
F^{\eps}_{m}(x):=(\eps^2+|x|^2)
^{\frac{m}{2}}$ for any $m\in\mR$.
We note that for any $x\in\mathbb{R}^d$,
\iffalse
\begin{align}\label{in:A2}
|\nabla F^{\eps}_{q}(x)|=qF^{\eps}_{q-2}(x)|x|\leq qF^{\eps}_{q-1}(x).
\end{align}
\fi
%
%
$$
\nabla F^{\eps}_{q}(x)=qF^{\eps}_{q-2}(x)x,\quad
\nabla^2 F^{\eps}_{q}(x)=q F^{\eps}_{q-2}(x)\mI+q (q-2)F^{\eps}_{q-4}(x)x\otimes x.
$$
It is clear that for any $x,y\in\mR^d$,
\begin{align}\label{in:A2}
|\nabla F^{\eps}_{q}(x)|\leq qF^{\eps}_{q-1}(x),\quad
\langle \nabla^2 F^{\eps}_{q}(x)y,y\rangle\leq q F_{q-2}^{\eps}(x) |y|^2.
\end{align}
%%%%%%%%
By \eqref{gZt} with $f=F^{\eps}_{q}$, using Young's  inequality, we have for any $t\in(0,1]$ and $q\in[1,\alpha)$
\begin{align*}
\mE\big[F^{\eps}_{q}&(\mathbf{Z}^{v,v_1}_{\ell,t})\big]
\overset{\eqref{in:A2}}\lesssim_{\bm{\theta}} \eps^q+ \int^t_0\mE\big[F^{\eps}_{q}(\mathbf{Z}^{v,v_1}_{\ell,s})\big]\dif s+\cQ^{\ell,x}_{t}+\cK^{\ell,x}_t,
\end{align*}
where
\begin{align*}
\cQ^{\ell,x}_{t}:=&\sum_{s\in(0,\bm{\tau}^{\ell}_t]}
\mathbb{E}\big[ |\Delta \mathcal{Z}^{v,v_1}_{\ell,s}|+|\mathbf{Z}^{v,v_1}_{\ell,s-}|^{q-1}|\Delta \mathcal{Z}^{v,v_1}_{\ell,s}|+|\mathbf{Z}^{v,v_1}_{\ell,s-}|^{q-1}|\Delta W_{\ell_s}|^{q-1}|\Delta \mathcal{Z}^{v,v_1}_{\ell,s}|
+|\Delta \mathcal{Z}^{v,v_1}_{\ell,s}|^{q}\big]\\
&+\sum_{s\in(0,\bm{\tau}^{\ell}_t]}
\mE\big[|\mathbf{Z}^{v,v_1}_{\ell,s-}|^{2}+\big(1+|\mathbf{Z}^{v,v_1}_{\ell,s-}|^{q-1}\big)\big(|\nabla_{v_1} \nabla_{v}X^{\ell,x}_{s-}|+|\nabla_{v_1} X^{\ell,x}_{s-}|^2+
|\nabla_{v_1} X^{\ell,x}_{s-}|^2\big)\big]\Delta \ell_s, \\
\cK^{\ell,x}_t:=&\sum_{s\in(\tau_\ell,t]}
\mathbb{E}\big[\delta_{\nabla_{\mathbf{Z}^{v,v_1}_{\ell,s-}}
\sigma(X^{\ell,x}_{s-})\Delta W_{\ell_s}}F^{\eps}_{q}
(\mathbf{Z}^{v,v_1}_{\ell,s-})-\big \langle \nabla F^{\eps}_{q}(\mathbf{Z}^{v,v_1}_{\ell,s-})
,\nabla_{\mathbf{Z}^{v,v_1}_{\ell,s-}}
\sigma(X^{\ell,x}_{s-})\Delta W_{\ell_s}\big \rangle \big]\mathbf{1}_{t>\tau_\ell}.
\end{align*}

For $\cQ^{\ell,x}_{t}$, we recall $\ell_{t-}\mathbf{1}_{(0,\tau_{\ell}]}(t)\leq 1$. As the same procedure in {\bf Case 1}, we have for any $t\in(0,1]$ and $q\in[1,\alpha)$,
\begin{align*}
\cQ^{\ell,x}_{t}\lesssim_{c_2,c_3}\ell_{\bm{\tau}^{\ell}_t}+\ell^{\frac{3q}{2}}_{\bm{\tau}^{\ell}_t}.
\end{align*}
Noting that $s^{r}=\frac{r}{\Gamma(1-r)}\int^\infty_0(1-\e^{-us})u^{-1-r}\dif u$ for any $r\in(0,\alpha/2)$, we have
$$
\mE[S^r_t]\overset{\eqref{sub:lap}}\lesssim_{\!\!\!r} \int^\infty_0(1-\e^{-u^{\frac{\alpha}{2}}t})u^{-1-r}\dif u\lesssim t^{\frac{2r}{\alpha}}\left(\frac{\alpha}{\alpha-2r}+\frac{\alpha}{2r}\right) \lesssim_r t^{\frac{2r}{\alpha}}.
$$
By \eqref{2th:k3}, as $q\in [1,\alpha)$, we have
\begin{align*}
G_t\lesssim_{c_2,c_3} \eps^q+\mathbb{E}\big[S_{t\wedge\tau}^{1-q}+S_{t\wedge\tau}^{q/2}\big]
+ \int^t_0G_s
\dif s+H_t
\lesssim_{c_2,c_3,q}\eps^q+t^{\frac{2(1-q)}{\alpha}}+ \int^t_0G_s
\dif s+H_t,
\end{align*}
where
$
H_t:=\mathbb{E}\big[\ell^{-q}_{\tau^{\ell}}\cK^{\ell,x}_t\big|_{\ell_\cdot=S_\cdot}\big].
$

Now, we focus on estimating $H_t$. Let $N(t,\dif z)$ be the Poisson random measure associated with $L^\alpha_t$, i.e.,
$$
N(t,A):=\sum_{s\in(0,t]}\mathbf{1}_{A}(L^\alpha_s-L^\alpha_{s-}),\  A\in\sB(\mR^d).
$$
Let
${\widetilde N}(t,\dif z):=N(t,\dif z)-c_{d,\alpha}t|z|^{-d-\alpha}\dif z$
be the compensated Poisson random measure, where $c_{d,\alpha}$ is the same constant as in \eqref{op:levy}.
By L\'evy-It\^o's decomposition, one can write
\begin{align*}
W_{S_t}=\int_{0<|z|\leq 1}z{\widetilde N}(t,\dif z)+\int_{|z|>1}z N(t,\dif z).
\end{align*}
Denote by $Z_t:=\mathbf{Z}^{v,v_1}_{\ell,t}|_{\ell_\cdot=S_\cdot}$ and $\bm{\tau}:=\tau_\ell|_{\ell_\cdot=S_\cdot}$. By \eqref{sde:L1} and \eqref{sde:fix_xp}, we have
for any $t>\bm{\tau}$,
\begin{align}
Z_t=&Z_{\bm{\tau}}+\int^t_{\bm{\tau}}\nabla_{Z_s}b(X_{s})\dif s
+\int^t_{\bm{\tau}}\nabla_{Z_{s-}}\sigma(X_{s-})\dif W_{S_s}\no\\
=&Z_{\bm{\tau}}+\int^t_{\bm{\tau}}\nabla_{Z_s}b(X_{s})\dif s+\int^t_{\bm{\tau}}\int_{0<|z|\leq 1}\nabla_{Z_{s-}}\sigma(X_{s-})z{\widetilde N}(\dif s,\dif z)\no\\
&+\int^t_{\bm{\tau}}\int_{|z|>1}\nabla_{Z_{s-}}\sigma(X_{s-})z N(\dif s,\dif z).\label{eq:A5}
\end{align}
By It\^o's formula (see \cite[Theorem 4.4.7]{Applebaum2009Levy}) and \eqref{eq:A5}, we have for any $t>\bm{\tau}$,
\begin{align*}
F^{\eps}_{q}(Z_t)=&F^{\eps}_{q}(Z_{\bm{\tau}})
+\int^t_{\bm{\tau}}\langle F^{\eps}_{q}(Z_s),\nabla_{Z_s}b(X_{s})\rangle \dif s+\int^t_{\bm{\tau}}\int_{\mR^d \setminus \{0\} }\big[\delta_{\nabla_{Z_{s-}}\sigma(X_{s-})z}F^{\eps}_{q}
(Z_{s-})\big]{\widetilde N}(\dif s,\dif z)\\
&+\int^t_{\bm{\tau}}\int_{\mR^d \setminus \{0\} }\big[\delta_{\nabla_{Z_{s-}}\sigma(X_{s-})z}F^{\eps}_{q}
(Z_{s-})-\big \langle \nabla F^{\eps}_{q}(Z_{s-})
,\nabla_{Z_{s-}}
\sigma(X_{s-})z\big \rangle \big]\frac{c_{d,\alpha}}{|z|^{d+\alpha}}\dif z\dif s.
\end{align*}
On the other hand, by \eqref{gZt}, we also have for any $t>\bm{\tau}$,
\begin{align*}
F^{\eps}_{q}(Z_t)=&F^{\eps}_{q}(Z_{\bm{\tau}})
+\int^t_{\bm{\tau}}\langle F^{\eps}_{q}(Z_s),\nabla_{Z_s}b(X_{s})\rangle \dif s+\big(\mathcal{E}^{\ell,x}_{t}|_{\ell_\cdot=S_\cdot}-\mathcal{E}^{\ell,x}_{\tau_\ell}|_{\ell_\cdot=S_\cdot}\big)\\
&+\sum_{s\in(\bm{\tau},t]}
\big[\delta_{\nabla_{Z_{s-}}
\sigma(X_{s-})\Delta W_{S_s}}f
(Z_{s-})-\big \langle \nabla f(Z_{s-})
,\nabla_{Z_{s-}}
\sigma(X_{s-})\Delta W_{S_s}\big \rangle \big].
\end{align*}
Then we have  for any $t>\bm{\tau}$,
\begin{align*}
&\sum_{s\in(\bm{\tau},t]}
\big[\delta_{\nabla_{Z_{s-}}
\sigma(X_{s-})\Delta W_{S_s}}f
(Z_{s-})-\big \langle \nabla f(Z_{s-})
,\nabla_{Z_{s-}}
\sigma(X_{s-})\Delta W_{S_s}\big \rangle \big]\\
=&\int^t_{\bm{\tau}}\mathcal{L}^\alpha_{\bm{\sigma}_s} F^{\eps}_{q}
(Z_{s})\dif s+\int^t_{\bm{\tau}}\int_{\mR^d \setminus \{0\} }\big[\delta_{\nabla_{Z_{s-}}\sigma(X_{s-})z}F^{\eps}_{q}(Z_{s-})\big]{\widetilde N}(\dif s,\dif z)-\big(\mathcal{E}^{\ell,x}_{t}|_{\ell_\cdot=S_\cdot}-\mathcal{E}^{\ell,x}_{\tau_\ell}|_{\ell_\cdot=S_\cdot}\big),
\end{align*}
where $\bm{\sigma}_s:=\nabla_{Z_{s}}\sigma(X_{s})$, and $\mathcal{L}^\alpha_{\bm{\sigma}_s}F^{\eps}_{q}(x)$ is given similarly to \eqref{op:levy} as
\begin{align}\label{op:levy1}
\mathcal{L}^\alpha_{\bm{\sigma}_s} F^{\eps}_{q}(x):= c_{d,\alpha}\int_{\mathbb{R}^d \setminus \{0\} } \frac{\delta_{\bm{\sigma}_s z}F^{\eps}_{q}(x) - \nabla_{\bm{\sigma}_s z}F^{\eps}_{q}(x)}{|z|^{d+\alpha}}\dif z.
\end{align}
\iffalse
$$
\int_{\mR^d}\big[\delta_{\nabla_{Z_{s-}}\sigma(X_{s-})z}F^{\eps}_{q}
(Z_{s-})-\big \langle \nabla F^{\eps}_{q}(Z_{s-})
,\nabla_{Z_{s-}}
\sigma(X_{s-})z\big \rangle \big]\frac{c_{d,\alpha}}{|z|^{d+\alpha}}\dif z
$$
\fi
By Galmarino's test \cite[p. 47]{Raj2023Algorithmic}, we have $\cF_{\bm{\tau}}:=\sigma\big\{\sigma\{ W_{r},r\leq S_{t\wedge\bm{\tau}}\}\otimes\sigma\{ S_{r\wedge\bm{\tau}}, r\leq t\},t\geq 0\big\}$ is the natural filtration before $\bm{\tau}$.
By Doob's optional sampling theorem, we then have
\begin{align*}
&\mathbb{E}\Big[S^{-q}_{\bm{\tau}}\mathbf{1}_{\{t>\bm{\tau}\}}\int^t_{\bm{\tau}}\int_{\mR^d \setminus \{0\} }\big[\delta_{\nabla_{Z_{s-}}\sigma(X_{s-})z}F^{\eps}_{q}(Z_{s-})\big]{\widetilde N}(\dif s,\dif z)\Big]\\
=&\mathbb{E}\Big[S^{-q}_{\bm{\tau}}\mathbf{1}_{\{t>\bm{\tau}\}}\mathbb{E}\Big[\int^t_{\bm{\tau}}\int_{\mR^d \setminus \{0\} }\big[\delta_{\nabla_{Z_{s-}}\sigma(X_{s-})z}F^{\eps}_{q}(Z_{s-})\big]{\widetilde N}(\dif s,\dif z)\mid \cF_{\bm{\tau}}\Big]\Big]=0,
\end{align*}
and similarly get
\begin{align*}
\mathbb{E}\Big[S^{-q}_{\bm{\tau}}\mathbf{1}_{\{t>\bm{\tau}\}}\big(\mathcal{E}^{\ell,x}_{t}|_{\ell_\cdot=S_\cdot}-\mathcal{E}^{\ell,x}_{\tau_\ell}|_{\ell_\cdot=S_\cdot}\big)\Big]
=\mathbb{E}\Big[\ell^{-q}_{\tau^\ell}\mathbf{1}_{\{t>\tau^\ell\}}\mathbb{E}\big[\mathcal{E}^{\ell,x}_{t}-\mathcal{E}^{\ell,x}_{\tau_\ell}\big]|_{\ell_\cdot=S_\cdot}\Big]\overset{\eqref{gZt}}=0.
\end{align*}
On the other hand, noting that
$
F^{\eps}_{q}(x)=\eps^qF_{q}
(x/\eps)
$ with $F_m(x):=(1+|x|^2)^{\frac{m}{2}}$ for $m\in\mR$, by \eqref{op:levy1}, we have
$$
\mathcal{L}^\alpha_{\bm{\sigma}_s} F^{\eps}_{q}
(x)=\eps^{q-\alpha}\mathcal{L}^\alpha_{\bm{\sigma}_s} F_{q}
(x/\eps).
$$
By \cite[Lemma A.1]{Chen2024wd} and Assumption \ref{hyp:A} \hyperref[A2]{$(A2)$},  as $\lim_{r\to\infty}\pi^r/\Gamma(r)=0$, we have
\begin{align*}
\mathcal{L}^\alpha_{\bm{\sigma}_s} F^{\eps}_{q}
(Z_{s})\lesssim_{c_2,c_3}\frac{2\pi^{\frac{d}{2}}}{\Gamma(\frac{d}{2})} \left(\frac{q}{2-\alpha}+\frac{2^{1+\alpha} }{\alpha-q}\right) \eps^{q-\alpha}|Z_s|^\alpha F_{q-\alpha}(Z_s/\eps)\lesssim_{c_2,c_3,q}(2-\alpha)^{-1}F^{\eps}_{q}(Z_s).
\end{align*}
Then we have
\begin{align*}
&H_t=\mathbb{E}\Big[S^{-q}_{\bm{\tau}}\mathbf{1}_{\{t>\bm{\tau}\}}\sum_{s\in(\bm{\tau},t]}
\big[\delta_{\nabla_{Z_{s-}}
\sigma(X_{s-})\Delta W_{S_s}}f
(Z_{s-})-\big \langle \nabla f(Z_{s-})
,\nabla_{Z_{s-}}
\sigma(X_{s-})\Delta W_{S_s}\big \rangle \big]\Big]\\
=&\mathbb{E}\Big[S^{-q}_{\bm{\tau}}\mathbf{1}_{\{t>\bm{\tau}\}}\int^t_{\bm{\tau}}\mathcal{L}^\alpha_{\bm{\sigma}_s} F^{\eps}_{q}
(Z_{s})\dif s\Big]\lesssim \int^t_0\mathbb{E}\Big[S^{-q}_{\bm{\tau}\wedge s}\mathcal{L}^\alpha_{\bm{\sigma}_s} F^{\eps}_{q}
(Z_{s})\Big]\dif s\lesssim_{q,c_2,c_3} (2-\alpha)^{-1}\int^t_0G_s
\dif s.
\end{align*}

By combining the above inequalities and setting $\eps=t^{\frac{2(1-q)}{q\alpha}}$, we have
\begin{align*}
G_t\lesssim_{q,c_2,c_3}t^{\frac{2(1-q)}{\alpha}}+ (2-\alpha)^{-1}\int^t_0G_s
\dif s,
\end{align*}
which implies by Gronwall's inequality, there is $C=C(q,c_2,c_3)>0$ such that  for any $t\in(0,1]$
$$
\mathbb{E}\Big[
\ell^{-q}_{\bm{\tau}^{\ell}_t}\mathbb{E}
\big[\big|\mathbf{Z}^{v,v_1}_{\ell,t}
\big|^{q}\big]\big|_{\ell_\cdot=S_\cdot}\Big]
\leq \mathbb{E}\big[
\ell^{-q}_{\bm{\tau}^{\ell}_t}\mE[F^{\eps}_{q}(\mathbf{Z}^{v,v_1}_{\ell,t})]|_{\ell_\cdot=S_\cdot}\big]\big|_{\eps=t^{\scriptscriptstyle\frac{2(1-q)}{q\alpha}}}=G_t\lesssim_{C} t^{\frac{2(1-q)}{\alpha}}\e^{C (2-\alpha)^{-1}}.
$$
Then we get \eqref{EA}. Combining step (i) and step (ii), we complete the proof.
\end{proof}

\begin{lemma}  \label{l:MomXGraX}
Let Assumption \ref{hyp:A} hold. For any $x\in\mR^d$, SDE \eqref{sde:L} has a unique solution $\{X^x_{t},t\geq0\}$. For any $q\in [1,\alpha)$, there exists a constant $C=C(c_2,c_3,q)>0$ such that for any $t\in(0,1]$ and $v\in\partial\mathbf{B}$,
\begin{align*}
 \mathbb{E}F_q(\nabla_v X^x_{t})\lesssim_{C}(2-\alpha)^{-1}
\end{align*}
with $F_q(x)=(1+|x|^2)^{\frac{q}{2}}$.
\end{lemma}

\begin{proof}[Proof of Lemma \ref{l:MomXGraX}]
Due to $\alpha$-stable process $\{L^{\scriptscriptstyle\!(\alpha)}_t,t\geq0\}$ has the L\'evy-It\^o decomposition (see  \cite[Theorem 2.4.16]{Applebaum2009Levy}) as
\begin{align*}
L^{\scriptscriptstyle\!(\alpha)}_t=\int_{0<|z|\leq 1}z\widetilde{N}(t,\dif z)+\int_{|z|>1}zN(t,\dif z),%\label{LI:d}
\end{align*}
where $N(t,\dif z)$ is the Poisson random measure and ${\widetilde N}(t,\dif z):=N(t,\dif z)-c_{d,\alpha}t|z|^{-d-\alpha}\dif z$
is the compensated Poisson random measure, where $c_{d,\alpha}$ is the same constant as in \eqref{op:levy}.
 Then we can rewrite SDE \eqref{sde:L} as
\begin{align}
\label{sde:r1}
\dif X^x_t=b(X^x_t)\dif t+\int_{0<|z|\leq 1} \sigma(X^x_{t-})z\widetilde{N}(\dif t,\dif z)+\int_{|z|>1}\sigma(X^x_{t-})zN(\dif t,\dif z)
\end{align}
with $X^x_0=x$.
For the estimation of $\nabla_vX^x_t$, by the chain rule to SDE \eqref{sde:r1}, it easy to get $\nabla_vX^x_t$ satisfies
\begin{align}
\label{1sde:r1}
\dif \nabla_vX^x_t=&\nabla_{\nabla_vX^x_t}b(X^x_t)\dif t+\int_{0<|z|\leq 1}\nabla_{\nabla_vX^x_{t-}}\sigma(X^x_{t-})z\widetilde{N}(\dif t,\dif z)\nonumber\\
&+\int_{|z|>1}\nabla_{\nabla_vX^x_{t-}}\sigma(X^x_{t-})zN(\dif t,\dif z)
\end{align}
with $\nabla_vX^x_0=v$. By It\^o's formula (see  \cite[Theorem 4.4.7]{Applebaum2009Levy}), denoting by
$F_q(x):=(1+|x|^2)^{\frac{q}{2}}$, we have
\begin{align*}
\frac{\dif}{\dif t} \mathbb{E}F_q(\nabla_v X^x_t)&=\mathbb{E}\big[\langle \nabla F_q(\nabla_vX^x_t),\nabla_{\nabla_vX^x_t}b(X^x_t)\rangle\big]+\mathbb{E}\big[\mathcal{L}^\alpha_{\tilde{\sigma}_t} F_q(\nabla_v X^x_t)\big]
:=\mI_1+\mI_2,
\end{align*}
where $\tilde{\sigma}_t=\nabla_{\nabla_vX^x_{t}}\sigma(X^x_{t})$.
By \eqref{in:A2} and \cite[Lemma A.1]{Chen2024wd}, we have
\begin{align*}
\mI_1\lesssim_{c_3} \mathbb{E}[F_q(\nabla_vX^x_t)],\quad
\mI_2\leq \frac{2\pi^{\frac{d}{2}}}{\Gamma(\frac{d}{2})} \left(\frac{q}{2-\alpha}+\frac{2^{1+\alpha} }{\alpha-q}\right) c_3^{\alpha}
\lesssim_{c_3}h_{\alpha,q},
\end{align*}
where $h_{\alpha,q}:=(2-\alpha)^{-1}+(\alpha-q)^{-1}$.
Then there is a constant $C=C(c_2,c_3)>0$ such that for any $t>0$,
\begin{align*}
\frac{\dif}{\dif t} \mathbb{E}F_q(\nabla_v X^x_{t})\leq C\mathbb{E}F_q(\nabla_v X^x_{t})+Ch_{\alpha,q},
\end{align*}
which implies that
\begin{align*}
 \mathbb{E}F_q(\nabla_v X^x_{t})\lesssim_C h_{\alpha,q}F_q(v)\e^{C t}.
\end{align*}
Then we complete the proof.
\end{proof}

Now, we give the proof of Lemma \ref{lemma:g} $(i)$.

\begin{proof}[Proof of Lemma  \ref{lemma:g} (i)]
It is enough to consider $g\in \mathcal{C}^2(\R^d)$ with $\|\nabla g\|_\infty\leq 1$. When $g\in\mathrm{Lip}(1)$, we consider the convolution  $g_\eps:=\rho_\eps*g$, where $\rho_\eps$ is the density function of standard Brownian motion $(B_\eps)_{\eps>0}$.  For example, if we have proven for any $\eps>0$,
\begin{align*}
\|\nabla P^\alpha_t g_\eps\|_{\infty} \leq C.
\end{align*}
Noting that
$$
\lim_{\eps\downarrow 0}\|g_\eps-g\|_\infty=0,\quad \|\nabla g_\eps\|_\infty\leq \|\nabla g\|_\infty,
$$
then we have  for any $v\in\partial\mathbf{B}$
$$
|\delta_v P^\alpha_t g(x)|=
\lim_{\eps\downarrow0}|\delta_v P^\alpha_t g_\eps(x)|
\leq\lim_{\eps\downarrow0}\int^1_0|\nabla_v P^\alpha_t g_\eps(x+rv)|\dif r\leq C,
$$
which directly implies that
\begin{align*}
\|\nabla P^\alpha_t g\|_{\infty} \leq C.
\end{align*}
Here and below, in this proof, we just assume $g\in \mathcal{C}^2(\R^d)$ with $\|\nabla g\|_\infty\leq 1$.

We fix $v\in\partial\mathbf{B}$ and $t\in(0,1]$. Here we also recall $\alpha\in[\alpha_0,\vartheta_0]$ with $1<\alpha_0<\vartheta_0<2$. On the one hand, we note that for any $x\in\mR^d$,
\begin{align*}
|\nabla_v P^\alpha_t g(x)|=|\mE[\nabla g(X^x_t)\nabla_{v}X^x_t]|\leq
\mE[|\nabla_{v}X^x_t|].
\end{align*}
Since $\alpha\in[\alpha_0,\vartheta_0]$, then by Lemma \ref{l:MomXGraX}, there exists a constant $C(c_2,c_3,\vartheta_0)>0$ such that  for any $x\in\mR^d$,
\begin{align*}
\|\nabla P^\alpha_t g\|_\infty=\sup_{x\in\mR^d, v\in \partial\mathbf{B}}|\nabla_v P^\alpha_t g(x)|
\leq  \sup_{x\in\mR^d,v\in \partial\mathbf{B}}\mE[|\nabla_{v}X^x_t|]\leq C.
\end{align*}
On the other hand, we note that for any $x\in\mathbb{R}^d$,
\begin{align}
\|\nabla^2 P^{\alpha}_tg(x)\|=\sup_{v,v_1\in \partial\mathbf{B}}\left|\nabla_{v_1}
\nabla_{v}P^{\alpha}_t g(x)\right|
=\sup_{v,v_1\in \partial\mathbf{B}}\big|\mathbb{E}
\big[\nabla_{v_1}\nabla_{v}
\mathbb{E}[g(X^{\ell,x}_t)]|_{\ell=S}\big]
\big|.\label{est:grad2:op}
\end{align}
By Proposition \ref{formermain}, we have
\begin{align*}
\nabla_{v_1}\nabla_{v}P^\ell_tg(x)
 \overset{\eqref{df:2th}}= \ell_{\bm{\tau}^{\ell}_t}^{-1}\mathbb{E}\big[ \nabla g(X^{\ell,x}_t) \nabla_{v_1}X^{\ell,x}_t \mathbf{L}^{v}_{\ell,t} \big] + \ell_{\bm{\tau}^{\ell}_t}^{-1}\mathbb{E}\big[ \nabla g(X^{\ell,x}_t)\mathbf{Z}^{v,v_1}_{\ell,t} \big],
\end{align*}
where $ \mathbf{L}^{v}_{\ell,t}$ and $ \mathbf{Z}^{v,v_1}_{\ell,t}$ are defined in \eqref{mal:dua:rew} and \eqref{fix:infi:lie} respectively. By Lemma \ref{lastmoment} and H\"older's inequality, and taking $q=\alpha_0$ and $p=\frac{\alpha_0}{\alpha_0-1}$, we have that
\begin{align*}
&\big|\mathbb{E}\big[\nabla_{v_1}\nabla_{v}\mathbb{E}[g(X^{\ell,x}_t)]|_{\ell=S}\big]\big|\\
\leq& \big|\mathbb{E}\big[\ell_{\bm{\tau}^{\ell}_t}^{-1}\mathbb{E}\big[ \nabla g(X^{\ell,x}_t) \nabla_{v_1}X^{\ell,x}_t \mathbf{L}^{v}_{\ell,t}\big] |_{\ell=S}\big]
\big|
+\big|\mathbb{E}\big[\ell_{\bm{\tau}^{\ell}_t}^{-1}\mathbb{E}\big[ \nabla g(X^{\ell,x}_t)\mathbf{Z}^{v,v_1}_{\ell,t} \big]|_{\ell=S}\big]
\big|\\
\leq &\mathbb{E}\big[\ell_{\bm{\tau}^{\ell}_t}^{-1}\mathbb{E}[ | \nabla_{v_1}X^{\ell,x}_t \mathbf{L}^{v}_{\ell,t}|] |_{\ell=S}\big]
+\mathbb{E}\big[\ell_{\bm{\tau}^{\ell}_t}^{-1}\mathbb{E}\big[|\mathbf{Z}^{v,v_1}_{\ell,t} |\big]|_{\ell=S}\big]\\
\leq &\big\{\mathbb{E}[|\nabla_{v_1}X^x_t|^q]\big\}^{\frac{1}{q}}\big\{\mathbb{E}\big[\ell_{\bm{\tau}^{\ell}_t}^{-p}\mathbb{E}[ |\mathbf{L}^{v}_{\ell,t}|^p] |_{\ell=S}\big]\big\}^{\frac{1}{p}}
+\mathbb{E}\big[\ell_{\bm{\tau}^{\ell}_t}^{-1}\mathbb{E}\big[|\mathbf{Z}^{v,v_1}_{\ell,t} |\big]|_{\ell=S}\big].
\end{align*}
By Lemma \ref{l:MomXGraX} and Lemma \ref{lastmoment} with $q=1$, for any $\alpha\in[\alpha_0,\vartheta_0]$, there exists a constant $C(c_2,c_3,\alpha_0,\vartheta_0)>0$ such that
\begin{align*}
\big|\mathbb{E}\big[\nabla_{v_1}\nabla_{v}\mathbb{E}[g(X^{\ell,x}_t)]|_{\ell=S}\big]
\big|&\lesssim_C \big\{\mathbb{E}\big[\ell_{\bm{\tau}^{\ell}_t}^{-p}\mathbb{E}[ |\mathbf{L}^{v}_{\ell,t}|^p] |_{\ell=S}\big]\big\}^{\frac{1}{p}}
+\mathbb{E}\big[\ell_{\bm{\tau}^{\ell}_t}^{-1}\mathbb{E}\big[|\mathbf{Z}^{v,v_1}_{\ell,t} |\big]|_{\ell=S}\big]\lesssim_C t^{-\frac{1}{\alpha}}.
\end{align*}
Then by \eqref{est:grad2:op}, we have there exists a constant $C(c_2,c_3,\alpha_0,\vartheta_0)>0$ such that for any $t\in(0,1]$,
$$
\|\nabla^2 P^{\alpha}_tg\|_\infty\lesssim_C t^{-\frac{1}{\alpha}}.
$$
Combining the above calculations, we complete the proof.
\end{proof}

\subsection{Proof of Lemma \ref{lemma:g} (ii)}\label{s:B2}

In this subsection, we present the proof of \eqref{est:grad12}. We recall SDE \eqref{sde:B}
\begin{align}
\dif Y^x_t = b(Y^x_t)\dif t + \sigma(Y^x_t)\dif B_t, \quad  Y^x_0=x. \label{sde:BW}
\end{align}
For any $v\in\mathbb{R}^d$, define $h^v$ as below:
\begin{align}\label{mal:hv}
h^v_t:=\int^t_0\sigma^{-1}(Y^x_r)\nabla_vY^x_r\dif r.
\end{align}
It is well-known that $h^v\in\mathcal{D}(\delta)$ satisfies
\begin{align}\label{mal:dhv}
D_{h^v}Y^x_t=t\nabla_vY^x_t,\quad  \delta(h^v)_t=\int^t_0\sigma^{-1}(Y^x_r)\nabla_vY^x_r\dif B_r.
\end{align}

\begin{lemma}\label{le:na:vv1}
Let Assumption \ref{hyp:A} hold and $p\geq 1$.
Then there is $C=C(p,c_3)>0$  such that  for any $t\in(0,1]$,
$$
\mE[|\nabla_{v} Y^x_t|^{2p}]\leq C,\quad \mE[|\nabla_{v_1}\nabla_{v} Y^x_t|^{2p}]+\mE[|\nabla_{v_2}\nabla_{v_1}\nabla_{v} Y^x_t|^{2p}]\leq C t^{2p}, \quad \forall v,v_1,v_2\in\partial\mathbf{B}.
$$
\end{lemma}

\begin{proof}[Proof of Lemma \ref{le:na:vv1}]
We consider the following SDE:
\begin{align*}
\dif Z_t=\nabla_{Z_t} b(Y^x_t)\dif t+\nabla_{Z_t} \sigma(Y^x_t)\dif B_t+H_t\dif t,\quad Z_0=\xi,
\end{align*}
where $H_t$ is a adapted process satisfying $R_p:=\sup_{t\in(0,1]}\mE[|H_t|^{2p}]<\infty$ for any $p\geq 1$. By BDG's inequality and Jensen's inequality, we have for any $t\in(0,1]$,
\begin{align*}
\mE[|Z_t|^{2p}]\lesssim_{p,c_3} \mE[|\xi|^{2p}]+t^{2p}R_p+\int^t_0\mE[|Z_s|^{2p}]\dif s.
\end{align*}
By Gronwall's inequality and $t\in (0,1]$, there is $C=C(p,c_3)>0$  such that
\begin{align}\label{in:c1}
\mE[|Z_t|^{2p}]\lesssim_C \mE[|\xi|^{2p}]+t^{2p}R_p.
\end{align}
By \eqref{sde:BW}, we note that
\begin{align*}
\dif \nabla_{v} Y^x_t=\nabla_{\nabla_{v} Y^x_t} b(Y^x_t)\dif t+\nabla_{\nabla_{v} Y^x_t} \sigma(Y^x_t)\dif B_t,\quad \nabla_{v} Y^x_0=v.
\end{align*}
By \eqref{in:c1}, we directly get there is $C=C(p,c_3)>0$  such that  for any $t\in(0,1]$,
$$
\mE[|\nabla_{v} Y^x_t|^{2p}]\leq C.
$$
By inductive argument, we further get  there is $C=C(p,c_3)>0$  such that  for any $t\in(0,1]$,
$$
\mE[|\nabla_{v_1}\nabla_{v} Y^x_t|^{2p}]+\mE[|\nabla_{v_2}\nabla_{v_1}\nabla_{v} Y^x_t|^{2p}]\leq C t^{2p-1}.
$$
Then we complete the proof.
\end{proof}

Now, we give the proof of Lemma \ref{lemma:g} $(ii)$. In this proof, for simplicity, we introduce the following conventions:
$$
x\!\cdot\! y:=\langle x,y\rangle,\quad  r\!\cdot\! y:=ry,\quad \forall x,y\in\mR^d, r\in\mR.
$$

\begin{proof}[Proof of Lemma \ref{lemma:g} (ii)]
(i)
For any $v_i\in \mR^d$ with $i=0,1,2$, we define $h_i:=h^{v_i}$ as \eqref{mal:hv}. Specifically, we set $v:=v_0$ and $h:=h_0$.  Without loss of generality, we assume $Y^x_t$ has $\cC^3$-stochastic flow and  $f\in\cC^3(\R^d)$ with $\|\nabla f\|_\infty\leq 1$. Then introduce the following notations for latter use:
$$
L^{\scriptscriptstyle\!i}_{k}:=[\nabla_{v_i},D_{h_k}],\quad
L^{\scriptscriptstyle\!i,j}_{k}:=[\nabla_{v_i}\nabla_{v_j},D_{h_k}],\quad
M^{\scriptscriptstyle i}_{t}:=t\nabla_{v_i}-D_{h_i},\quad \forall i,j,k\in\{0,1,2\},
$$
where $[A,B]:=AB-BA$. By \eqref{mal:div} and \eqref{mal:dhv}, we have for any $i\in\{0,1,2\}$
and stochastic process $Z^x_t$,
\begin{align*}
&\nabla_{v_i}\mathbb{E}[f(Y^{x}_t)Z_t^x]=\mathbb{E}[\nabla f(Y^{x}_t)\!\cdot\! \nabla_{v_i}Y^{x}_t\!\cdot\!Z_t^x]+\mathbb{E}[f(Y^{x}_t)\!\cdot\!\nabla_{v_i}Z_t^x]\\
=&t^{-1}\mathbb{E}[\nabla f(Y^{x}_t)\!\cdot\! D_{h_i}Y^{x}_t\!\cdot\!Z_t^x]+\mathbb{E}[f(Y^{x}_t)\!\cdot\!\nabla_{v_i}Z_t^x]\\
=&t^{-1}\mathbb{E}[D_{h_i}(f(Y^{x}_t)Z_t^x)]-t^{-1}\mathbb{E}[f(Y^{x}_t)\!\cdot\!D_{h_i}Z_t^x]+\mathbb{E}[f(Y^{x}_t)\!\cdot\!\nabla_{v_i}Z_t^x]\\
=&t^{-1}\mathbb{E}[f(Y^{x}_t)\!\cdot\!Z_t^x\!\cdot\!\delta(h_i)_t]+t^{-1}\mathbb{E}[f(Y^{x}_t)\!\cdot\!M^i_tZ_t^x].
\end{align*}
Then we have
\begin{align}
\nabla_{v_1}\nabla_{v}Q_t f(x)
=t^{-1}\mathbb{E}[\nabla f(Y^{x}_t)\!\cdot\! M^1_t\nabla_{v}Y^{x}_t]
+t^{-1}\mathbb{E}[\nabla f(Y^{x}_t) \!\cdot\! \nabla_{v_1}Y^{x}_t\!\cdot\! \delta(h)_t], \label{key:eq12}
\end{align}
and
\begin{align}
\nabla_{v_2}\nabla_{v_1}\nabla_{v}Q_t f(x)
=&t^{-1}\nabla_{v_2}\mathbb{E}[\nabla f(Y^{x}_t)\!\cdot\!M^1_t\nabla_{v}Y^{x}_t]
+t^{-1}\nabla_{v_2}\mathbb{E}[\nabla f(Y^{x}_t) \!\cdot\!\nabla_{v_1}Y^{x}_t\!\cdot\!\delta(h)_t]\no\\
=&t^{-2}\mathbb{E}[\nabla f(Y^{x}_t)\!\cdot\!M^1_t\nabla_{v}Y^{x}_t\!\cdot\!\delta(h_2)_t]
+t^{-2}\mathbb{E}[\nabla f(Y^{x}_t) \!\cdot\! \nabla_{v_1}Y^{x}_t\!\cdot\!\delta(h)_t\!\cdot\!\delta(h_2)_t]\no\\
&+t^{-2}\mathbb{E}[\nabla f(Y^{x}_t)\!\cdot\! M^2_tM^1_t\nabla_{v}Y^{x}_t]
+t^{-2}\mathbb{E}[\nabla f(Y^{x}_t) \!\cdot\! M^2_t(\nabla_{v_1}Y^{x}_t\!\cdot\!\delta(h)_t)].\label{key:eq123}
\end{align}

(ii) By \eqref{sde:BW} and \eqref{mal:div}, we have
\begin{align*}
&\dif D_{h} Y^x_t=\nabla_{D_{h} Y^x_t} b(Y^x_t)\dif t+\nabla_{D_{h} Y^x_t} \sigma(Y^x_t)\dif B_t+\nabla_{v}Y^x_t\dif t, \quad D_{h} Y^x_0 = 0, \\
&\qquad \dif \nabla_{v_1} Y^x_t=\nabla_{\nabla_{v_1} Y^x_t} b(Y^x_t)\dif t+\nabla_{\nabla_{v_1} Y^x_t}\sigma(Y^x_t)\dif B_t, \quad \nabla_{v_1} Y^x_0=v_1.
\end{align*}
Since $Y^x_t$ has $\cC^3$-stochastic flow, $D_{h}Y^x_t$ and $\nabla_{v_1} Y^x_t$ have $\cC^2$-stochastic flow. Then we have $\nabla_{v_1}D_{h} Y^x_0=0$ and
\begin{align*}
\dif \nabla_{v_1}D_{h} Y^x_t=&\nabla_{\nabla_{v_1}D_{h} Y^x_t} b(Y^x_t)\dif t+\nabla_{\nabla_{v_1}D_{h} Y^x_t} \sigma(Y^x_t)\dif B_t+\nabla_{v}\nabla_{v_1}Y^x_t\dif t\nonumber\\
&+\nabla_{\nabla_{v_1}Y^x_t}\nabla_{D_{h} Y^x_t} b(Y^x_t)\dif t+\nabla_{\nabla_{v_1}Y^x_t}\nabla_{D_{h} Y^x_t} \sigma(Y^x_t)\dif B_t.
\end{align*}
Similarly $D_{h}\nabla_{v_1} Y^x_0=0$ and
\begin{align*}
\dif D_{h}\nabla_{v_1} Y^x_t=&\nabla_{D_{h}\nabla_{v_1} Y^x_t} b(Y^x_t)\dif t+\nabla_{D_{h}\nabla_{v_1} Y^x_t}\sigma(Y^x_t)\dif B_t+\nabla_{D_{h} Y^x_t}\nabla_{\nabla_{v_1}Y^x_t} b(Y^x_t)\dif t\\
&+\nabla_{D_{h} Y^x_t}\nabla_{\nabla_{v_1}Y^x_t} \sigma(Y^x_t)\dif B_t+\nabla_{\nabla_{v_1} Y^x_t}\sigma(Y^x_t)\sigma^{-1}(Y^x_t)\nabla_{v}Y^x_t\dif t.
\end{align*}
Combining the above equations, noting that $M^1_t\nabla_{v}Y^{x}_t=\nabla_{v_1}D_{h} Y^x_t-D_{h}\nabla_{v_1} Y^x_t=L^1_0Y^x_t$, we have
\begin{align}
\dif M^1_t\nabla_{v}Y^{x}_t=&\nabla_{M^1_t\nabla_{v}Y^{x}_t} b(Y^x_t)\dif t+\nabla_{M^1_t\nabla_{v}Y^{x}_t} \sigma(Y^x_t)\dif B_t+\nabla_{v}\nabla_{v_1}Y^x_t\dif t\nonumber\\
&-\nabla_{\nabla_{v_1} Y^x_t}\sigma(Y^x_t)\sigma^{-1}(Y^x_t)\nabla_{v}Y^x_t\dif t.\label{vD-Dv}
\end{align}
As the procedure of Lemma \ref{le:na:vv1}, by \eqref{in:c1}, for any $p\in[1,\infty)$,
there is $C=C(p,c_2,c_3)>0$  such that  for any $t\in(0,1]$,
\begin{align}\label{in:M1}
\mE[|L^1_0Y^x_t|^{2p}]=\mE[|M^1_t\nabla_{v}Y^{x}_t|^{2p}]\leq C t^{2p}.
\end{align}
On the other hand, by Lemma \ref{le:na:vv1} and \eqref{mal:dhv}, using BDG's inequality, we have for any $p\geq 1$ and $t\in(0,1]$,
\begin{align}
\mE[|\delta(h)_t|^{2p}]=\mE\Big[\Big|\int^t_0\sigma^{-1}(Y^x_r)\nabla_vY^x_r\dif B_r\Big|^{2p}\Big]\lesssim_p \mE\Big[\Big|\int^t_0|\sigma^{-1}(Y^x_r)\nabla_vY^x_r|^2\dif r\Big|^{p}\Big]\lesssim_{p,c_2,c_3}t^p.\label{in:M2}
\end{align}
By \eqref{key:eq12} and Young's  inequality, we have
\begin{align*}
|\nabla_{v_1}\nabla_{v}Q_t f(x)|
\leq t^{-1}(\mathbb{E}[|M^1_t\nabla_{v}Y^{x}_t|^2])^{\frac{1}{2}}
+t^{-1}(\mathbb{E}[|\nabla_{v_1}Y^{x}_t|^2])^{1/2}(\mathbb{E}
[|\delta(h)_t|^2])^{\frac{1}{2}}\lesssim_{c_2,c_3} t^{-\frac{1}{2}}.
\end{align*}

(iii) By the chain rule, we have
\begin{align}
&\dif (\nabla_{v_2}M^1_t\nabla_{v}Y^{x}_t)=\nabla_{\nabla_{v_2}M^1_t\nabla_{v}Y^{x}_t} b(Y^x_t)\dif t+\nabla_{\nabla_{v_2}M^1_t\nabla_{v}Y^{x}_t} \sigma(Y^x_t)\dif B_t+\nabla_{v_2}\nabla_{v}\nabla_{v_1}Y^x_t\dif t\no\\
&+\nabla_{\nabla_{v_2}Y^x_t}\nabla_{M^1_t\nabla_{v}Y^{x}_t} b(Y^x_t)\dif t+\nabla_{\nabla_{v_2}Y^x_t}\nabla_{M^1_t\nabla_{v}Y^{x}_t} \sigma(Y^x_t)\dif B_t-\nabla_{\nabla_{v_2}\nabla_{v_1} Y^x_t}\sigma(Y^x_t)\sigma^{-1}(Y^x_t)\nabla_{v}Y^x_t\dif t\no\\
&-\nabla_{\nabla_{v_1} Y^x_t}\sigma(Y^x_t)\nabla_{\nabla_{v_2}Y^x_t}\sigma^{-1}(Y^x_t)\nabla_{v}Y^x_t\dif t-\nabla_{\nabla_{v_1} Y^x_t}\sigma(Y^x_t)\sigma^{-1}(Y^x_t)\nabla_{v_2}\nabla_{v}Y^x_t\dif t.\label{vD-Dv1}
\end{align}
Then by It\^o's formula,  we have
\begin{align}
\dif (t\nabla_{v_2}M^1_t\nabla_{v}Y^{x}_t)&=\nabla_{t\nabla_{v_2}M^1_t\nabla_{v}Y^{x}_t} b(Y^x_t)\dif t+\nabla_{t\nabla_{v_2}M^1_t\nabla_{v}Y^{x}_t} \sigma(Y^x_t)\dif B_t+\nabla_{v}\nabla_{v_1}D_{h_2}Y^x_t\dif t\no\\
&+\nabla_{D_{h_2}Y^x_t}\nabla_{M^1_t\nabla_{v}Y^{x}_t} b(Y^x_t)\dif t+\nabla_{D_{h_2}Y^x_t}\nabla_{M^1_t\nabla_{v}Y^{x}_t} \sigma(Y^x_t)\dif B_t\no\\
&-\nabla_{\nabla_{v_1} D_{h_2}Y^x_t}\sigma(Y^x_t)\sigma^{-1}(Y^x_t)\nabla_{v}Y^x_t\dif t-\nabla_{\nabla_{v_1} Y^x_t}\sigma(Y^x_t)\nabla_{D_{h_2}Y^x_t}\sigma^{-1}(Y^x_t)\nabla_{v}Y^x_t\dif t\no\\
&-\nabla_{\nabla_{v_1} Y^x_t}\sigma(Y^x_t)\sigma^{-1}(Y^x_t)\nabla_{v}D_{h_2}Y^x_t\dif t
+\nabla_{v_2}M^1_t\nabla_{v}Y^{x}_t\dif t.\label{vD-Dv1}
\end{align}
Similarly, we also get
\begin{align}
&\dif(  D_{h_2}M^1_t\nabla_{v}Y^{x}_t)
=\nabla_{D_{h_2}M^1_t\nabla_{v}Y^{x}_t} b(Y^x_t)\dif t+\nabla_{D_{h_2}M^1_t\nabla_{v}Y^{x}_t} \sigma(Y^x_t)\dif B_t+D_{h_2}\nabla_{v}\nabla_{v_1}Y^x_t\dif t\no\\
&+\nabla_{D_{h_2}Y^x_t}\nabla_{M^1_t\nabla_{v}Y^{x}_t} b(Y^x_t)\dif t+\nabla_{D_{h_2}Y^x_t}\nabla_{M^1_t\nabla_{v}Y^{x}_t} \sigma(Y^x_t)\dif B_t-\nabla_{D_{h_2}\nabla_{v_1} Y^x_t}\sigma(Y^x_t)\sigma^{-1}(Y^x_t)\nabla_{v}Y^x_t\dif t\no\\
&-\nabla_{\nabla_{v_1} Y^x_t}\sigma(Y^x_t)\nabla_{D_{h_2}Y^x_t}\sigma^{-1}(Y^x_t)\nabla_{v}Y^x_t\dif t-\nabla_{\nabla_{v_1} Y^x_t}\sigma(Y^x_t)\sigma^{-1}(Y^x_t)D_{h_2}\nabla_{v}Y^x_t\dif t\no\\
&+\nabla_{M^1_t\nabla_{v}Y^{x}_t}\sigma(Y^x_t)\sigma^{-1}(Y^x_t)\nabla_{v_2}Y^x_t\dif t.\label{vD-Dv2}
\end{align}
By \eqref{vD-Dv1} and \eqref{vD-Dv2}, we have
\begin{align*}
&\dif (M^2_tM^1_t\nabla_{v}Y^{x}_t)=\dif[t\nabla_{v_2}M^1_t\nabla_{v}Y^{x}_t-D_{h_2}M^1_t\nabla_{v}Y^{x}_t]\\
=&\nabla_{M^2_tM^1_t\nabla_{v}Y^{x}_t} b(Y^x_t)\dif t+\nabla_{M^2_tM^1_t\nabla_{v}Y^{x}_t} \sigma(Y^x_t)\dif B_t+L^{\scriptscriptstyle\!0,1}_{2}Y^x_t\dif t-\nabla_{L^{\scriptscriptstyle\!1}_{2}Y^x_t}\sigma(Y^x_t)\sigma^{-1}(Y^x_t)\nabla_{v}Y^x_t\dif t\no\\
&-\nabla_{\nabla_{v_1} Y^x_t}\sigma(Y^x_t)\sigma^{-1}(Y^x_t)L^{\scriptscriptstyle\!0}_{2}Y^x_t\dif t+\nabla_{v_2}M^1_t\nabla_{v}Y^{x}_t\dif t-\nabla_{M^1_t\nabla_{v}Y^{x}_t} \sigma(Y^x_t)\sigma^{-1}(Y^x_t)\nabla_{v_2}Y^x_t\dif t.
\end{align*}
As the procedure of Lemma \ref{le:na:vv1}, by \eqref{in:c1}, for any $p\in[1,\infty)$,
there is $C=C(p,c_2,c_3)>0$  such that  for any $t\in(0,1]$,
\begin{align}\label{in:M}
\mE[|M^2_tM^1_t\nabla_{v}Y^{x}_t|^{2p}]\leq C t^{2p}.
\end{align}
By It\^o's formula,  we have
\begin{align*}
\dif (\nabla_{v_1} Y^x_t\delta(h)_t)=&\nabla_{v_1} Y^x_t\sigma^{-1}(Y^x_t)\nabla_vY^x_t\dif B_t+\delta(h)_t\nabla_{\nabla_{v_1} Y^x_t} b(Y^x_t)\dif t\\
&+\delta(h)_t\nabla_{\nabla_{v_1} Y^x_t} \sigma(Y^x_t)\dif B_t+\nabla_{\nabla_{v_1} Y^x_t}\sigma(Y^x_t)\sigma^{-1}(Y^x_t)\nabla_vY^x_t\dif t.
\end{align*}
By the same procedure as above, we have
\begin{align*}
\dif {M^2_t(\nabla_{v_1} Y^x_t\delta(h)_t) }
&=L^1_2 Y^x_t\sigma^{-1}(Y^x_t)\nabla_vY^x_t\dif B_t+\nabla_{v_1} Y^x_t\sigma^{-1}(Y^x_t)L^0_2Y^x_t\dif B_t+M^2_t\delta(h)_t\nabla_{\nabla_{v_1} Y^x_t} b(Y^x_t)\dif t\\
&+\delta(h)_t\nabla_{L^1_2 Y^x_t} b(Y^x_t)\dif t+M^2_t\delta(h)_t\nabla_{\nabla_{v_1} Y^x_t} \sigma(Y^x_t)\dif B_t+\delta(h)_t\nabla_{L^1_2 Y^x_t} \sigma(Y^x_t)\dif B_t\\
&+\nabla_{L^1_2 Y^x_t}\sigma(Y^x_t)\sigma^{-1}(Y^x_t)\nabla_vY^x_t\dif t+\nabla_{\nabla_{v_1} Y^x_t}\sigma(Y^x_t)\sigma^{-1}(Y^x_t)L^0_2Y^x_t\dif t
+\nabla_{v_2}(\nabla_{v_1} Y^x_t\delta(h)_t)\dif t\\
&-\nabla_{v_1} Y^x_t\sigma^{-1}(Y^x_t)\nabla_vY^x_t\sigma^{-1}(Y^x_t)\nabla_{v_2}Y^x_t\dif t-\delta(h)_t\nabla_{\nabla_{v_1} Y^x_t} \sigma(Y^x_t)\sigma^{-1}(Y^x_t)\nabla_{v_2}Y^x_t\dif t.
\end{align*}
Then we have
\begin{align}
&\mE[|M^2_t(\nabla_{v_1}Y^x_t\delta(h)_t)|^{2p}]\no\\
&\lesssim_{p,c_2,c_3} t^{p-1}\int^t_0\mE[|L^1_2 Y^x_s|^{2p}|\nabla_vY^x_s|^{2p}+
|L^0_2 Y^x_s|^{2p}|\nabla_{v_1}Y^x_s|^{2p}+|\delta(h)_s|^{2p}|L^1_2 Y^x_s|^{2p}]\dif s\no\\
&+t^{2p-1}\int^t_0\mE[|\nabla_{v} Y^x_s|^{2p}|\nabla_{v_1} Y^x_s|^{2p}|\nabla_{v_2} Y^x_s|^{2p}+
|\delta(h)_s|^{2p}|\nabla_{v_1} Y^x_s|^{2p}|\nabla_{v_2} Y^x_s|^{2p}]\dif s\no\\
&+t^{p-1}\int^t_0\mE[|M^2_s\delta(h)_s|^{2p}|\nabla_{v_1} Y^x_s|^{2p}+|\nabla_{v_2}\nabla_{v_1}Y^x_s|^{2p} |\delta(h)_s|^{2p}+|\nabla_{v_1} Y^x_s|^{2p}|\nabla_{v_2}\delta(h)_s|^{2p}]\dif s.\label{in:M4}
\end{align}
On the other hand, since  $\nabla_v\sigma^{-1}(x)=-\sigma^{-1}(x)\nabla_v\sigma(x)\sigma^{-1}(x)$, we also have
$$
\nabla_{v_2}\delta(h)_t=
-\int^t_0\sigma^{-1}(Y^x_s)\nabla_{\nabla_{v_2}Y^x_s}\sigma(Y^x_s)\sigma^{-1}(Y^x_s)\nabla_{v}Y^x_s\dif B_s+\int^t_0\sigma^{-1}(Y^x_s)\nabla_{v_2}\nabla_{v}Y^x_s\dif B_s
$$
and
\begin{align*}
M^2_t\delta(h)_t= \int^t_0\sigma^{-1}(Y^x_r)L^1_2Y^x_t \dif B_r
+\int^t_0\nabla_{v_2}\delta(h)_s\dif s-\int^t_0\sigma^{-1}(Y^x_s)\nabla_vY^x_s\sigma^{-1}(Y^x_s)\nabla_vY^x_s\dif s.
\end{align*}
Then we directly get
\begin{align}\label{in:M3}
\mE[|\nabla_{v_2}\delta(h)_t|^{2p}]\lesssim_{p,c_2,c_3}  t^{p},\quad
\mE[|M^2_t\delta(h)_t|^{2p}]\lesssim_{p,c_2,c_3}  t^{2p}.
\end{align}
Using Lemma \ref{le:na:vv1} and combining \eqref{in:M1}, \eqref{in:M2}, \eqref{in:M3} and
\eqref{in:M4}, we have
\begin{align*}
\mE[|M^2_t(\nabla_{v_1}Y^x_t\delta(h)_t)|^{2p}]\lesssim_{p,c_2,c_3}  t^{2p}.
\end{align*}
By Young's inequality and \eqref{key:eq123}, we have
\begin{align*}
|\nabla_{v_2}\nabla_{v_1}\nabla_{v}Q_t f(x)|\lesssim_{c_2,c_3} t^{-1}.
\end{align*}
Then we complete the proof.
\end{proof}

% Supplementary Materials
\appendix

\setcounter{equation}{0}

\renewcommand{\theequation}{$\mathbb{A}$.\arabic{equation}}

\section{Proof of Lemma \ref{ergodicx}}\label{supp:sec}

In this section, we prove Lemma~\ref{ergodicx}, which establishes exponential ergodicity for SDEs \eqref{sde:L} and \eqref{sde:B} by the coupling method. Related results exist in \cite{Liang2020gradient,Wang2020exponential}, but they cannot be directly applied to attain the required precision here. Thus, we provide complete proofs to ensure rigorous presentation.

\subsection{Exponential ergodicity for SDE \eqref{sde:L}}
In this subsection, we prove (i) of Lemma~\ref{ergodicx} based on Assumption \ref{hyp:A} and dissipativity condition \hyperref[DC']{$(DC')$}. By Assumption \ref{hyp:A} and \eqref{op:levy}, we first introduce some preparations. 
For $\eta>0$, we define 
\begin{align*}
\nu(\dif z) := \frac{c_{d,\alpha}}{|z|^{d+\alpha}}\dif z,
\quad
\nu_0(\dif z) := {\bf 1}_{ \{ |z| \leq \eta  \} } \frac{c_{d,\alpha}}{|z|^{d+\alpha}}\dif z.
\end{align*}
For any $\kappa>0$, define $(x)_\kappa:=(\kappa\wedge|x|)x/|x|
$ for any $x\in\mathbb{R}^d$.
Moreover, for any continuous and bijective mapping $\Psi:\mathbb{R}^d\to\mathbb{R}^d$ with $\Psi(0) \neq 0$, denote by  
\begin{eqnarray*}
\mu_{\nu_0,\Psi}(A) &=& \int_{A \cap \{|z|>0\} } \nu_0(\dif z) 
\wedge \int_{\Psi(A) \cap \{|z|>0\} } \nu_0(\dif z).
\end{eqnarray*}

\begin{lemma}\label{le:A1}
Let Assumption \ref{hyp:A} hold, and $\alpha_0,\vartheta_0\in(1,2)$ with $\alpha_0<\vartheta_0$ such that $\alpha\in[\alpha_0,\vartheta_0]$. For any $x,y\in\mathbb{R}^d$ with $x\neq y$, denote  $\Psi(z):=\sigma^{-1}(y)[\sigma(x)z+(x-y)_{\kappa}]$ for any $z\in\mathbb{R}^d$.
Then there is $C=C(d,\alpha_0,\vartheta_0,c_2)\geq 1$ such that for any $x,y\in\mathbb{R}^d$ with $x\neq y$,
$$
\int_{\R^d} |z|^m (\mu_{\nu_0,\Psi}+\mu_{\nu_0,\Psi^{-1}})(\dif z)\leq C (|x-y|\wedge\kappa)^{m-\alpha},
\quad \forall m\in[0,1], 
$$
and for any $x,y\in\mathbb{R}^d$ with $0< |x-y|\leq \frac{\kappa\wedge\eta}{4c^2_2}$,
$$
\mu_{\nu_0,\Psi}(\R^d)+\mu_{\nu_0,\Psi^{-1}}(\R^d)\geq C^{-1} |x-y|^{-\alpha}.
$$
\end{lemma}

\begin{proof}
By basic calculations, we have 
$$\mu_{\nu_0,\Psi}(\dif z)/\dif z=\Big({\bf 1}_{ \{ 0< |z| \leq \eta  \} } \frac{c_{d,\alpha}}{|z|^{d+\alpha}}\Big)\wedge \Big({\bf 1}_{\{ 0<|\Psi^{-1}(z)|\leq\eta\}}\frac{c_{d,\alpha} {\rm det}(\sigma^{-1}(x)\sigma(y)) }{|\Psi^{-1}(z)|^{d+\alpha}}\Big)=:q(z).$$
Note that $\Psi^{-1}(z)=\sigma^{-1}(x)[\sigma(y)z-(x-y)_{\kappa}]$.
When $|z|\leq (2c_2)^{-1}(|x-y|\wedge\kappa)$, by Assumption \ref{hyp:A}  \hyperref[A1]{$(A1)$}, we get
\begin{align*}
  |\Psi^{-1}(z)|\geq c^{-1}_2|\sigma(y)z-(x-y)_{\kappa}|
  \geq c^{-1}_2(|x-y|\wedge\kappa-c_2|z|)\geq (2c_2)^{-1}(|x-y|\wedge\kappa).
\end{align*}
Recalling $S_d=\frac{2\pi^{\frac{d}{2}}}{\Gamma(\frac{d}{2})}$ is the surface area of the unit ball in $\mR^d$, then we have for any $m\in[0,1]$
\begin{align*}
&\quad\,\int_{\R^d} |z|^m \mu_{\nu_0,\Psi}(\dif z)
=\int_{\R^d} |z|^m q(z)\dif z 
\lesssim_{c_2,d} c_{d,\alpha}\int_{\mathbb{R}^d}\frac{|z|^m{\bf 1}_{\{0<|z|\leq \eta\} }}{|z|^{d+\alpha}}\wedge
\frac{|z|^m{\bf 1}_{\{ 0<|\Psi^{-1}(z)|\leq\eta\}}}{|\Psi^{-1}(z)|^{d+\alpha}}\dif z \\
&\lesssim_{c_2,d} c_{d,\alpha}\int_{|z|\leq (2c_2)^{-1}(|x-y|\wedge\kappa)}
\frac{|z|^m{\bf 1}_{\{ 0<|\Psi^{-1}(z)|\leq\eta\}}}{|\Psi^{-1}(z)|^{d+\alpha}}\dif z
+c_{d,\alpha}\int_{|z|\leq (2c_2)^{-1}(|x-y|\wedge\kappa)}\frac{|z|^m{\bf 1}_{\{0<|z|\leq \eta\} }}{|z|^{d+\alpha}}\dif z \\
&\lesssim_{c_2,d} c_{d,\alpha}(|x-y|\wedge\kappa)^{m}\int_{(2c_2)^{-1}(|x-y|\wedge\kappa)\leq |z|}
\frac{\dif z}{|z|^{d+\alpha}}+c_{d,\alpha}\int_{|z|> (2c_2)^{-1}(|x-y|\wedge\kappa)}|z|^{m-d-\alpha}\dif z \\
&\lesssim_{c_2,d} (|x-y|\wedge\kappa)^{m} c_{d,\alpha}S_d \int_{
(2c_2)^{-1}(|x-y|\wedge\kappa)}^{\infty} r^{-\alpha-1}\dif r
+c_{d,\alpha}S_d \int_{
(2c_2)^{-1}(|x-y|\wedge\kappa)}^{\infty} r^{m-\alpha-1}\dif r\\
&\lesssim_{c_2,d} \frac{c_{d,\alpha}S_d}{\alpha-m}(|x-y|\wedge\kappa)^{m-\alpha}
\overset{\eqref{CS12}}\lesssim_{c_2,d,\alpha_0} (|x-y|\wedge\kappa)^{m-\alpha}.
\end{align*}
On the other hand, noting that 
$$
|\Psi^{-1}(z)|\leq c_2 |\sigma(y)z-(x-y)_{\kappa}|\leq c^2_2|z|+c_2(\kappa\wedge|x-y|),
$$
then as $0<|x-y|\leq \frac{\kappa\wedge\eta}{4c^2_2}$,  it holds that,
\begin{align*}
q(z) &\gtrsim_{c_2,d}  \left( \frac{c_{d,\alpha}}{|z|^{d+\alpha}} \wedge \frac{c_{d,\alpha}}{|\Psi^{-1}(z)|^{d+\alpha}} \right) {\bf 1}_{ \{ |z|\leq \eta, |\Psi^{-1}(z)|\leq \eta, |x-y|\leq |z|  \} } \\
&\gtrsim_{c_2,d} \frac{c_{d,\alpha}}{|z|^{d+\alpha}} {\bf 1}_{ \{ |z|\leq \eta, |\Psi^{-1}(z)|\leq \eta, |x-y|\leq |z|  \} } \gtrsim_{c_2,d} \frac{c_{d,\alpha}}{|z|^{d+\alpha}} {\bf 1}_{ \{|x-y|\leq |z|\leq \frac{\eta}{2c^2_2}  \} }.
\end{align*}
Then as $0<|x-y|\leq \frac{\kappa\wedge\eta}{4c^2_2}$, we have
\begin{align*}
\mu_{\nu_0,\Psi}(\R^d)&=\int_{\R^d} q(z) \dif z 
\gtrsim_{c_2,d}  c_{d,\alpha}\int_{|x-y|\leq |z|\leq \frac{\eta}{2c^2_2}} \frac{\dif z}{|z|^{d+\alpha}} \\
&\gtrsim_{c_2,d}  \frac{c_{d,\alpha} S_d}{\alpha}(|x-y|^{-\alpha}-(2c^2_2/\eta)^\alpha) \overset{\eqref{op:levy}}\gtrsim_{c_2,d,\vartheta_0,\alpha_0} |x-y|^{-\alpha}.
\end{align*}
By symmetry, the rest of the proof is trival. Then we complete the proof.
\end{proof}

\begin{proof}[Proof of Lemma \ref{ergodicx} (i)]	
Let $\kappa=\frac{\eta}{4c_2^2}$. By Lemma \ref{le:A1}, we have there is $K=C(\bm{\theta},d,\alpha_0,\vartheta_0)>1$ such that 
\begin{align*}
 A_1:=c_2 \mu_{\nu_0,\Psi}(\R^d)(|x-y|\wedge \kappa)+\left(1+\frac{c_2^2}{2}\right)\int_{\R^d} |z| (\mu_{\nu_0,\Psi}+\mu_{\nu_0,\Psi^{-1}})(\dif z)
 \leq K (|x-y|\wedge \kappa)^{1-\alpha}
\end{align*}
and
$$
A_2:=\int_{|z|> 1} |z| \nu(\dif z) + \frac{c_3}{2}\int_{0<|z|\leq 1} |z|^2 \nu(\dif z)\leq K.
$$
By \hyperref[DC']{$(DC')$}, there exist some constants $K_1,K_2$ and $\ell_0$ depending on $\bm{\theta},d,\alpha_0,\vartheta_0$ such that 
\begin{eqnarray*}
\frac{\langle b(x)-b(y),x-y  \rangle }{|x-y|} + (A_1+A_2)\| \sigma(x)-\sigma(y) \|
\leq
\left\{
\begin{array}{lll}
K_1|x-y|^{2-\alpha}, & \text{if \ \  } |x-y| < \ell_0,  \\
-K_2|x-y|, & \text{if \ \ }  |x-y| \geq  \ell_0.
	\end{array}
	\right.
\end{eqnarray*}
Denote by 
$$
J(r):=\inf_{x,y\in\mathbb{R}^d,|x-y|\leq r} \mu_{\nu_0,\Psi}(\R^d).
$$
By Lemma \ref{le:A1}, there is $K_3=C(\bm{\theta},d,\alpha_0,\vartheta_0)>1$ such that  for any $r>0$,
$$
\frac{1}{2r}(r\wedge\kappa)^2J(r\wedge\kappa)\geq K^{-1}_3 r^{-1}(r\wedge\kappa)^{2-\alpha}.
$$
Taking
$\Phi_1(r):=K_1 r^{2-\alpha}$ and $\ \varrho(r):=K^{-1}_3 r^{-1}(r\wedge\kappa)^{2-\alpha}$, we can define
$$
g_1(r):=\int^{r}_{0}\frac{1}{\varrho(s)}\mathrm{d}s,\quad
g_2(r):=\int^{r}_{0}\frac{\Phi_1(s)}{s\varrho(s)}\mathrm{d}s.
$$
It follows from \cite[Theorem 4.4]{Liang2020gradient} that the process $(X_t)_{t\geq 0}$ is exponential ergodicity. More precisely,
we have
\begin{align*}
\cW_{1}\big(\sL_{t}^{\alpha,x}, \sL_{t}^{\alpha,y}\big) \leq C e^{-\lambda t}|x-y|,
\end{align*}
where
$$
C=\frac{1+(2K_2)\wedge g_1(2\ell_0)^{-1}}{(2K_2)\wedge g_1(2\ell_0)^{-1}},\quad \lambda=\frac{(2K_2)\wedge g_1(2\ell_0)^{-1}}{1+\exp\big\{((2K_2)\wedge g_1(2\ell_0)^{-1}) g_1(2\ell_0)+2 g_2(2\ell_0)\big\}}.
$$
It is clear that as $\alpha\in[\alpha_0,\vartheta_0]$, $C,\lambda$ above only depend on
$\bm{\theta},d,\alpha_0,\vartheta_0$. Then we complete the proof.
\end{proof}

\subsection{Exponential ergodicity for SDE \eqref{sde:B}}

We shall use the reflection coupling in \cite[Theorem 2.6]{Wang2020exponential} to show the exponential ergodicity for diffusion $(Y_t)_{t\geq 0}$ with multiplicative noise. A continuous function $\kappa(r)$ is defined as below:
\begin{eqnarray}\label{e:kappa}
\kappa(r)
=
\left\{
\begin{array}{lll}
c_0-c_1r^{-1}, & \text{\ \ if \ \ } 0< r \leq 1,  \\
c_0-c_1r^{-2}, & \text{\ \ if  \ \ }  r>1,
\end{array}
\right.
\end{eqnarray}
satisfying that
\begin{align}
\kappa:(0,\infty) \to \mathbb{R} \quad \text{with}
\quad \liminf_{r\to\infty}\kappa(r)=\kappa_{\infty}>0,
\quad
\int_0^1 (r \kappa(r))^{-} \dif r < \infty. \label{kappa}
\end{align}
Obviously, $\kappa(r)\geq c_0-c_1 r^{-2}, \forall r>0$. We first construct a {\bf concave} function $\Psi(r)$ by the method of Eberle (see \cite[p. 856]{Eberle2016reflection}). Minutely, for any $\sigma_0\in(0, c^{-1}_2)$, we define
\begin{equation*}
\begin{array}{c}
r_0:=\inf\{s\geq 0\mid \kappa(r)\geq 0\ \text{for any}\ r\geq s\}, \\
r_1:=\inf\{s\geq r_0 \mid s(s-r_0)\kappa(r)\geq 8\sigma_0^{2}\ \text{for any}\ r\geq s\},
\end{array}
\end{equation*}
and consider the following functions on $\mathbb{R}_+$
\begin{align*}
\varphi(r)&=\exp\left(-\sigma_0^{-2}\int^r_0  (s\kappa(s))^{-} \dif s\right),\quad  \Phi(r)=\int^r_0\varphi(s)\dif s, \quad
g(r)=1-\frac{\lambda}{4\sigma^2_0 }\int^{r\wedge r_1}_0\Phi(s)/\varphi(s)\dif s,
\end{align*}
where the constant $\lambda$ is defined as
\begin{align}\label{kappa:coe}
\lambda:=2\sigma^2_0 \left(\int^{r_1}_0\Phi(s)/\varphi(s)\dif s\right)^{-1}.
\end{align}
Then referring to \cite[p. 857]{Eberle2016reflection} again, there is a {\bf concave} function
\begin{align}
\Psi(r)=\int^r_0\varphi(s)g(s)\dif s,\quad r\geq 0,\label{f:aux}
\end{align}
with below properties stating that
\begin{align}
&\frac{r}{2} \varphi(r_0)\leq \Psi(r)\leq r,\quad 0\leq\Psi'(r)\leq 1,\qquad\forall r\geq 0,\label{f:pro1}\\
&2\sigma^2_0\Psi''(r)-\frac{r\kappa(r)}{2}\Psi'(r)
\leq -\lambda \Psi(r),\quad \forall r\in\mathbb{R}_+\setminus \{r_1\}\label{f:pro2}.
\end{align}

\begin{proof}[Proof of Lemma \ref{ergodicx} (ii)]

We present the exponential ergodicity of SDE \eqref{sde:B} by reflection coupling. Rewrite SDE \eqref{sde:B} as
\begin{eqnarray*}
\dif Y_t = b(Y_t) \dif t+\tilde{\sigma}(Y_t)\dif \tilde{B}_t+\sigma_0\dif B_t, & t\geq 0,
\end{eqnarray*}
where $\tilde{B}_t$ is an independent $d$-dimensional Brownian motion with $B_t$, and the positive constant $\sigma_0$ satisfies the matrices $\sigma(x) \sigma^{*}(x)-\sigma_0^2 \mathbb{I}, \forall x\in \R^d$ are all positive definite and
\begin{eqnarray*}
\sigma(x) \sigma^{*}(x)=\sigma_0^2 \mathbb{I} + \tilde{\sigma}(x)\tilde{\sigma}^{*}(x), \quad \forall x\in \R^d.
\end{eqnarray*}

We consider the following coupling process $(Y_t,\tilde{Y}_t)$ as
\begin{equation}
\left\{
\begin{array}{ll}
\dif Y_t=b(Y_t)\dif t+\tilde{\sigma}(Y_t)\dif \tilde{B}_t+\sigma_0\dif B_t, & t\geq 0,\\
\dif \tilde{Y}_t=b(\tilde{Y}_t)\dif t+\tilde{\sigma}(\tilde{Y}_t)\dif\tilde{B}_t+\sigma_0(\mathbb{I}-2e_te_t^*)\dif B_t, & t<\tau, \\
\end{array}
\right. \label{sde:cou}
\end{equation}
where $\tau$ is the coupling time, i.e.,
$\tau=\inf\{t\geq 0\mid \tilde{Y}_t=Y_t\}$
and $e_t:=(\tilde{Y}_t-Y_t)/|\tilde{Y}_t-Y_t|.$
For $t\geq \tau$, we assume $|\tilde{Y}_t-Y_t|=0$.
Let $(\tilde{Y}_t, Y_t)$ be a strong solution of SDE \eqref{sde:cou} with $(\tilde{Y}_0, Y_0)=(x,y)$. For any $t\geq 0$, we denote by
$$
Z_t:=\tilde{Y}_t-Y_t,\quad R_t:=|\tilde{Y}_t-Y_t|.
$$
Then by \eqref{sde:cou} and It\^o's formula, we have for $t<\tau$
\begin{align*}
\dif R_t^2=&\big[2\langle Z_t, b(\tilde{Y}_t)-b(Y_t)\rangle+\|\tilde{\sigma}
(\tilde{Y}_t)-\tilde{\sigma}(Y_t)\|
^2_{\mathrm{HS}}+\|2\sigma_0 e_te_t^*\|^2_{\mathrm{HS}}\big]\dif t+2\dif M_t,
\end{align*}
where
$$
\dif M_t=\langle Z_t, 2\sigma_0 e_te_t^*\dif B_t\rangle+\langle Z_t, [\tilde{\sigma}(\tilde{Y}_t)-\tilde{\sigma}(Y_t)]
\dif \tilde{B}_t\rangle.
$$
Since for $t<\tau$, $R_t>0$, $\| e_te_t^*\|^2_{\mathrm{HS}}=1$ and $|Z^*_t e_te_t^*|^2=R_t^2$, by It\^o's formula, we have
\begin{align*}
\dif R_t =(2R_t)^{-1}\dif R^2_t-(2R_t)^{-3}\dif \langle R^2_\cdot,R^2_\cdot \rangle_t =(2R)^{-1}_t H_t \dif t+R^{-1}_t\dif M_t,
\end{align*}
where 
$$
H_t:=2\langle Z_t, b(\tilde{Y}_t)-b(Y_t)\rangle
+\|\tilde{\sigma}(\tilde{Y}_t)-\tilde{\sigma}(Y_t)\|
^2_{\mathrm{HS}}
-R_t^{-2}|Z^*_t(\tilde{\sigma}(\tilde{Y}_t)
-\tilde{\sigma}(Y_t))|^2.
$$
Then taking $\Psi$ as \eqref{f:aux}, we have for $t<\tau$
\begin{align}
\dif \Psi(R_t)=&\Psi'(R_t)R^{-1}_t\big[\langle Z_t, b(\tilde{Y}_t)-b(Y_t)\rangle+2^{-1}\|
\tilde{\sigma}(\tilde{Y}_t)-\tilde{\sigma}(Y_t)\|
^2_{\mathrm{HS}} \big]\dif t\nonumber\\
&-\Psi'(R_t)R_t^{-3}
\big[2^{-1}|Z^*_t(\tilde{\sigma}(\tilde{Y}_t)
-\tilde{\sigma}(Y_t))|^2 \big]\dif t\nonumber\\
&+\Psi''(R_t)R_t^{-2}
\big[2^{-1}|Z^*_t(\tilde{\sigma}(\tilde{Y}_t)
-\tilde{\sigma}(Y_t))|^2
+2\sigma^2_0R^2_t\big]\dif t+\Psi'(R_t)R^{-1}_t\dif M_t\nonumber\\
\leq &\big[\Psi'(R_t)(2R_t)^{-1}H_t
+2\sigma^2_0\Psi''(R_t)\big]\dif t+\Psi'(R_t)R^{-1}_t\dif M_t.
\label{f:ito}
\end{align}
By dissipativity condition \hyperref[DC]{$(DC)$}, the property of $\Psi$ in \eqref{f:pro1}, \eqref{f:pro2} and the fact $\kappa(r)\geq c_0-c_1 r^{-2}, \forall r>0$, we have
\begin{align*}
\Psi'(R_t)(2R_t)^{-1}H_t+2\sigma_0^2 \Psi^{\prime \prime}(R_t)
\leq& \Psi^{\prime}(R_t) \frac{1}{2R_t}(-c_0 R_t^2+c_1)+2\sigma_0^2 \Psi^{\prime \prime}(R_t)  \\
\leq & -\frac{\Psi'(R_t)}{2}R_t\kappa(R_t)+2\sigma_0^2 \Psi^{\prime \prime}(R_t) \\
\leq& -\lambda \Psi(R_t),
\end{align*}
combining this with  \eqref{f:ito}, one knows
\begin{align*}
\dif \Psi(R_t)\leq - \lambda \Psi(R_t)\dif t
+\Psi'(R_t)R^{-1}_t\dif M_t,\quad t<\tau.
\end{align*}

For $n\geq 1$, defining the stopping time
$$
T_n:=\inf\{t>0\mid R_t\in[1/n,n]^c\},
$$
we have $T_n<\tau,\ T_n\uparrow \tau$ and
$$
\dif \e^{\lambda t}\Psi(R_t)=\e^{\lambda t}\dif \Psi(R_t)+\lambda \e^{\lambda t}\Psi(R_t)\dif t
\leq \e^{\lambda t}\Psi'(R_t)R^{-1}_t\dif M_t,
\quad t\leq T_n,
$$
which means that by taking expectation on both side
$$
\mathbb{E}[\e^{\lambda (T_n\wedge t)}\Psi(R_{T_n\wedge t})\mathbf{1}_{\{t<\tau\}}]
\leq \mathbb{E}\Psi(R_0).
$$
Taking $n\to\infty$, by Fatou's Lemma, we have
$$
\mathbb{E}[\e^{\lambda t}\Psi(R_{t})\mathbf{1}_{\{t<\tau\}}]
\leq \liminf_{n\to\infty}\mathbb{E}[\e^{\lambda (T_n\wedge t)}\Psi(R_{T_n\wedge t}) \mathbf{1}_{\{t<\tau\}}]
\leq \mathbb{E}\Psi(R_0).
$$
Then for any $t>0$, we have
\begin{align*}
\mathbb{E}[\e^{\lambda t}\Psi(R_{t})]=\mathbb{E}[\e^{\lambda t}\Psi(R_{t})\mathbf{1}_{\{t<\tau\}}]\leq \mathbb{E}\Psi(R_0).
\end{align*}
By \eqref{f:pro2}, we have
\begin{align*}
\frac{\varphi(r_0)}{2}\mathbb{E}R_t
\leq \mathbb{E}\Psi(R_t)
\leq \e^{-\lambda t}\mathbb{E}\Psi(R_0)
\leq \e^{-\lambda t}\mathbb{E}R_0, %\label{f:E}
\end{align*}
that is,
\begin{align*}
\mathbb{E}|\tilde{Y}_t^x-Y_t^y|
\leq \frac{2}{\varphi(r_0)} \e^{-\lambda t} \mathbb{E}|x-y|, %\label{f:E}
\end{align*}
which implies the desired result by the fact: $\lambda$ and $\varphi(r_0)$ only depend on $c_0,c_1$ and $c_2$. 
\end{proof}

\section{Proof of Lemma \ref{fix:infi}} \label{sec:secJf}

\setcounter{equation}{0}

\renewcommand{\theequation}{$\mathbb{B}$.\arabic{equation}}

Let  $\beta:[0,\infty)\to[0,\infty)$ be an absolutely continuous function with $\beta_0=0$ and locally bounded derivative $\dot{\beta}_t$ and $h_t$ be  defined as \eqref{mal:der}, we set
\begin{align}
	W^\beta_t:=\int^t_0\dot{\beta}_s\dif W_s,\quad  \lambda^\beta_t:=\int^t_0|\dot{\beta}_s|^2\dif s,\quad h^\beta_t:=\int^t_0\dot{\beta}_s\dif h_{s},\quad \forall t\geq 0.\label{W_beta}
\end{align}
To present the proof of Lemma \ref{fix:infi}, we introduce the following approximation lemma (see \cite{Wang2015gradient}).
\begin{lemma}
	\label{le:aux:app}
	Let $\xi^\eps_t$ be an $\sF^{\mW}_{\ell^\eps_t}$-adapted c\'adl\'ag $\mR^d\otimes\mR^d$-valued process
	with $\eps\in[0,1)$, $p_1,p_2\in\{0,1\}$ and $q_1,q_2\in\{0,1\}^d$. We define
	$$
	F^\eps_t:=\sum_{s\leq t}\left\langle
	\xi^\eps_{s-}\left(p_1\Delta W_{\ell^\eps_s}+q_1\right),p_2\Delta W^\beta_{\ell^\eps_s}+q_2\Delta\beta_{\ell^\eps_s}\right\rangle.
	$$
	If for any $T>0$,
	$$
	\sup_{t\in[0,T]}\sup_{\eps\in[0,1)}\mE[\|\xi^\eps_t\|^2]<\infty,
	$$
	and
	$$
	\lim_{s\uparrow t}\mE[\|\xi^\eps_{s}-\xi^\eps_{t-}\|^2]=0,\quad \lim_{\eps\downarrow 0}\sup_{t\in[0,T]}\mE[\|\xi^\eps_{t}-\xi^0_{t}\|^2]=0,
	$$
	then we have
	$$
	\lim_{\eps\downarrow 0}\sup_{t\in[0,T]} \mE[|F^\eps_t-F^0_t|]=0.
	$$
	\iffalse
	Furthermore, if there exists $p>2$ such that
	$$
	\sup_{t\in[0,T]}\sup_{\eps\in[0,1)}\mE[\|\xi^\eps_t\|^p]<\infty,
	$$
	then we have
	$$
	\lim_{\eps\downarrow 0}\sup_{t\in[0,T]} \mE(|F^\eps_t-F^0_t|^2)=0.
	$$
	\fi
\end{lemma}
\begin{proof}
	It is enough to assume
	\begin{align*}
		F^\eps_t=&\int^t_0 A_{s-}^{\eps}\dif \beta_{\ell^\eps_s}+\sum_{s\in(0,t]}\big\langle V_{s-}^{\eps}(1+\Delta\beta_{\ell^\eps_s}),\Delta W^\beta_{\ell^\eps_s}\big\rangle+\sum_{s\in(0,t]}\big\langle M_{s-}^{\eps} \Delta W^\beta_{\ell^\eps_{s}},\Delta W_{\ell^\eps_{s}}\big\rangle\\
		=:& I_1^{\eps}(t)+I_2^{\eps}(t)+I_3^{\eps}(t),
	\end{align*}
	where $A_{t}^{\eps}, V_{t}^{\eps}$ and $M_{t}^{\eps}$ have the same properties as $\xi^\eps_t$. We only prove
	\begin{align}
		\lim_{\eps\downarrow 0}\sup_{t\in[0,T]}\mathbb{E}\big[|I_1^{\eps}(t)-I_1^{0}(t)|^2+|I_2^{\eps}(t)-I_2^{0}(t)|^2\big]=0, \label{kk1}
	\end{align}
	and
	\begin{align}
		\lim_{\eps\downarrow 0}\sup_{t\in[0,T]}\mathbb{E}[|I_3^{\eps}(t)-I_3^{0}(t)|]=0. \label{kk2}
	\end{align}
	For \eqref{kk1}, by \eqref{app:fix:p1}, we have $\ell^0_t=\ell_t$,
	\begin{align*}
		&\mathbb{E}\big[|I_1^{\eps}(t)-I_1^{0}(t)|^2+|I_2^{\eps}(t)-I_2^{0}(t)|^2\big]\\
		\lesssim & \beta_{\ell_t}\int^t_0\mathbb{E}\big[ \big|A_{s-}^{0}-A_{s-}^{\eps}\big|^2\big]\dif \beta_{\ell_s}+(1+\beta_{\ell_t})^2\int^t_0\mathbb{E} \big[\big|V_{s-}^{0}-V_{s-}^{\eps}\big|^2\big]\dif \lambda^{\beta}_{\ell_s}+R^\eps_t,
	\end{align*}
	where
	\begin{align*}
		R^\eps_t
		:=&
		\mathbb{E}\Big[\Big|\sum_{s\in(0,t]}\big\langle V_{s-}^{\eps}(1+\Delta\beta_{\ell^\eps_s}),\Delta W^\beta_{\ell^\eps_s}\big\rangle-\sum_{s\in(0,t]}\big\langle V_{s-}^{\eps}(1+\Delta\beta_{\ell^\eps_s}),\Delta W^\beta_{\ell_s}\big\rangle \Big|^2\Big]\\
		&
		+\Big(\sum_{s\leq t}\Delta \beta_{\ell_s}\mathbf{1}_{\Delta\ell_s<\eps} \Big)^2\sup_{s\leq T;\eps<1}\mathbb{E}[|A_{s}^{\eps}|^2]
		+\sum_{s\leq t}(\Delta \beta_{\ell_s})^2\Delta\lambda^{\beta}_{\ell_s}\mathbf{1}_{\Delta\ell_s<\eps} \sup_{s\leq T;\eps<1}\mathbb{E}[|V_{s}^{\eps}|^2]
		\\
		\lesssim&\mathbb{E}\Big[\Big|\sum_{s\in(0,t]}\big\langle V_{s-}^{\eps}(1+\Delta\beta_{\ell^\eps_s}),\Delta W^\beta_{\ell^\eps_s}\big\rangle-\sum_{s\in(0,t]}\big\langle V_{s-}^{\eps}(1+\Delta\beta_{\ell^\eps_s}),\Delta W^\beta_{\ell_s}\big\rangle \Big|^2\Big]\\
		&+\beta_{\ell_t}(1+\lambda^{\beta}_{\ell_s})\Big(\sum_{s\leq t}\Delta \beta_{\ell_s}\mathbf{1}_{\Delta\ell_s<\eps} \Big)\sup_{s\leq T;\eps<1}\mathbb{E}[|A_{s}^{\eps}|^2+|V_{s}^{\eps}|^2].
	\end{align*}
	Using \cite[Lemma 2.2]{Wang2015gradient} and the dominated convergence theorem, we have
	\begin{align*}
		\lim_{\eps\downarrow0}\sup_{t\in[0,T]}\mathbb{E}\big[|I_1^{\eps}(t)-I_1^{0}(t)|^2+|I_2^{\eps}(t)-I_2^{0}(t)|^2\big]
		\lesssim\lim_{\eps\downarrow0}\sup_{t\in[0,T]} R^\eps_t=0.
	\end{align*}
	For \eqref{kk2}, we rewrite $I_3^{0}(t)-I_3^{\eps}(t)$ as
	\begin{align*}
		I_3^{0}(t)-I_3^{\eps}(t)
		=&\sum_{s\in(0,t]}\big\langle (M_{s-}^{0}-M_{s-}^{\eps})\Delta W^\beta_{\ell_{s}},\Delta W_{\ell_{s}}\big\rangle+\sum_{s\in(0,t]}\big\langle M_{s-}^{\eps}(\Delta W^\beta_{\ell_{s}}-\Delta W^\beta_{\ell^\eps_{s}}),\Delta W_{\ell_{s}}\big\rangle\\
		&+\sum_{s\in(0,t]}\big\langle M_{s-}^{\eps}\Delta W^\beta_{\ell^\eps_{s}},\Delta W_{\ell_{s}}-\Delta W_{\ell^\eps_{s}}\big\rangle=:K^\eps_1+K^\eps_2+K^\eps_3.
	\end{align*}
	By Minkovski's inequality and H\"older's inequality, we have
	\begin{align*}
		\mathbb{E}[|K^\eps_1|]\leq& \sum_{s\in(0,t]}\big(\mathbb{E}[|(M_{s-}^{0}-M_{s-}^{\eps})\Delta W^\beta_{\ell_{s}}|^2]\big)^{\frac{1}{2}}\big(\mathbb{E}[|\Delta W_{\ell_{s}}|^{2}]\big)^{\frac{1}{2}}\\
		\leq&C\sum_{s\in(0,t]}\big(\mathbb{E}[\|M_{s-}^{0}-M_{s-}^{\eps}\|^2]\big)^{\frac{1}{2}}(\Delta\lambda^\beta_{\ell_s})^{\frac{1}{2}}(\Delta\ell_s)^{\frac{1}{2}}\\
		\leq&C\ell_t\sup_{s\in(0,t]}\dot{\beta}_s\sup_{s\in(0,t]}\big(\mathbb{E}[|M_{s-}^{0}-M_{s-}^{\eps}|^2]\big)^{\frac{1}{2}},
	\end{align*}
	and
	\begin{align*}
		\mathbb{E}[|K^\eps_2|]
		\leq &\sum_{s\in(0,t]}\big(\mathbb{E}[|\Delta W^\beta_{\ell_{s}}-\Delta W^\beta_{\ell^\eps_{s}}|^2]\big)^{\frac{1}{2}}\big(\mathbb{E}[\|M_{s-}^{\eps}\|^{2}|\Delta W_{\ell_{s}}|^{2}]\big)^{\frac{1}{2}}\\
		\leq&C(\ell_t)^{\frac{1}{2}}\sup_{s\in(0,t]}\big(\mathbb{E}[\|M_{s-}^{\eps}\|^{2}]\big)^{\frac{1}{2}}\Big( \sum_{s\in(0,t]}\mathbb{E}[|\Delta W^\beta_{\ell_{s}}-\Delta W^\beta_{\ell^\eps_{s}}|^2]\Big)^{\frac{1}{2}}.
	\end{align*}
	By \cite[Lemma 2.2]{Wang2015gradient} and the dominated convergence theorem, we have
	\begin{align*}
		\lim_{\eps\downarrow0}\sup_{t\in[0,T]}\mathbb{E}[|K^\eps_1|]=0,\quad \lim_{\eps\downarrow0}\sup_{t\in[0,T]}\mathbb{E}[|K^\eps_2|]=0.
	\end{align*}
	Similar to $K^\eps_2$, we get
	\begin{align*}
		\lim_{\eps\downarrow0}\sup_{t\in[0,T]}\mathbb{E}[|K^\eps_3|]=0.
	\end{align*}
	Combining the above calculations, we complete the proof.
\end{proof}

We also need the following lemma.
\begin{lemma}
	\label{le:M22F}
	Given two $\sF^{\mW}_{\ell_t}$-adapted c\'adl\'ag processes
	$$
	\xi^{(1)}_{t}:\mR_+\times\Omega\times\mR^d\to\mR^d\otimes\mR^d,\quad \xi^{(2)}_{t}:\mR_+\times\Omega\times\mR^d\to\mR,
	$$
	for $p_1,p_2\in\{0,1\}$ and $q_1,q_2\in\{0,1\}^d$, we define
	$$
	F_t:=\sum_{s\leq t}\left\langle\xi^{(1)}_{s-}H_s,H^\beta_s \right\rangle+\int^t_0\xi^{(2)}_{s}\dif s,
	$$
	where
	$H_s:=p_1\Delta W_{\ell_s}+q_1$ and $ H^\beta_s:=p_2\Delta W^\beta_{\ell_s}+q_2\Delta\beta_{\ell_s}.$
	Then we have
	\begin{align*}
		(D_{h}-\beta_{\ell_t}\nabla_{v})F_t=&\sum_{s\leq t}\left\langle(D_{h}-\beta_{\ell_{s-}}\nabla_{v})\xi^{(1)}_{s-}
		H_s,H^\beta_s\right\rangle+\int^t_0(D_{h}-\beta_{\ell_s}
\nabla_{v})\xi^{(2)}_{s}\dif s+\sum_{s\leq t}\nabla_{v}F_s\Delta\beta_{\ell_s}\nonumber\\
		&-p_1\sum_{s\leq t}\left\langle\xi^{(1)}_{s-}
		\Delta h_{\ell_s},H^\beta_s\right\rangle-p_2\sum_{s\leq t}\left\langle\xi^{(1)}_{s-}
		H_s,\Delta h^\beta_{\ell_s}\right\rangle.
	\end{align*}
\end{lemma}
\begin{proof}
	Noting that for any $s\in(0,t]$,
	$$
	D_{h}\Delta W^\beta_{\ell_{s}}
	=D_{h}\int^{\ell_s}_{\ell_{s-}} \dot{\beta}_r\dif W_{r}=\int^{\ell_s}_{\ell_{s-}} \dot{\beta}_r\dif h_{r}
	=\Delta h^\beta_{\ell_s},\quad D_{h}\Delta W_{\ell_{s}}=\Delta h_{\ell_s},
	$$
	then we have
	\begin{align}
		D_{h}F_t=&\sum_{s\leq t}\left\langle D_{h}\xi^{(1)}_{s-}
		H_s,H^\beta_s\right\rangle
		+p_1\sum_{s\leq t}\left\langle
		\xi^{(1)}_{s-}\Delta h_{\ell_s},H^\beta_s\right\rangle\nonumber\\
		&+p_2\sum_{s\leq t}\left\langle\xi^{(1)}_{s-}
		H_s,\Delta h^\beta_{\ell_s}\right\rangle
		+\int^t_0D_{h}\xi^{(2)}_{s}\dif s.\label{h^2F}
	\end{align}
	Due to
	\begin{align*}
		\nabla_{v}F_t=&\sum_{s\leq t}\left\langle\nabla_{v}\xi^{(1)}_{s-}
		H_s,H^\beta_s\right\rangle+\int^t_0\nabla_{v}\xi^{(2)}_{s}\dif s,
	\end{align*}
	then we have
	\begin{align}
		\beta_{\ell_t}\nabla_{v}F_t=\sum_{s\leq t}\nabla_{v}F_s\Delta\beta_{\ell_s}
		+\int^t_0\beta_{\ell_{s}}\nabla_{v}\xi^{(2)}_{s}\dif s
		+\sum_{s\leq t}\left\langle\beta_{\ell_{s-}}\nabla_{v}\xi^{(1)}_{s-}
		H_s,H^\beta_s\right\rangle,
		\label{v_2F}
	\end{align}
	which is based on the integration by parts for right-continuous functions (see \cite[Lemma 1.39]{He1992semimartingale}).
	Combining \eqref{h^2F} and \eqref{v_2F}, we have the desired result and complete the proof.
\end{proof}

\begin{proof}[Proof of Lemma \ref{fix:infi}]
	{\bf Step1.} Given $v_i\in \mR^d$ with $i=0,1$. Let $v_0=v$ and $|v_i|=1$ with $i=0,1$.
	For $p\geq 1$, we denote by
	$$
	F^p_n(t):=\mathbb{E}\Big[\Big|\Big(\prod^n_{i=0}\nabla_{v_{i}}\Big)X^{\ell^{\eps},x}_t\Big|^{2p}\Big],\quad \forall n=0,1.
	$$
	{ \bf (1.1)} We consider $n=0$. By \eqref{sde:fix_xp}, we have
	\begin{align}
		\dif \nabla_vX^{\ell^\eps,x}_t=\nabla_{\nabla_vX^{\ell^\eps,x}_t}b(X^{\ell^\eps,x}_t)\dif t+\nabla_{\nabla_vX^{\ell^\eps,x}_{t-}}\sigma(X^{\ell^\eps,x}_{t-})\dif W_{\ell^\eps_t},\quad \nabla_vX^{\ell^\eps,x}_0=v.\label{fix:1g}
	\end{align}
	Then by It\^o's formula for purely jump semimartingales (see \cite[Theorem 9.35]{He1992semimartingale}), we have
	\begin{align*}
		|\nabla_{v}X^{\ell^{\eps},x}_t|^{2p}=&1+2p\int^t_0|\nabla_{v}X^{\ell^{\eps},x}_s|^{2p-2}\langle\nabla_{v}X^{\ell^{\eps},x}_s,\nabla_{ \nabla_v X^{\ell^{\eps},x}_{s}}b(X^{\ell^{\eps},x}_{s})\rangle\dif s+M_t\\
		&+\sum_{s\in(0,t]}\big[\delta_{\Delta (\nabla_{v}X^{\ell^{\eps},x}_s)}|\nabla_{v}X^{\ell^{\eps},x}_{s-}|^{2p}-\nabla_{\Delta (\nabla_{v}X^{\ell^{\eps},x}_s)} |\nabla_{v}X^{\ell^{\eps},x}_{s-}|^{2p}\big],
	\end{align*}
	where $M_t$ is a martingale. Taking expectation, by Assumption \ref{hyp:A} \hyperref[A2]{$(A2)$}, we have
	\begin{align}\label{fix:est:ito}
		\mathbb{E}\left[|\nabla_{v}X^{\ell^{\eps},x}_t|^{2p}\right]\leq&1+
		2p c_3\int^t_0\mathbb{E}\left[|\nabla_{v}X^{\ell^{\eps},x}_s|^{2p}\right]\dif s+I_1,
	\end{align}
	where
	$$
	I_1:=\sum_{s\in(0,t]}\mathbb{E}\big[\delta_{\Delta (\nabla_{v}X^{\ell^{\epsilon},x}_s)}|\nabla_{v}X^{\ell^{\eps},x}_{s-}|^{2p}-\nabla_{\Delta (\nabla_{v}X^{\ell^{\epsilon},x}_s)} |\nabla_{v}X^{\ell^{\eps},x}_{s-}|^{2p}\big].
	$$
	Note that
	\begin{align}
		\nabla |x|^{2p}=2p|x|^{2p-2}x, \quad \|\nabla^2 |x|^{2p}\|\leq 2p(2p-1)|x|^{2p-2}.\label{auxi:g}
	\end{align}
	Then by Newton-Leibniz's formula, we have for any $s\in(0,t]$
	\begin{align*}
		\delta_{u}|x|^{2p}-\nabla_{u}|x|^{2p}
		=\int^1_0\langle \nabla^2 |x+ru|^{2p},uu^*\rangle_{\small{\mathrm{HS}}}\dif r
		\lesssim_p &|x|^{2p-2}|u|^2+|u|^{2p}.
	\end{align*}
	For $I_1$, as
	\begin{align*}
		\Delta (\nabla_{v}X^{\ell^{\eps},x}_s)=\nabla_{ \nabla_v X^{\ell^{\eps},x}_{s-}}\sigma(X^{\ell^{\eps},x}_{s-})\Delta W_{\ell^{\eps}_s},
	\end{align*}
	then by Burkholder's inequality (see \cite[Lemma 2.1]{Wang2015gradient}) and  Assumption \ref{hyp:A} \hyperref[A2]{$(A2)$}, we have
	\begin{align*}
		&\mathbb{E}\big[\delta_{\Delta (\nabla_{v}X^{\ell^{\eps},x}_s)}|\nabla_{v}X^{\ell^{\eps},x}_{s-}|^{2p}-\nabla_{\Delta (\nabla_{v}X^{\ell^{\eps},x}_s)} |\nabla_{v}X^{\ell^{\eps},x}_{s-}|^{2p}\big]\\
		\lesssim&_p \mathbb{E}\big[(|\nabla_{v}X^{\ell^{\eps},x}_{s-}|^{2p-2}|\nabla_{ \nabla_v X^{\ell^{\eps},x}_{s-}}\sigma(X^{\ell^{\eps},x}_{s-})\Delta W_{\ell^{\eps}_s}|^2+|\nabla_{ \nabla_v X^{\ell^{\eps},x}_{s-}}\sigma(X^{\ell^{\eps},x}_{s-})\Delta W_{\ell^{\eps}_s}|^{2p})\big]\\
		\lesssim&_{p,c_3}\left(1+(\ell^{\eps}_s)^{p-1}\right)\Delta \ell^{\eps}_s\mathbb{E}\big[|\nabla_{v}X^{\ell^{\eps},x}_{s-}|^{2p}\big].
	\end{align*}
	Combining \eqref{fix:est:ito}, there is $C=C(p,c_3)>0$ such that
	\begin{align*}
		F^p_0(t)
		\leq&1+2pc_3\int^t_0F^p_0(s)\dif s +C\left(1+(\ell^{\eps}_t)^{p-1}\right)\int^t_0F^p_0(s-)\dif \ell_s.
	\end{align*}
	By Gronwall's inequality (see \cite[Lemma 3.2]{Friz2010Multidimensional}), noting that $|v|=1$, we have
	\begin{align}
		\label{fe:1th}
		&F^p_0(t)\leq C\exp\{C\left(\ell^\eps_t+(\ell^\eps_t)^{p}+t\right)\},
	\end{align}
	where $C=C(p,c_3)>0$.
	
	{ \bf (1.2)} We consider $n=1$. Using \eqref{fe:1th} and repeating the above processes, we similarly have
	\begin{align*}
		F^p_1(t)
		\lesssim_{p,c_3}& \int^t_0F^p_1(s)\dif s +\left(1+(\ell^{\eps}_t)^{p-1}\right)\int^t_0F^p_1(s-)\dif \ell_s\\
		&+\int^t_0F^{2p}_0(s)\dif s +\left(1+(\ell^{\eps}_t)^{p-1}\right)\int^t_0F^{2p}_0(s-)\dif \ell_s\\
		\lesssim_{p,c_3}& \int^t_0F^p_1(s)\dif s +\left(1+(\ell^{\eps}_t)^{p-1}\right)\int^t_0F^p_1(s-)\dif \ell_s
		+\exp\left\{C\left(\ell^\eps_t+(\ell^\eps_t)^{2p}+t\right)\right\},
	\end{align*}
	which implies by Gronwall's inequality that there exists $C=C(p,c_3)>0$ such that 
	\begin{align*}
		&F^p_1(t)\leq C\exp\left\{C\left(\ell^\eps_t+(\ell^\eps_t)^{2p}+t\right)\right\}.
	\end{align*}

{\bf Step2.} We prove \eqref{B5} by following steps.
	
	{\bf (2.1)} We first prove
	\begin{align}
		&\lim_{\eps\downarrow 0}\sup_{t\in[0,T]}\mathbb{E}\big|X^{\ell^\eps,x}_t- X^{\ell,x}_t|^2=0.\label{fix:infi:k1}
	\end{align}
	Denoting by $Y^\eps_t:=X^{\ell,x}_t- X^{\ell^\eps,x}_t$, by \eqref{sde:fix_xp} and  \eqref{app:fix:sde}, we have
	\begin{align*}
		Y^\eps_t=\int^t_0\delta_{Y^\eps_s}b(X^{\ell^\eps,x}_s)\dif s+\int^t_0\delta_{Y^\eps_{s-}}\sigma(X^{\ell^\eps,x}_{s-})\dif W_{\ell_s}+R^\eps_t,
	\end{align*}
	where
	$$
	R^\eps_t:=\int^t_0\sigma(X^{\ell^\eps,x}_{s-})\dif W_{\ell_s}-\int^t_0\sigma(X^{\ell^\eps,x}_{s-})\dif W_{\ell^\eps_s}.
	$$
By Assumption \ref{hyp:A} \hyperref[A2]{$(A2)$}, Jensen's inequality and  
$\mathbf{C}_r$-inequality, for $t\in[0,T]$, we have
	\begin{align*}
		\mathbb{E}[ |Y^\eps_t|^2]\leq3T\|\nabla b\|^2_{\infty}\int^t_0 \mathbb{E}[|Y^\eps_s|^2]\dif s+3\|\nabla \sigma\|^2_{\infty}\int^t_0\mathbb{E}[|Y^\eps_s|^2]\dif \ell_s+3\mathbb{E}[ |R^\eps_t |^2].
	\end{align*}
	And by Lemma \ref{le:aux:app}, we have
	$$
	\lim_{\eps\downarrow 0}\sup_{t\in[0,T]}\mathbb{E} [|R^\eps_t |^2]=0.
	$$
	Then by Gronwall's inequality (see \cite[Lemma 3.2]{Friz2010Multidimensional}), we have
	\begin{align*}
		\lim_{\eps\downarrow 0}\sup_{t\in[0,T]}\mathbb{E} [|Y^\eps_t|^2]\leq 3\e^{3T^2\|\nabla b\|^2_{\infty}+3\|\nabla \sigma\|^2_{\infty}\ell_T}\lim_{\eps\downarrow 0}\sup_{t\in[0,T]}\mathbb{E} [|R^\eps_t |^2]=0.
	\end{align*}
	{\bf (2.2)} We secondly consider to prove
	\begin{align*}
		&\lim_{\eps\downarrow 0}\sup_{t\in[0,T]}\mathbb{E}\big[\big|\nabla_{v}(X^{\ell^\eps,x}_t- X^{\ell,x}_t)\big|^2\big]=0.
	\end{align*}
	Denoting by $Z^\eps_t:=\nabla_v X^{\ell,x}_t- \nabla_vX^{\ell^\eps,x}_t$, by \eqref{fix:1g}, we have
	\begin{align}
		Z^\eps_t=\int^t_0\nabla_{Z^\eps_s}b(X^{\ell,x}_s)\dif s+\int^t_0\nabla_{Z^\eps_{s-}}\sigma(X^{\ell,x}_{s-})\dif W_{\ell_s}+\tilde{R}^\eps_t+\bar{R}^\eps_t,\label{fix:est:k1}
	\end{align}
	where
	\begin{align*}
		&\tilde{R}^\eps_t:=\int^t_0\nabla_{ \nabla_v X^{\ell^\eps,x}_{s-}}\sigma(X^{\ell^\eps,x}_{s-})\dif W_{\ell_s}-\int^t_0\nabla_{ \nabla_v X^{\ell^\eps,x}_{s-}}\sigma(X^{\ell^\eps,x}_{s-})\dif W_{\ell^\eps_s},\\
		&\qquad\quad\bar{R}^\eps_t:=\int^t_0\nabla_{ \nabla_v X^{\ell^\eps,x}_{s}}b(X^{\ell,x}_{s})-\nabla_{ \nabla_v X^{\ell^\eps,x}_{s}}b(X^{\ell^\eps,x}_{s})\dif s\\
		&\qquad\qquad+\int^t_0\nabla_{ \nabla_v X^{\ell^\eps,x}_{s-}}\sigma(X^{\ell,x}_{s-})-\nabla_{ \nabla_v X^{\ell^\eps,x}_{s-}}\sigma(X^{\ell^\eps,x}_{s-})\dif W_{\ell_s}.
	\end{align*}
	By Lemma \ref{le:aux:app}, we have
	\begin{align}\label{fix:est:k2}
		\lim_{\eps\downarrow 0}\sup_{t\in[0,T]}\mathbb{E}[ |\tilde{R}^\eps_t |^2]=0.
	\end{align}
	By $\mathbf{C}_r$-inequality and H\"older inequality, we have
	\begin{align*}
		\sup_{t\in[0,T]}\mathbb{E}[|\bar{R}^\eps_t|^2]\leq &2T\int^T_0\mathbb{E}\Big[\big|\nabla_{ \nabla_v X^{\ell^\eps,x}_{s}}b(X^{\ell,x}_{s})-\nabla_{ \nabla_v X^{\ell^\eps,x}_{s}}b(X^{\ell^\eps,x}_{s})\big|^2\Big]\dif s\\
		&+2\int^T_0\mathbb{E}\Big[\big\|\nabla_{ \nabla_v X^{\ell^\eps,x}_{s-}}\sigma(X^{\ell,x}_{s-})-\nabla_{ \nabla_v X^{\ell^\eps,x}_{s-}}\sigma(X^{\ell^\eps,x}_{s-})\big\|^2\Big]\dif \ell_s\\
		\leq &2T\int^T_0\big(\mathbb{E}\big[\big|\nabla_v X^{\ell^\eps,x}_{s}\big|^4\big]\big)^{\frac{1}{2}}\big(\mathbb{E}\big[\big\|\nabla b(X^{\ell,x}_{s})-\nabla b(X^{\ell^\eps,x}_{s})\big\|^4\big]\big)^{\frac{1}{2}}\dif s\\
		&+2\int^T_0\big(\mathbb{E}\big[\big|\nabla_v X^{\ell^\eps,x}_{s-}\big|^4\big]\big)^{\frac{1}{2}}\big(\mathbb{E}\big[\big\|\nabla \sigma(X^{\ell,x}_{s-})-\nabla \sigma(X^{\ell^\eps,x}_{s-})\big\|^4\big]\big)^{\frac{1}{2}}\dif \ell_s.
	\end{align*}
	By \eqref{fe:1th}, \eqref{fix:infi:k1} and the dominated convergence theorem, we have
	\begin{align*}
		\lim_{\eps\downarrow0}\sup_{t\in[0,T]}\mathbb{E}\big[\big|\nabla b(X^{\ell,x}_{t})-\nabla b(X^{\ell^\eps,x}_{t})\big|^4\big]=0,
	\end{align*}
	which implies
	\begin{align}\label{fix:est:k3}
		\lim_{\eps\downarrow 0}\sup_{t\in[0,T]}\mathbb{E} [|\bar{R}^\eps_t |^2]=0.
	\end{align}
	Combining \eqref{fix:est:k1}, \eqref{fix:est:k2} and \eqref{fix:est:k3}, as the proof procedure of \eqref{fix:infi:k1}, we have
	\begin{align*}
		\lim_{\eps\downarrow 0}\sup_{t\in[0,T]}\mathbb{E}[ |Z^\eps_t|^2]=0.
	\end{align*}
	
{\bf (2.3)} We thirdly consider to prove
	\begin{align*}
		&\lim_{\eps\downarrow 0}\sup_{t\in[0,T]}\mathbb{E}\big[\big|\nabla_{v_1}\nabla_{v}(X^{\ell^\eps,x}_t- X^{\ell,x}_t)\big|^2\big]=0.
	\end{align*}
Since the rest of the proof is repetitious, we leave out it and complete the proof.

{\bf Step3.} Recall $\{\ell^\eps_\cdot, \eps\in(0,1)\}$ as \eqref{app:fix:p1} and $\mathbf{v}:=(v,v_1)\in\mR^{2d}$ and $\beta_t=t\wedge\ell_{\tau_\ell}$. Our proof includes four parts as below.
	
	{\bf (3.1)} We prove for any $p\in[2,\infty)$,
	\begin{align}
		\lim_{\eps\downarrow0}\sup_{t\in[0,T]}\mathbb{E}\big[|\mathbf{L}^{v}_{\ell,t}-\delta(h)^{\ell^\eps}_t |^p\big]=0,\label{mal:dua:app}
	\end{align}
	where $\mathbf{L}^{v}_{\ell,t}$ is defined as \eqref{mal:dua:rew}.
	For any $\eps\in(0,1)$, we denote by
	\begin{align*}
		&\qquad A_t^{\ell^\eps,1}:=\sigma^{-1}(X^{\ell^\eps,x}_{t})\nabla_vX^{\ell^\eps,x}_{t},\\
		&A_t^{\ell^\eps,2}:=\big\langle\sigma^{-1}(X^{\ell^\eps,x}_{t}),\nabla_{\nabla_vX^{\ell^\eps,x}_{t}}\sigma(X^{\ell^\eps,x}_{t})\big\rangle_{\small{\mathrm{HS}}},\\
		&\quad A_t^{\ell^\eps,3}:=\sigma^{-1}(X^{\ell^\eps,x}_{t})\nabla_{\nabla_vX^{\ell^\eps,x}_{t}}\sigma(X^{\ell^\eps,x}_{t}).
	\end{align*}
	%By \eqref{mal:dua:rew} and \eqref{mal:dua},
	By \eqref{mal:dua}, recalling $\beta_t=t\wedge\ell_{\tau_\ell}$ and
	$
	\bm{\tau}^{\ell}_t=t\wedge\tau_\ell
	$ with $\tau_\ell$ defined in \eqref{de:tl},
	we rewrite $\delta(h)^{\ell^\eps}_t$ as
	\begin{align*}
		\delta(h)^{\ell^\eps}_t&=Q^\eps(t)+R^\eps(t),
	\end{align*}
	where
	\begin{align*}
		Q^\eps(t)&=\int^{\bm{\tau}^{\ell}_t}_0\big\langle A_{s-}^{\ell^\eps,1},\dif W_{\ell^\eps_s}\big\rangle-\int^{\bm{\tau}^{\ell}_t}_0 A_{s-}^{\ell^\eps,2}\dif\ell^\eps_s+\sum_{s\in(0,\bm{\tau}^{\ell}_t]}\big\langle A_{s-}^{\ell^\eps,3}\Delta W_{\ell^\eps_{s}},\Delta W_{\ell^\eps_{s}}\big\rangle,\\
		R^\eps(t)&=\int^t_{\bm{\tau}^{\ell}_t}\left\langle A_{s-}^{\ell^\eps,1},\dif W^\beta_{\ell^\eps_s}\right\rangle
		-\int^t_{\bm{\tau}^{\ell}_t} A_{s-}^{\ell^\eps,2}\dif\beta_{\ell^\eps_s}
		+\sum_{s\in(\bm{\tau}^{\ell}_t,t]}\left\langle A_{s-}^{\ell^\eps,3}\Delta W^\beta_{\ell^\eps_{s}},\Delta W_{\ell^\eps_{s}}\right\rangle\\
		&=:R_1^\eps(t)-R_2^\eps(t)+R_3^\eps(t).
	\end{align*}
	By Lemma \ref{le:aux:app} and the dominated convergence theorem, we have
	\begin{align}
		\lim_{\eps\downarrow0}\sup_{t\in[0,T]}\mathbb{E}[|\mathbf{L}^{v}_{\ell,t}-Q^\eps(t) |^2]=0.\label{mda1}
	\end{align}
	Denote by
	$$
	K:=\sup_{i\leq 3,\,\eps<1,\,t\leq T}\mathbb{E}[|A_{t}^{\ell^\eps,i}|^2].
	$$
	For $R_1^\eps(t)$, as $\beta_t=t\wedge\ell_{\tau_\ell}$, we have
	\begin{align*}
		\mE\big[|R_1^\eps(t)|^2\big]
		\leq K\sum_{s\in(\bm{\tau}^{\ell}_t,t]}\int^{\ell^\eps_s}_{\ell^\eps_{s-}}\dot{\beta}_r\dif \beta_r
		\leq K(\ell^\eps_t\wedge\ell_{\tau_\ell}-\ell^\eps_{\bm{\tau}^{\ell}_t}\wedge\ell_{\tau_\ell})\leq K(\ell_{\bm{\tau}^{\ell}_t}-\ell^\eps_{\bm{\tau}^{\ell}_t}).
	\end{align*}
	Noting that $
	\ell^\eps_t\uparrow\ell_t
	$ as $\eps\downarrow 0$, 
	then we have
	\begin{align}
		\lim_{\eps\downarrow 0}\sup_{t\in[0,T]}\mE\big[|R_1^\eps(t)|^2\big]
		\leq \lim_{\eps\downarrow 0}\sup_{t\in(\tau_\ell,T]}K(\ell_{\bm{\tau}^{\ell}_t}-\ell^\eps_{\bm{\tau}^{\ell}_t})
		\leq \lim_{\eps\downarrow 0}K(\ell_{\tau_\ell}-\ell^\eps_{\tau_\ell})=0.\label{tau-t}
	\end{align}
	Similarly, we also have
	\begin{align*}
		\lim_{\eps\downarrow 0}\sup_{t\in[0,T]}\mE\big[\left|R_2^\eps(t)\right|\big]=0.
	\end{align*}
	For $R_3^\eps(t)$, by Minkovski's inequality and H\"older's inequality, we have
	\begin{align*}
		\mathbb{E}[|R_3^\eps(t)|]\leq& \sum_{s\in(\bm{\tau}^{\ell}_t,t]}\big(\mathbb{E}[|A_{s-}^{\ell^\eps,3}\Delta W^\beta_{\ell_{s}}|^2]\big)^{\frac{1}{2}}\big(\mathbb{E}[|\Delta W_{\ell_{s}}|^{2}]\big)^{\frac{1}{2}}.
	\end{align*}
By the Burkholder-Davis-Gundy inequality, we have
	\begin{align*}
		\mathbb{E}|R_3^\eps(t)|\leq& CK^{\frac{1}{2}}\sum_{s\in(\bm{\tau}^{\ell}_t,t]}
		\left(\int^{\ell^\eps_s}_{\ell^\eps_{s-}}|\dot{\beta}_r|^2\dif r\right)^{\frac{1}{2}}(\Delta\ell_s)^{\frac{1}{2}}.
	\end{align*}
	Repeating the procedure of $R_1^\eps(t)$, it is clear that
	\begin{align*}
		\lim_{\eps\downarrow 0}\sup_{t\in[0,T]}\mE[|R_3^\eps(t)|]=0.
	\end{align*}
	Combining the above calculations, we have
	\begin{align}
		\lim_{\eps\downarrow 0}\mE\left|R^\eps(t)\right|\leq\lim_{\eps\downarrow 0}
		\mE\left[|R_1^\eps(t)|+|R_2^\eps(t)|+|R_3^\eps(t)|\right]= 0.\label{mda2}
	\end{align}
	On the other hand, by (i) of Lemma \ref{fix:infi}, it is easy to get
	$$
	\sup_{t\in[0,T]}\sup_{\eps\in(0,1)}\mathbb{E}\left[|\mathbf{L}^{v}_{\ell,t}|^p+|\delta(h)^{\ell^\eps}_t |^p\right]<\infty.
	$$
	Then by \eqref{mda1} and \eqref{mda2}, using the dominated convergence theorem, we obtain \eqref{mal:dua:app}.
	
	{\bf (3.2)} We prove 
	\begin{align}
		\lim_{\eps\downarrow0}\sup_{t\in[0,T]}\mathbb{E}\big[|\mathbf{Z}^{v,v_1}_{\ell,t}- [\nabla_{v_1},D_{h}] X^{\ell^\eps,x}_t |^2 \big]=0,\label{Z12}
	\end{align}
	where $\mathbf{Z}^{v,v_1}_{\ell,t}$ is defined as \eqref{fix:infi:lie} and $[D_{h},\nabla_{v_1}]:=D_{h}\nabla_{v_1}-\nabla_{v_1}D_{h}$. We recall
	\begin{align*}
		&\quad V^{\ell^\eps,x}_{t}:=\nabla_{v_1} \nabla_{v}X^{\ell^\eps,x}_{t},\quad M^{\ell^\eps,x}_{t}:=\nabla_{v_1}X^{\ell^\eps,x}_{t}(\nabla_{v}X^{\ell^\eps,x}_{t})^*,\\
		&\dif \nabla_{v_1}X^{\ell^\eps,x}_t=\nabla_{\nabla_{v_1}X^{\ell^\eps,x}_t}b(X^{\ell^\eps,x}_t)\dif t+\nabla_{\nabla_{v_1}X^{\ell^\eps,x}_{t-}}\sigma(X^{\ell^\eps,x}_{t-})\dif W_{\ell^\eps_t}.
	\end{align*}
By \eqref{mal:nor}, since $h_s$ is defined in \eqref{mal:der}, we know that
	\begin{align*}
		& [D_{h},\nabla_{v_1}]X^{\ell^\eps,x}_t =(D_{h}-\beta_{\ell^\eps_t}\nabla_{v})\nabla_{v_1}X^{\ell^\eps,x}_t,\quad 
		(D_{h}-\beta_{\ell^\eps_t}\nabla_{v})X^{\ell^\eps,x}_{t}=0,
	\end{align*}
	then by Lemma \ref{le:M22F}, we have
	\begin{align}
		\dif[D_{h},\nabla_{v_1}]X^{\ell^\eps,x}_t=&\int^t_0\nabla_{[D_{h},\nabla_{v_1}]X^{\ell^\eps,x}_s}
		b(X^{\ell^\eps,x}_{s})\dif s+\int^t_0\nabla_{[D_{h},\nabla_{v_1}]X^{\ell^\eps,x}_{s-}}
		\sigma(X^{\ell^\eps,x}_{s-})\dif W_{\ell^\eps_s}+Z_t^{\ell^\eps},\label{FIL}
	\end{align}
	where
	\begin{align*}
		Z_t^{\ell^\eps}:=\sum_{s\in(0,t]}V^{\ell^\eps,x}_{s}\Delta\beta_{\ell^{\eps}_s}-
		\sum_{s\in(0,t]}\nabla_{\nabla_{v_1}X^{\ell^\eps,x}_t}\sigma(X^{\ell^\eps,x}_{t-})\Delta h_{\ell^\eps_s}.
	\end{align*}
Contrasting \eqref{FIL} with \eqref{fix:infi:lie}, as the proof procedure of \eqref{fix:infi:k1}, we only need to prove
	\begin{align}\label{FEK}
		\lim_{\eps\downarrow 0}\sup_{t\in[0,T]}\mathbb{E}\big[|Z_t^{\ell^\eps}-\mathcal{Z}^{v,v_1}_{\ell,t}|^2\big]=0,
	\end{align}
where $\mathcal{Z}^{v,v_1}_{\ell,t}$ is given as
\begin{align*}
\mathcal{Z}^{v,v_1}_{\ell,t}
:=&\sum_{s\in(0,\bm{\tau}
^{\ell}_t]}\Big[\nabla_{v_1} \nabla_{v}X^{\ell,x}_{s-}-\nabla_{ \nabla_{v_1} X^{\ell,x}_{s-}}
\sigma(X^{\ell,x}_{s-})\sigma^{-1}(X^{\ell,x}_{s-})
\left(\nabla_vX^{\ell,x}_{s-}+
 \nabla_{\nabla_vX^{\ell,x}_{s-}}
 \sigma(X^{\ell,x}_{s-})\Delta W_{\ell_{s}}\right)\Big]\Delta\ell_{s}.
\end{align*}
	Since $\beta_t=t\wedge\ell_{\tau_\ell}$,
	$
	\bm{\tau}^{\ell}_t=t\wedge\tau_\ell
	$, 
	we have
	\begin{align*}
		Z_t^{\ell^\eps}-\mathcal{Z}^{v,v_1}_{\ell^\eps,t}
=\sum_{s\in(\bm{\tau}^{\ell^\eps}_t,t]}V^{\ell^\eps,x}_{s}\Delta\beta_{\ell^{\eps}_s}-
		\sum_{s\in(\bm{\tau}^{\ell^\eps}_t,t]}\nabla_{\nabla_{v_1}X^{\ell^\eps,x}_t}\sigma(X^{\ell^\eps,x}_{t-})\Delta h_{\ell^\eps_s}.
	\end{align*}
Similar to \eqref{tau-t}, we obtain
	\begin{align*}
		\lim_{\eps\downarrow0}\sup_{t\in[0,T]}\mathbb{E}\big[|Z_t^{\ell^\eps}-\mathcal{Z}_t^{\ell^\eps}|^2\big]=0.
	\end{align*}
By Lemma \ref{le:aux:app} and the dominated convergence theorem, we have
	\begin{align*}
		\lim_{\eps\downarrow0}\sup_{t\in[0,T]}\mathbb{E}\big[|\mathcal{Z}_t^{\ell^\eps}-\mathcal{Z}^{v,v_1}_{\ell,t} |^2\big]=0.
	\end{align*}
Then we  get the desired result as step {\bf(3.1)}.
\end{proof}

\section*{Acknowledgements}
X. Zhang is supported the National Natural Science Foundation of China grant (No. 12401172). X. Jin is supported by the National Natural Science Foundation of China grant (No. 12401176), and the
Fundamental Research Funds for the Central Universities grant (No. JZ2025HGTB0178).

\newpage

\end{document}